\documentclass[12pt,reqno]{amsart}

\usepackage[T1]{fontenc}
\usepackage[english]{babel}

\usepackage[letterpaper,top=2cm,bottom=2cm,left=3cm,right=3cm,marginparwidth=1.75cm]{geometry}

\usepackage{amsbsy,amssymb,amsmath,amsthm,amscd,amsfonts,latexsym}
\usepackage{graphicx}
\usepackage{booktabs}
\usepackage[colorlinks=true, allcolors=blue]{hyperref}
\usepackage{indentfirst}
\usepackage{microtype}
\newtheorem{thm}{Theorem}
\newtheorem{prop}{Proposition}
\newtheorem{lem}[prop]{Lemma}
\newtheorem{cor}[prop]{Corollary}
\newtheorem{clm}[prop]{Claim}
\newtheorem{conj}[prop]{Conjecture}

\theoremstyle{definition}
\newtheorem{df}[prop]{Definition}

\newtheorem{rmk}[prop]{Remark}
\newtheorem{ex}[prop]{Example}

\newcommand{\T}{{\mathbb{T}}}
\newcommand{\R}{{\mathbb{R}}}
\newcommand{\Z}{{\mathbb{Z}}}
\newcommand{\C}{{\mathbb{C}}}
\newcommand{\Q}{{\mathbb{Q}}}
\newcommand{\N}{{\mathbb{N}}}
\newcommand{\F}{{\mathbb{F}}}
\newcommand{\bP}{{\mathbb{P}}}

\newcommand{\fT}{{\mathfrak{t}}}
\newcommand{\fw}{{\mathsf{w}}}

\newcommand{\ham}{\operatorname{Ham} }

\newcommand{\gw}{\operatorname{GW}}
\newcommand{\rH}{{\mathcal{H}}}

\newcommand{\ah}{\mathcal{A}_H}
\newcommand{\crit}{\operatorname{Crit}}
\newcommand{\ind}{\operatorname{ind}}
\newcommand{\pss}{\operatorname{PSS}}
\newcommand{\spec}{\operatorname{Spec}}

\newcommand{\cl}{\operatorname{cuplength}}
\newcommand{\bcl}{{{\rm qcl}_\zeta}}
\counterwithin{equation}{section}
\counterwithin{thm}{section}
\counterwithin{prop}{section}

\title[Fixed points and Gromov--Witten invariants]{Degenerate symplectic fixed points and Gromov--Witten invariants}
\author{Wenmin Gong}

\address{Laboratory of Mathematics and Complex Systems (Ministry
of Education), School of Mathematical Sciences, Beijing Normal University,
	Beijing, 100875, China}

\email{wmgong@bnu.edu.cn}

\date{}
\begin{document}

\begin{abstract}
We establish a connection between Gromov--Witten invariants and the number of fixed points of Hamiltonian diffeomorphisms on closed
rational symplectic manifolds, without assuming nondegeneracy, via deformed Hamiltonian spectral invariants. We obtain
a Hamiltonian fixed point estimate involving Gromov--Witten invariants with one point insertion under an explicit minimum-area hypothesis. We propose three conjectures about Hamiltonian fixed point estimate and prove them for every generalized flag manifold $X=G/P$ equipped with a rational K\"ahler form. We extend Schwarz's quantum cuplength to deformed quantum cuplength for symplectic periods and use it to obtain estimates on monotone symplectic manifolds admitting nonzero mixed Gromov--Witten invariants. We use the quantum-factorization estimate to recover the Fano toric case of Givental's symplectic fixed-point theorem. Finally, we use a graded spectral-carrier argument to give uniform fixed-point bounds on negatively monotone manifolds.
\end{abstract}
\maketitle

\section{Introduction}\label{sec:1}

The Arnol'd conjecture~\cite{Ar0,Ar,Ar1} asserts that every Hamiltonian diffeomorphism $\varphi$ of a closed symplectic manifold $(M,\omega)$ possesses at least as many fixed points as a smooth function $f:M\to\R$ on $M$ possesses critical points. This longstanding conjecture has a rich history, detailed in the work of Hofer and Zehnder~\cite[Chapter~6]{HZ}. This paper concerns lower bounds for the total number of fixed points, allowing degeneracy. We use the cup-length convention
\[
\operatorname{cuplength}(M;R)
=1+\max\{k\mid a_1\cup\cdots\cup a_k\ne0,
\ a_i\in H^{>0}(M;R)\},
\]
where $R$ denotes any ring. The cup-length form of the homological Arnol'd conjecture predicts $\#\operatorname{Fix}(\varphi)\geq\operatorname{cuplength}(M;R)$. Recent counterexamples to the cup-length form of the homological Arnol'd conjecture by OpenAI~\cite{OAI} and Jiao~\cite{Jiao} further motivate the search for fixed-point lower bounds that remain valid in the degenerate setting.

For nondegenerate Hamiltonian diffeomorphisms, Floer theory gives homological lower bounds in a much broader setting \cite{Fl1,Fl2,Flo,FO,LT0,Ru,Pa,FW}; refinements with other coefficients are developed in~\cite{AB,BX,Re}.  Degenerate fixed points require
additional arguments.  Classical cup-length results include tori \cite{CZ}, the aspherical setting~\cite{Fl3,Ho,RO}, complex projective space~\cite{For}, and negatively monotone manifolds under a minimal Chern-number hypothesis~\cite{LO}.  Quantum cup-length and spectral invariants provide further estimates in the presence of holomorphic spheres~\cite{Flo,Sc2,GG,Go,Go1}.

Gromov--Witten invariants record constraints on pseudo-holomorphic spheres and connect them to Hamiltonian dynamics.  Their applications include capacity estimates~\cite{Lu1,Us} and fixed-point bounds obtained from quantum products~\cite{Gi,Sc2,Go1}.  Here we use mixed Gromov--Witten invariants and the big quantum product $*_\zeta$ on $H_*(M;\Lambda)$, where $\Lambda$ is a Novikov field over $\mathbb C$ and $\zeta\in\bigoplus_{i=0}^{n-1}H_{2i}(M;\Lambda_0)$ is an even bulk parameter when $\dim M=2n$.  The coefficient ring and deformation are specified in Section~\ref{subsec:bigquan}.  The minimum-area assumption in Theorem~\ref{thm:onept} is stated explicitly. For generalized flag manifolds we prove both the proposed cup-length conjecture and the class-dependent Chern-number conjecture formulated below. The flag-manifold argument uses the undeformed quantum product, Schubert positivity, and Gromov--Witten invariants with one point insertion in Schubert curve classes. The examples recover established bounds of Schwarz and Givental and provide explicit formulas for rational nonmonotone flag manifolds; the comparison with the known monotone estimates is made in Section~\ref{sec:flags}.

The mixed Gromov-Witten invariant ${\rm GW}_{0,l+k,A}^{\{1,\ldots,l\}}(a_1,\ldots,a_l,b_1,\ldots, b_k)$ from Ruan and Tian~\cite{RT}  (see also~\cite[Chapter~7]{MS}) formally counts (with appropriate signs) the tuples $(u,\vec{w})=(u,\{w_j\}_{j=1}^k)$ consisting of a $\{J_z\}$-holomorphic sphere $u:S^2\to M$ representing the class $A$ with $u(z_i)\in\alpha_i$ for $1\leq i\leq l$ and $u(w_j)\in \beta_j$ for $1\leq j\leq k$, where $\{z_i\}_{i=1}^l$ is a tuple of fixed pairwise distinct points on $S^2$ and the pairwise distinct marked points $w_j\in S^2\setminus\{z_i\}_{i=1}^l$ are allowed to vary freely for $1\leq j\leq k$, and $\alpha_i, \beta_j$ are generic representatives of the classes $a_i$ and $b_j$ respectively.  Note that when fixing $3$ pairwise distinct points on $S^2$, the mixed Gromov-Witten invariant ${\rm GW}_{0,k+3,A}^{\{1,2,3\}}(a_1,a_2,a_3,b_1,\ldots, b_k)$ is precisely the usual Gromov-Witten invariant ${\rm GW}_{0,k+3,A}(a_1,a_2,a_3,b_1,\ldots, b_k)$ as defined in~\cite[Chapter~7]{MS} for the semipositive case and~\cite{FO} for the general case.

Throughout, the phrases ``with one point insertion'' and ``with two point insertions'' refer to one or two specified insertions of the point class $[pt]$; they do not specify the total number of marked points, and additional insertions may also be point classes.

Our first theorem gives a concrete connection, for a rational symplectic manifold $(M,\omega)$, between the number of fixed points of Hamiltonian diffeomorphisms and nonzero mixed Gromov-Witten invariants with one point insertion. Here  ``rational'' means that  $\omega(\pi_2(M))=p\cdot\Z$ for some positive number $p$, called the {\it minimal symplectic period} of $(M,\omega)$. Throughout this paper, we set $H_2(M)=H_2(M;\mathbb{Z})/{\rm torsion}$. When a Gromov--Witten invariant is indexed by a class in this quotient, it denotes the sum over its integral lifts differing by torsion.  All curve-class sums use this convention. 

\begin{thm}\label{thm:onept}
Let $(M^{2n},\omega)$ be a closed rational symplectic manifold with  minimal symplectic period $p>0$ which admits a nonzero mixed Gromov-Witten invariant of the form:
\[
    {\rm GW}_{0,l+k+1,A}^{\{1,\ldots,l+1\}}(a_1,\ldots,a_l,[pt],b_1,\ldots,b_k)
\]
where $A\in H_2(M)$, $a_1,\ldots, a_l\in H_{*<2n}(M;\Q)$ with $2\leq l\in\N$, and the classes $b_1,\ldots,b_k$ are rational homology classes of even degree  (when $k=0$ there is no class $b_i$ insertion). Assume that
\begin{equation}\label{e:infarea}
\langle [\omega],A\rangle=\inf\big\{\langle [\omega],B\rangle\big|\;{\rm GW}_{0,l+i+1,B}^{\{1,\ldots,l+1\}}(a_1,\ldots,a_l,c,b_1,\ldots,b_i)\neq 0\big\}
\end{equation}
where the infimum is taken over any $c\in H_*(M;\Q)$, any $i\in\N\cup\{0\}$ and any $b_j\in H_{\rm even}(M;\Q)$ with $j\in\{1,\ldots,i\}$. Then, for each $\varphi\in\ham(M,\omega)$,
\begin{equation}\label{e:thm1}\#{\rm Fix}(\varphi)\geq\bigg\lceil\frac{pl} {\langle[\omega],A\rangle}\bigg\rceil\end{equation}
where $\lceil\cdot\rceil$ denotes the smallest integer that is greater or equal to the given number.
\end{thm}

Condition~\eqref{e:infarea} keeps $a_1,\ldots,a_l$ fixed
and ranges over all terminal classes and all numbers of additional
even insertions. Thus it controls the lowest area in the full
bulk-deformed product. A relation for the small quantum product
alone does not verify this condition. The factorization criterion
in Theorem~\ref{thm:factorization} provides a separate route to
the toric applications below.

Theorem~\ref{thm:onept} and its proof yield the following two corollaries.

\begin{cor}\label{cor:twopts}
Let $(M^{2n},\omega)$ be a closed rational symplectic manifold with  minimal symplectic period $p>0$ which admits a nonzero Gromov-Witten invariant of the form:
\[
    {\rm GW}_{0,k+3,A}(a_0,a_1,[pt],b_1,\ldots,b_k)
\]
where $A\in H_2(M)$, $a_0,a_1\in H_{*<2n}(M;\Q)$, and the classes $b_1,\ldots,b_k$ are rational homology classes of even degree. Assume in addition that $a_0\cap a_1=0$. If $\langle[\omega],A\rangle=p$, then each Hamiltonian diffeomorphism $\varphi$ of $(M,\omega)$ has at least two fixed points.
\end{cor}


A symplectic manifold is called \textit{uniruled} if there is a nonzero genus-zero Gromov--Witten invariant involving a point insertion and with a nontrivial curve class. Uniruledness is a fundamental concept in birational geometry and is invariant under symplectic birational cobordism (see~\cite{HLR}). For more results on the notion of uniruledness and related applications, we refer to~\cite{Ko,Ru,Mc,LR,Lu1,Lu2}, etc.

The \textit{symplectic effective semigroup} is
\[
K^{\rm eff}(M,\omega)=
\left\{\sum_{i=1}^r\beta_i\ \middle|\
\begin{array}{l}
r\geq0,\quad \beta_i\in H_2(M),\\
{\rm GW}_{0,3,\beta_i}\ne0\text{ for some insertions}
\end{array}\right\}.
\]

\begin{cor}\label{cor:hbar}
    Let $(M^{2n},\omega)$ be a closed rational symplectic manifold with  minimal symplectic period $p>0$ which admits a nonzero  Gromov-Witten invariant with fixed marked points:
 \[
    {\rm GW}_{0,l+1,A}^{\{1,\ldots,l+1\}}(a_1,\ldots,a_l,[pt])
\]
where $A\in H_2(M)$ and $a_1,\ldots,a_l\in H_{*<2n}(M;\Q)$ are homogeneous classes with $l\geq2$. Assume that
$a_1\cap\cdots\cap a_l=0$. Let
\[
\hbar^{\rm eff}(M,\omega)=\inf\{\langle[\omega],\beta\rangle\mid
\beta\in K^{\rm eff}(M,\omega),\ \langle[\omega],\beta\rangle>0\}.
\]
If $\langle [\omega],A\rangle=\hbar^{\rm eff}(M,\omega)$, then
\begin{equation}\label{e:thm2}
\#{\rm Fix}(\varphi)\geq \bigg\lceil \frac{2p}{\max_i\{2n-\deg(a_i)\}}\cdot\frac{\langle c_1(TM),A\rangle}{\langle[\omega],A\rangle}\bigg\rceil
\end{equation}
for all $\varphi\in\ham(M,\omega)$.
\end{cor}

For $\zeta\in \oplus_{i=0}^{n-1}H_{2i}(M;\Lambda_0)$,
we denote by $\gamma_\zeta(\varphi)$  the $\zeta$-deformed spectral norm of $\varphi\in\ham(M;\omega)$ (see~Definition~\ref{df:zetanorm}).  We have the following fixed point estimate of $\varphi$ in terms of $\gamma_\zeta(\varphi)$.
\begin{thm}\label{thm:Arnoldspec}
Let $(M,\omega)$ be  a closed rational symplectic manifold with period $p>0$. Then, for any
 $\zeta\in\oplus_{i=0}^{n-1}H_{2i}(M;\Lambda_0)$ and any $\varphi\in\ham(M;\omega)$, we have
 \[
		\#\operatorname{Fix}(\varphi)\geq  \frac{\operatorname{cuplength}(M;\C)}{\lfloor \gamma_\zeta(\varphi)/p\rfloor+1}
\]
where $\lfloor a \rfloor$ denotes the largest integer that is less than or equal to $a$.
\end{thm}

Since the Hofer norm $\|\varphi\|$ of $\varphi\in\ham(M;\omega)$ is not less than the $\zeta$-deformed spectral norm $\gamma_\zeta(\varphi)$ (cf.~(\ref{e:hofer})), we have the
 following corollary which is a straightforward generalization of~\cite[Theorem~8]{Go}.
\begin{cor}\label{thm:Arnold}
Let $(M^{2n},\omega)$ be a closed rational symplectic manifold with  minimal symplectic period $p>0$. If the Hofer norm of $\varphi$ satisfies $\|\varphi\|<p$, then the cup-length form of the homological Arnol'd conjecture
holds for $\varphi$ over $\C$:
\[
\#\operatorname{Fix}(\varphi)\geq\operatorname{cuplength}(M;\C).
\]
\end{cor}

\begin{rmk}
Prior to Corollary~\ref{thm:Arnold},
 Schwarz proved the cup-length estimate over $\Z_2$ for $\|\varphi\|<p$;  see~\cite[Theorem~1.1]{Sc2}.
\end{rmk}

Through introducing the \textit{$\zeta$-deformed quantum cuplength} $\bcl$ for symplectic periods (cf. Definition~\ref{df:bcl}), our approach developed in this paper gives the following extension of Schwarz's estimate~\cite[Theorem~1.2]{Sc2} for  symplectic fixed points to all closed rational symplectic manifolds.

\begin{thm}\label{thm:bcl}
    Let $(M^{2n},\omega)$ be a closed rational symplectic manifold with period $p>0$. Then, for each $\varphi\in\ham(M,\omega)$, we have the estimate
    \[
        \#{\rm Fix}(\varphi)\geq
    \sup\limits_{\zeta\in\oplus_{i=0}^{n-1}H_{2i}(M;\Lambda_0)}\sup\limits_{g\in\Gamma_\omega\cap[0,\infty)}\bigg\lceil \frac{p\cdot\bcl(g)}{p+g+\|\varphi\|} \bigg\rceil.
    \]

\end{thm}

Schwarz's quantum cup-length method gives fixed-point estimates on closed monotone symplectic manifolds~\cite{Sc2}. The next theorem expresses an asymptotic version of this method in terms of mixed Gromov--Witten invariants; its undeformed specialization recovers~\cite[Corollary~1.3]{Sc2}.

\begin{thm}\label{thm:schtype}
Let $(M^{2n},\omega)$ be a closed monotone symplectic manifold with minimal Chern number $N\geq 1$, and let $a\in H_{*<2n}(M;\Q)$ be homogeneous. Suppose that
for each $2n< l\in\N$, there exists a nonzero mixed Gromov-Witten invariant of the form:
\begin{equation}\label{e:mixgw}
    {\rm GW}_{0,l+k+1,A_l}^{\{1,\ldots,l+1\}}(\underbrace{a,\ldots,a}_l,b_0^l,\ldots,b_k^l)
\end{equation}
for some integer $k=k(l)\in\N\cup\{0\}$, where $A_l\in H_2(M;\mathbb{Z})/{\rm torsion}$, and $b_0^l,\ldots,b_k^l$ are rational homology classes of even degree. Then, for any $\varphi\in\ham(M,\omega)$,
\[
        \#{\rm Fix}(\varphi)\geq \frac{2N}{2n-\deg(a)}
\]
provided that the function $l\mapsto k(l)$ is bounded.

\end{thm}

For $\zeta\in\bigoplus_{i=0}^{n-1}H_{2i}(M;\Lambda_0)$, an ordinary class $a\in H_{*<2n}(M;\C)$ is called \textit{$\zeta$-nonnilpotent} if $a^{*_{\zeta}l}\ne0$ for every positive integer $l$. Expanding a nonzero power detects a nonzero mixed invariant, with complex coefficients if necessary. For
general bulk parameters this observation does not give the uniform bound on the number of free insertions required in Theorem~\ref{thm:schtype}. In the undeformed case there are no such insertions, and the quantum-cuplength argument gives the following classical consequence, including for complex classes.
\begin{cor}[{\cite[Corollary~1.3]{Sc2}}]\label{cor:nonnilp}
    Let $(M,\omega)$ be a monotone symplectic manifold with minimal Chern number $N$ determined by $c_1(\pi_2(M))=N\cdot \Z$. Suppose that $a\in H_{k}(M;\C)$ ($0\leq k<2n$) is a $0$-nonnilpotent element in $QH_*(M,\Lambda)_0$. Then, for all $\varphi\in\ham(M,\omega)$,
\[
\#{\rm Fix}(\varphi)\geq \frac{2N}{2n-k}.
\]
\end{cor}
\begin{proof}
Write $\delta=2n-k>0$ and $\omega|_{\pi_2(M)}=\lambda c_1(TM)|_{\pi_2(M)}$, so $p=\lambda N$. For each sufficiently large $l$, choose a nonzero coefficient of $a^{*l}$ at an area $g_l$. The dimension formula for a contributing homogeneous output of degree $j_l\in\{0,\ldots,2n\}$ gives
\[
 \frac{2g_l}{\lambda}=l\delta-2n+j_l.
\]
This follows by homogeneous expansion of the complex insertions. Consequently, we obtain $g_l/l\to\lambda\delta/2$ and $\operatorname{qcl}_0(g_l)\ge l+1$. Applying Theorem~\ref{thm:bcl} and taking $l\to\infty$ yields $\#\operatorname{Fix}(\varphi)\ge 2N/\delta$.
\end{proof}

Using Corollary~\ref{cor:nonnilp}, Schwarz~\cite[Sect.2.5]{Sc2} proved the following:
\begin{itemize}
\item[(i)] Let $M=\bP^n\times W$ be the product symplectic manifold where $\bP^n$ is equipped with the standard symplectic form, and $W$ is symplectically aspherical. Then, each Hamiltonian diffeomorphism of $M$ has at least $n+1$ fixed points.
\item[(ii)] Let ${\rm Gr}(k,n)$ be the Grassmannian  of $k$-planes in $\C^n$ equipped with the symplectic form $\omega$ satisfying $c_1({\rm Gr}(k,n))=n[\omega]$. Then, each $\varphi\in \ham({\rm Gr}(k,n),\omega)$ has at least $n$ fixed points. In particular, this recovers Fortune's theorem~\cite{For}.
\item[(iii)] Let ${\rm F}_{n+1}$ be the complete flag manifold  of
sequences $V^1\subset\ldots\subset V^n$ of subspaces of $\C^{n+1}$ with ${\rm dim}_\C V^i=i$ equipped with the symplectic form $\omega$ satisfying $c_1({\rm F}_{n+1})=2[\omega]$. Then, for any $\varphi\in \ham({\rm F}_{n+1},\omega)$, $\#{\rm Fix}(\varphi)\geq 2$.

\end{itemize}

Inspired by Schwarz's quantum cuplength estimate, a different way for estimating symplectic fixed points was obtained by the author via introducing the so-called \textit{fundamental quantum factorization};  see~\cite[Definition~23]{Go}.
The monotone, undeformed version appears in~\cite[Theorem~9]{Go}.
Using the big quantum homology ring, we extend the notion and the resulting estimate to the rational setting as follows.

\begin{df}\label{def:qf}
    Let $(M^{2n},\omega)$ be a closed rational symplectic manifold with  minimal symplectic period $p>0$. We say that $M$ has a \textit{principal fundamental quantum factorization} associated to $\zeta\in \oplus_{i=0}^{n-1}H_{2i}(M;\Lambda_0)$ (denoted by PFQF for short)  \textit{of length $l\in\N$ with order $g\in\N$} if there exist $u_1,\ldots,u_l\in H_{*<2n}(M;\C)$ such that
    $$ u_1*_\zeta\cdots *_\zeta u_l=k T^{gp}[M]+\beta$$
    with $I_\nu(\beta)\leq-gp$, where $0\neq k\in\C$, $\beta\in \oplus_{i=0}^{2n-1}H_i(M;\Lambda)$ and $I_\nu:H_*(M;\Lambda)\to\Gamma_\omega\cup\{-\infty\}$ is the valuation map as given in~(\ref{e:val}).
\end{df}

\begin{thm}\label{thm:factorization}
		Let $(M,\omega)$ be  a closed rational symplectic manifold admitting a PFQF of length $l$ with order $g$. Then each $\varphi\in\ham(M,\omega)$ has at least $\lceil \frac{l}{g}\rceil$ fixed points.
\end{thm}

Although the minimum-area hypothesis in Theorem~\ref{thm:onept} is essential to its proof, we formulate the following conjectures without that hypothesis. All the invariants in the next two conjectures are ordinary primary genus-zero Gromov--Witten invariants.

\begin{conj}\label{thm:mainthm}
Let $(M^{2n},\omega)$ be a closed connected rational symplectic manifold,
where $n\geq1$ and $\omega(\pi_2(M))=p\mathbb Z$, $p>0$.
Suppose that
\begin{equation}\label{e:blowup}
 {\rm GW}_{0,k+3,A}
 ([pt],a_0,[pt],a_1,\ldots,a_k)\ne0
\end{equation}
for some $k\geq0$, some class $A\in H_2(M)$ with $\omega(A)>0$,
and even rational homology classes $a_0,\ldots,a_k$.
Define
\[
 \Theta(M,\omega)=\min\left\{\omega(B)\ \middle|\
 \begin{array}{l}
 {\rm GW}_{0,j+3,B}([pt],b_0,[pt],b_1,\ldots,b_j)\ne0,\\
 j\geq0,\quad b_0,\ldots,b_j\in H_{\rm even}(M;\mathbb Q)
 \end{array}\right\}.
\]
Then every $\varphi\in\operatorname{Ham}(M,\omega)$ satisfies
\[
 \#\operatorname{Fix}(\varphi)\geq
 \left\lceil\frac{p\,\operatorname{cuplength}(M;\mathbb C)}
 {\Theta(M,\omega)}\right\rceil.
\]
In particular, if $\Theta(M,\omega)=p$, then $\#\operatorname{Fix}(\varphi)\geq\operatorname{cuplength}(M;\mathbb C)$.
\end{conj}

We also propose a class-dependent Chern-number estimate under the weaker geometric requirement of one point insertion. Its relation to Conjecture~\ref{thm:mainthm} is established for flag manifolds below; estimates for an arbitrary fixed class and for a class realizing $\Theta(M,\omega)$ must be distinguished.

\begin{conj}\label{conj:two-point-chern}
Let $(M^{2n},\omega)$ be a closed connected rational symplectic manifold with minimal symplectic period $p>0$. Suppose that a spherical class $A\in H_2(M)$ with $\omega(A)>0$ satisfies
\[
 {\rm GW}_{0,m+1,A}([pt],b_1,\ldots,b_m)\ne0
\]
for some integer $m\geq2$ and rational homology classes
$b_1,\ldots,b_m$. Then every
$\varphi\in\operatorname{Ham}(M,\omega)$ satisfies
\[
 \#\operatorname{Fix}(\varphi)\geq
 \left\lceil\frac{p\,\langle c_1(TM),A\rangle}{\langle[\omega],A\rangle}\right\rceil.
\]
\end{conj}

\begin{conj}\label{conj:monotone}
Let $(M,\omega)$ be a closed monotone symplectic manifold with minimal Chern number $N\geq 1$. Then every
$\varphi\in\operatorname{Ham}(M,\omega)$ has at least $N$ fixed points. 
\end{conj}

For generalized flag manifolds, Gromov-Witten invariants with one point insertion in the Schubert curve classes give the following uniform estimate. Since every effective class is a nonnegative integral combination of these classes, the theorem proves Conjectures~\ref{conj:two-point-chern} and \ref{conj:monotone}. The realization of $\Theta(X,\omega)$ by an invariant with two point insertions also proves Conjecture~\ref{thm:mainthm}.

\begin{thm}\label{thm:flags}
Let $X=G/P$ have positive complex dimension $D$, where $G$ is a connected complex semisimple group and $P$ is parabolic. Equip $X$ with a rational K\"ahler form $\omega$, with $\omega(\pi_2(X))=p\mathbb Z$, $p>0$. Let $C_1,\ldots,C_\rho$ be the primitive Schubert curve classes, and put $c_i=\langle c_1(TX),C_i\rangle$. Then every $\varphi\in\operatorname{Ham}(X,\omega)$ satisfies
\begin{equation}\label{eq:flag-uniform-bound}
 \#\operatorname{Fix}(\varphi)\geq
 \left\lceil p\max_{1\leq i\leq\rho}
                  \frac{c_i}{\omega(C_i)}\right\rceil.
\end{equation}
Consequently, for every nonzero effective integral class $A\in H_2(X;\mathbb Z)$,
\begin{equation}\label{eq:flag-effective-bound}
 \#\operatorname{Fix}(\varphi)\geq
 \left\lceil\frac{p\,\langle c_1(TX),A\rangle}{\omega(A)}\right\rceil.
\end{equation}
In particular, Conjectures~\ref{thm:mainthm} and~\ref{conj:two-point-chern} hold for $(X,\omega)$, and
\[
 \#\operatorname{Fix}(\varphi)\geq
 \left\lceil\frac{p(D+1)}{\Theta(X,\omega)}\right\rceil
 =\left\lceil\frac{p\,\operatorname{cuplength}(X;\mathbb C)}
 {\Theta(X,\omega)}\right\rceil.
\]
If $\omega$ is monotone with minimal Chern number $N$, then Conjectures~\ref{conj:monotone} holds for $(X,\omega)$. 
\end{thm}

Nonzero Gromov--Witten invariants with two point insertions are related to the concept of \textit{rational connectedness} in algebraic geometry. Voisin~\cite{Vo} and Tian~\cite{Ti1,Ti2} studied this question for rationally connected threefolds and fourfolds. A conjecture asserts that every rationally connected projective manifold carries a nonzero genus-zero Gromov--Witten invariant with two point insertions; see, for example,~\cite[Conjecture~3.4]{Hu1}. Such nonvanishing would provide further settings in which to investigate Conjecture~\ref{thm:mainthm}.

The strict spectral inequalities also yield estimates in the negatively monotone case, where the quantum-factorization hypothesis of Theorem~\ref{thm:factorization} cannot hold. The following result combines these inequalities with the grading and support of local Floer homology. Here and below, dimensions and cohomological degrees are real unless otherwise specified.

\begin{thm}\label{thm:negative-monotone}
Let $(M^{2n},\omega)$ be a closed connected symplectic manifold such that
\[
 [\omega]|_{\pi_2(M)}=-\kappa\,c_1(TM)|_{\pi_2(M)},\qquad \kappa>0,
\]
and suppose that $c_1(TM)(\pi_2(M))=N\mathbb Z$ with $N\geq1$. Put $L=\operatorname{cuplength}(M;\mathbb C)$. Then every $\varphi\in\operatorname{Ham}(M,\omega)$
satisfies
\begin{equation}\label{eq:negative-monotone-bound}
 \#\operatorname{Fix}(\varphi)\geq
 \max\left\{
 \left\lceil\frac{L}{\left\lceil 2n/(2N+1)\right\rceil}\right\rceil,
 \left\lceil\frac{n+1}{\left\lceil n/(N+1)\right\rceil}\right\rceil
 \right\}.
\end{equation}
More generally, if homogeneous classes $\eta_1,\ldots,\eta_k\in
H^{>0}(M;\mathbb C)$ satisfy $\eta_1\cup\cdots\cup\eta_k\ne0$ and
$\deg\eta_i\geq d>0$, then
\begin{equation}\label{eq:negative-degree-bound}
 \#\operatorname{Fix}(\varphi)\geq
 \left\lceil\frac{k+1}{\left\lceil 2n/(2N+d)\right\rceil}\right\rceil.
\end{equation}
\end{thm}

For $N\geq n$, the first term of~\eqref{eq:negative-monotone-bound} is $L$, recovering the complex-coefficient cup-length bound in the range considered by L\^{e}--Ono~\cite{LO}; see also~\cite[Section~6.2.3]{GG}. For $1\leq N<n$, the theorem gives a uniform estimate without a spectral-norm restriction. Its proof is given in Section~\ref{sec:negative-monotone}.

\subsection{Applications and examples}\label{sect:app}

The numerical estimates in this subsection are already accessible by Schwarz's quantum-cuplength method for quadrics and Givental's theorem for toric manifolds.  We include them to illustrate the spectral and quantum-factorization estimates, and to make their hypotheses and normalizations explicit.  Generalized flag manifolds are treated in Section~\ref{sec:flags}.

\subsubsection{Smooth quadrics}
Let $Q\subset\bP^{n+1}$ be a smooth complex quadric of complex dimension $n\geq2$, with the standard symplectic form normalized so that a line
has area one. For $n=2$, one may identify $Q$ with $S^2\times S^2$.  Write $h\in H_{2n-2}(Q;\Z)$ for the hyperplane class. Then $Q$ is monotone, its minimal Chern number is $N=n$, and its minimal symplectic period is $p=1$.  For $n\geq 2$, Beauville's computation \cite[formulas (1.7), (2.2) and the main theorem]{Bea} gives
\[
h^{*n}=2[pt]+2T[Q],\qquad h^{*(n+1)}=4Th.
\]
Now we give two independent proofs of the fixed-point estimate $\#\operatorname{Fix}(\varphi)\geq n$ for any Hamiltonian diffeomorphism $\varphi$ on $Q^n$.

\paragraph{\bf First proof: Schwarz's nonnilpotent-element estimate.} 
The relation $h^{*(n+1)}=4Th$ implies $h^{*(kn+1)}=4^kT^kh\ne0$ for every $k\geq0$.  Hence $h$ is nonnilpotent.  Corollary~\ref{cor:nonnilp}, with $\deg h=2n-2$ and $N=n$, yields the known estimate
\[
\#\operatorname{Fix}(\varphi)\geq n,
\qquad \varphi\in\operatorname{Ham}(Q,\omega).
\]

\paragraph{\bf Second proof: spectral separation.}
Write $\varphi=\varphi_H^1$ and $\rho=\rho_0$, and assume the fixed-point set is finite.  Put $a_j=h^{*j}$ for $1\leq j\leq n+1$. These classes are nonzero, so Theorem~\ref{thm:LjS} gives
\[
\rho(a_{n+1},H)<\rho(a_n,H)<\cdots<\rho(a_1,H).
\]
The quantum-shift and projective-invariance properties in Proposition~\ref{prop:property} give $\rho(a_{n+1},H)=\rho(h,H)-1$. Thus the $n$ values $\rho(a_1,H),\ldots,\rho(a_n,H)$ lie in the interval $(\rho(h,H)-1,\rho(h,H)]$ and are distinct modulo $\Gamma_\omega=\Z$. By spectrality they are carried by distinct contractible one-periodic orbits, since recapping one orbit changes its action by an integer. Their underlying fixed points are distinct, proving the same bound.

\subsubsection{Fano toric manifolds}
Let $(M^{2m}_\tau,\omega_\tau)$ be a compact toric symplectic manifold
with an irredundant Delzant presentation, as in
Definition~\ref{df:toric}.  Its toric prime divisors
$X_1,\ldots,X_n$ are all nonempty.  For a tuple
$d=(d_1,\ldots,d_n)$ of nonnegative integers, set
\begin{equation}\label{e:Xd}
[X]^{*d}=[X_1]^{*d_1}*\cdots*[X_n]^{*d_n},
\end{equation}
where $*$ is the small quantum product and the empty product is $[M]$.
The lattice cone $\mathcal D^{\rm eff}(\tau)$ and the corresponding
class $A_{\xi_d}$ are defined in Appendix~\ref{Sec:toric}.
If $M$ is Fano and $d\in\mathcal D^{\rm eff}(\tau)$ is nonzero with
all $d_i\geq0$, Theorem~\ref{thm:fano} gives
\begin{equation}\label{e:fundamental}
[X]^{*d}=T^{\omega(A_{\xi_d})}[M].
\end{equation}
Moreover,
$\langle c_1(TM),A_{\xi_d}\rangle=\sum_i d_i$.
This is a principal fundamental quantum factorization and gives the
following recovery of Givental's estimate.

\begin{cor}[Fano case of {\cite[Theorem~1.1]{Gi}}]\label{cor:Nmonotoric}
Let $(M,\omega)$ be a Fano toric manifold with minimal symplectic
period $p>0$.  Let $\mathcal E$ consist of the classes $A_{\xi_d}$
corresponding to nonzero $d\in\mathcal D^{\rm eff}(\tau)$ with
$d_i\geq0$ for every $i$.  Then
\[
\#\operatorname{Fix}(\varphi)\geq
\sup_{A\in\mathcal E}
\left\lceil\frac{p\langle c_1(TM),A\rangle}{\omega(A)}\right\rceil,
\qquad \varphi\in\operatorname{Ham}(M,\omega).
\]
\end{cor}
\begin{proof}
For $A=A_{\xi_d}\in\mathcal E$, identity~\eqref{e:fundamental} has length $\ell=\sum_i d_i=\langle c_1(TM),A\rangle$ and order $g=\omega(A)/p\in\N$.  Theorem~\ref{thm:factorization} and the undeformed quantum relation~\eqref{e:fundamental} give $\#\operatorname{Fix}(\varphi)\geq\lceil\ell/g\rceil$. Taking the supremum proves the assertion.
\end{proof}

Givental normalized the symplectic form to have primitive integral periods.  Rescaling to that normalization explains the factor $p$ above.  His theorem also applies to nonmonotone toric forms \cite{Gi}; the present corollary recovers its Fano case. 
\begin{cor}\label{thm:monotoric}
Let $(M,\omega)$ be a closed monotone toric manifold with minimal
Chern number $N\geq1$.  Every
$\varphi\in\operatorname{Ham}(M,\omega)$ has at least $N$ fixed points.
\end{cor}
\begin{proof}
Write $\dim_{\mathbb C}M=m$ and $[\omega]=\lambda c_1(TM)$, with $\lambda>0$.  The toric variety is Fano Definition~\ref{df:fano}.  Choose an integral ample divisor $L$ and an integer $r>0$ such that $rL$ is very ample.  For $m>1$, we intersect $m-1$ general members of $|rL|$ to obtain a curve class $A$; for $m=1$, we take $A=[M]$.  Every toric prime divisor satisfies $\operatorname{PD}[X_i](A)>0$, so $A\in\mathcal E$. Since $p=\lambda N$, Corollary~\ref{cor:Nmonotoric} gives $\#\operatorname{Fix}(\varphi)\geq N$.
\end{proof}
This corollary is a uniform specialization of the preceding known
toric bound.  When $N=1$, it provides no numerical information beyond
the existence of a fixed point.

\subsubsection{Projective space blown up at one point}
Normalize the Fubini--Study form on $\bP^n$, $n\geq2$, so that a line
has area one.  Take the standard torus-equivariant blow-up at a torus
fixed point, with ball capacity
$c=\pi\lambda^2=(n-1)/(n+1)<1$.  This is the usual corner truncation
of the moment simplex, and the resulting symplectic manifold is toric.
Let $H$ be the pullback of the hyperplane class and $e$ the Poincar\'e
dual of the exceptional divisor.  The standard blow-up formulas
\cite[Section~7.1]{MS0} give
\begin{equation}\label{e:omegac1}
[\widetilde\omega]=H-ce,\qquad
c_1(T\widetilde{\bP}^n)=(n+1)H-(n-1)e
                         =(n+1)[\widetilde\omega].
\end{equation}
A line avoiding the blown-up point and a line in the exceptional
divisor generate $H_2(\widetilde{\bP}^n;\Z)$ and have Chern numbers
$n+1$ and $n-1$, respectively.  Hence the minimal Chern number is
$N=\gcd(n+1,n-1)=\gcd(n+1,2)$.  Corollary~\ref{thm:monotoric} gives
\[
\#\operatorname{Fix}(\varphi)\geq\gcd(n+1,2).
\]

\subsubsection{Projective bundles over projective space}
Let $V=\bigoplus_{i=1}^r\mathcal O_{\bP^n}(m_i)$, where $n\geq1$,
$r\geq2$, all $m_i\geq1$, at least one $m_i=1$, and
$m:=\sum_i m_i<(2n+2+r)/2$.  Use the quotient convention for
$\pi:\bP(V)\to\bP^n$, and write
$h=\pi^*c_1(\mathcal O_{\bP^n}(1))$ and
$\xi=c_1(\mathcal O_{\bP(V)}(1))$.
The relative Euler sequence~\cite{OS}
\[
0\longrightarrow\mathcal O\longrightarrow
\pi^*V^*\otimes\mathcal O_{\bP(V)}(1)
\longrightarrow T_{\bP(V)/\bP^n}\longrightarrow0
\]
and the tangent sequence give
\begin{equation}\label{e:chern}
c_1(\bP(V))=r\xi+(n+1-m)h.
\end{equation}
The split projective bundle is toric.  Let $B$ be a fibre line and
$S$ the section over a base line induced by a quotient
$V\to\mathcal O_{\bP^n}(1)$.  This section splits the homotopy
sequence, so $B,S$ generate $\pi_2(\bP(V))\cong\Z^2$.
The divisors $h$ and $\xi-h$ are nef, since $V(-1)$ is globally
generated, and are dual to $S$ and $B$, respectively.  Hence $B,S$
generate the cone of curves.  Their Chern numbers are $r$ and
$n+1+r-m>0$, so $\bP(V)$ is Fano.
Choose a K\"ahler form $\omega$ representing $c_1(\bP(V))$.
Consequently,
\[
p=N=\gcd(r,n+1+r-m)=\gcd(r,|n+1-m|).
\]
Corollary~\ref{thm:monotoric} already gives
\[
\#\operatorname{Fix}(\varphi)\geq\gcd(r,|n+1-m|).
\]
For comparison, the explicit computation of Qin and Ruan
\cite[Proposition~5.19]{QR}, under precisely the numerical assumptions
above, gives
\begin{equation}\label{e:proj}
\prod_{i=1}^r(\bar\xi-m_i\bar h)=T^r[\bP(V)],
\end{equation}
where bars denote Poincar\'e duals and the product is the small quantum
product.  This has PFQF length $r$ and order $r/N$, so
Theorem~\ref{thm:factorization} gives exactly the same estimate.
This calculation exhibits the factorization underlying the toric bound.

\subsection{Organization of the paper}
Section~\ref{Sec:pre} reviews the coefficient rings, big quantum homology,
deformed Floer theory, and the PSS isomorphism.
Section~\ref{Sec:spec} proves the deformed Lusternik--Schnirelmann
inequality.  Section~\ref{sec:gw} relates mixed Gromov--Witten invariants
to nonvanishing big quantum products, and Section~\ref{sec:mainpf} proves
Theorems~\ref{thm:onept}--\ref{thm:factorization}.
Section~\ref{sec:flags} proves Theorem~\ref{thm:flags} and treats unitary
coadjoint orbits and monotone partial flag manifolds.
Section~\ref{sec:negative-monotone} proves Theorem~\ref{thm:negative-monotone}
and compares the negatively monotone estimates with the spectral-norm
and quantum-factorization bounds.
Appendix~\ref{sec:coadjoint} gives the root-theoretic computation for
compact simple groups; Appendix~\ref{Sec:toric} records the toric
conventions used in Section~\ref{sec:1}.

\section{Preliminaries}\label{Sec:pre}

\subsection{Notations and conventions}
Throughout these preliminaries, $(M^{2n},\omega)$ is closed and connected,
with $n\geq1$. All virtual counts below are signed rational counts,
using coherent orientations; the notation $|\mathsf s^{-1}(0)|$
denotes this signed count, not an unsigned cardinality.
The set of spherical symplectic periods of $(M,\omega)$ is denoted by
\[
\Gamma_\omega=\{\langle [\omega],A\rangle\;|\;A\in\pi_2(M)\}.
\]
We say that a symplectic manifold $(M,\omega)$ is \textit{rational} if there is a number $p>0$ such that
$\Gamma_\omega=p\cdot \Z$.

\begin{df}\label{df:Nov}
\begin{enumerate}
    \item[(1)] The Novikov ring associated to $\Gamma_\omega$ is the ring of formal Laurent series given by
\[
\Lambda=\bigg\{\sum_{i=1}^{\infty}a_iT^{\lambda_i}\;\big|\;a_i\in\C, \;\lambda_i\in\Gamma_\omega,  \;\forall C>0, \;\#\{i\in \N|a_i\neq 0, \lambda_i<C\}<\infty \bigg\}
\]
where $T$ is a formal variable.
    \item[(2)] The valuation map on $\Lambda$ is the map
    \[
    \nu:\Lambda\to \Gamma_\omega\cup\{-\infty\},\quad \nu\big(\sum_{i=1}^{\infty}a_iT^{\lambda_i}\big)=\max\{-\lambda_i\;|\;a_i\neq 0\},\qquad \nu(0)=-\infty.
    \]
\end{enumerate}
\end{df}

Denote
\[
\Lambda_0:=\bigg\{\sum_ia_iT^{g_i}\in\Lambda\;\big|\;g_i\geq 0\;\forall i\bigg\}.
\]
Note that $\Lambda_0=\{0\}\cup\nu^{-1}((-\infty,0])$ and $\Lambda$ is the field of fractions of the ring $\Lambda_0$.
For $a\in\Lambda_0$, write $a=a_0+a_+$, where $a_0\in\C$ is the coefficient of $T^0$ and $a_+$ has strictly positive exponents. Define
\[
\exp(a)=e^{a_0}\sum_{k=0}^{\infty}\frac{a_+^k}{k!}.
\]
The series in $a_+$ converges in the Novikov topology, while $e^{a_0}$ is the usual complex exponential. In particular, $\exp(0)=1$.

We place the Novikov coefficients in even parity and retain only the
natural $\Z_2$-gradings on bulk-deformed Floer homology and big quantum
homology. 

\begin{df} Let $V$ be a vector space over $\Lambda$, and $\ell:V\to \R\cup\{-\infty\}$ a function. The pair $(V,\ell)$ is called a non-Archimedean normed vector space if
\begin{itemize}
    \item[(i)] $\ell(0)=-\infty$, and $\ell(c)\in\R$ for all nonzero $c\in V$;
    \item[(ii)] $\ell(\alpha\cdot c)=\ell(c)+\nu(\alpha)$ for all $\alpha\in \Lambda$ and $c\in V$;
    \item[(iii)] $\ell(c_1+c_2)\leq\max\{\ell(c_1),\ell(c_2)\}$ for all $c_1,c_2\in V$.
\end{itemize}
\end{df}

\begin{rmk}\label{rmk:nonArchi}
There is a symbolic difference between our property (ii) in the above definition and the property (F2) in \cite[Definition~2.2]{UZ} because of the symbolic difference for the valuation map $\nu$. It can be shown that the equality in the (iii) holds if $\ell(c_1)\neq\ell(c_2)$; see~\cite[Proposition~2.1]{EP} or~\cite{UZ}.
\end{rmk}

\subsection{Big quantum homology ring}\label{subsec:bigquan}

In this subsection we will review the \textit{big quantum homology} which is closely related to the deformed spectral invariants.

Given $A\in H_2(M)$, for homogeneous classes $a_i\in H_*(M;\Z)/{\rm torsion}, i=1\ldots,k$ with $k\geq 3$ which satisfy
\begin{equation}\label{e:Dformula}
\sum_{i=1}^k(2n-{\rm deg}(a_i))=2n+2\langle c_1(TM),A\rangle + 2k-6,
\end{equation}
we denote ${\rm GW}_{0,k,A}(a_1,\ldots,a_k)$ the Gromov-Witten invariant by counting (in the appropriate virtual sense) the number of  $J$-holomorphic spheres in the class $A$ passing through the submanifolds $N_i$, where $J$ is a generic almost complex structure compatible with $\omega$ and the $N_i$ are generic cycles representing the classes $a_i$. See Ruan and Tian~\cite{RT}, Fukaya and Ono~\cite{FO} and Ruan~\cite{Ru}. We put $\gw_{0,k,A}(a_1,\ldots,a_k)=0$ unless (\ref{e:Dformula}) holds, and define
\[
    {\rm GW}_k(a_1,\ldots,a_k)=\sum_{A\in H_2(M)} {\rm GW}_{0,k,A}(a_1,\ldots,a_k)T^{\langle [\omega],A\rangle}.
\]
Then, by the Gromov compactness theorem, one can extend the above definition linearly over $\Lambda$ to a $\Lambda$-module homomorphism
\[
 {\rm GW}_k:H_*(M;\Lambda)^{\otimes k}\longrightarrow\Lambda.
\]

For each $\zeta\in   \oplus_{i=0}^{n-1}H_{2i}(M;\Lambda_0)$, we define the $\zeta$-deformed
quantum multiplication on $H_*(M;\Lambda)$ as follows: for any $a,b\in H_*(M;\Lambda)$, the product
\[
*_\zeta: H_*(M;\Lambda)\otimes H_*(M;\Lambda)\to H_*(M;\Lambda)
\]
is given by the formula
\[\langle a*_\zeta b,c\rangle_{\rm PD_M} =\sum_{k=0}^\infty\frac{1}{k!}{\rm GW}_{k+3}(a,b,c,\underbrace{\zeta,\ldots,\zeta}_k)\]
for any $c\in H_*(M;\Lambda)$
where $\langle\cdot,\cdot\rangle_{\rm PD_M}$ denotes the Poincar\'e pairing. Alternatively, one may express this product in terms of a homogeneous basis as follows. Let $\{c_j\}_{j=1}^K$ be a homogeneous basis for $H_*(M;\Q)$ with dual basis $\{c^j\}$, namely, $c^j \circ c_i=\delta_i^j$ where $\circ$ denotes the Poincar\'{e} intersection pairing. We use this left-dual convention throughout; it fixes the signs when odd-degree classes occur. Then the product $*_\zeta$ can be obtained by extending linearly from the following formula
\[
a*_\zeta b=\sum_{A\in H_2(M)}\sum_{k=0}^\infty\frac{1}{k!}\sum_{j=1}^K{\rm GW}_{0,k+3,A} (a,b,c_j,\underbrace{\zeta,\ldots,\zeta}_k)T^{\langle [\omega],A\rangle}c^j.
\]
Write $\zeta=D+B$, where $D\in H_{2n-2}(M;\Lambda_0)$ and
$B\in\bigoplus_{i=0}^{n-2}H_{2i}(M;\Lambda_0)$.
Insertions of $D$ in these series are summed using the divisor axiom,
with weight $\exp(\langle\operatorname{PD}(D),A\rangle)$ and the
exponential convention above. In particular, the constant part of a divisor coefficient
is summed by the ordinary complex exponential. The remaining bulk
insertions have real codimension at least four. For any fixed energy
bound, only finitely many curve classes contribute; the dimension
formula then bounds the number of these remaining insertions.
Together with the Novikov finiteness condition on their coefficients,
this makes the product well-defined for every permitted $\zeta$.
It is well-known that the product $*_\zeta$ is associative and supercommutative with a parity decomposition $H_*(M;\Lambda)=H_{\rm even}(M;\Lambda)\oplus H_{\rm odd}(M;\Lambda)$; see~\cite{KM} or~\cite[Chapter~11]{MS}. So we obtain a  ring
\[QH_*(M,\omega)_\zeta=\big(H_*(M;\Lambda),*_\zeta\big)\]
which obviously has the fundamental class $[M]$ as its unity, and which is called the \textit{$\zeta$-deformed homology ring} of $(M,\omega)$.

Following Ruan and Tian~\cite{RT}, we denote by ${\rm GW}_{0,l+k,A}^{\{1,\ldots,l\}}(a_1,\ldots,a_l,b_1,\ldots, b_k)$ the mixed Gromov-Witten invariant. This invariant formally counts (with appropriate signs) the tuples $(u,\vec{w})=(u,\{w_j\}_{j=1}^k)$ consisting of a $J$-holomorphic sphere $u:S^2\to M$ representing the class $A$ with $u(z_i)\in\alpha_i$ and $u(w_j)\in \beta_j$ for all $i,j$, where $\{z_i\}_{i=1}^l$ is a tuple of fixed pairwise distinct points on $S^2$ and $w_j\in S^2\setminus\{z_i\}_{i=1}^l$ are pairwise distinct marked points for all $j$. Here $\alpha_i$ and $\beta_j$ are generic representatives of the classes $a_i$ and $b_j$, respectively. For $l\geq3$, the precise definition uses the stabilization map
\[
 \operatorname{st}_I:\overline{\mathcal M}_{0,l+k}(M,A)
 \longrightarrow\overline{\mathcal M}_{0,l},\qquad I=\{1,\ldots,l\},
\]
which forgets the $k$ free marked points and stabilizes the domain.
Writing $\eta_l$ for the cohomology class dual to a point in the smooth
locus of $\overline{\mathcal M}_{0,l}$, the invariant is
\[
 \int_{[\overline{\mathcal M}_{0,l+k}(M,A)]^{\rm vir}}
 \operatorname{st}_I^*\eta_l\,
 \prod_{i=1}^{l}\operatorname{ev}_i^*\operatorname{PD}(a_i)
 \prod_{j=1}^{k}\operatorname{ev}_{l+j}^*\operatorname{PD}(b_j).
\]
The order of the factors displayed here determines the signs for odd insertions. For $l=3$, $\eta_3=1$, so this is the usual Gromov--Witten invariant. For $l>3$, the domain constraint must not be
omitted. The virtual class and its splitting identities are understood
in the framework of~\cite{FO,KM,FOOO}.

This quantity is zero unless
\begin{equation}\label{e:mixdim}
\sum_{i=1}^l(2n-\deg(a_i))+\sum_{j=1}^k(2n-\deg(b_j))=2n+2\langle c_1(TM),A\rangle+2k.
\end{equation}
Using the splitting, multilinearity and symmetry properties of the mixed Gromov--Witten invariants (cf.~\cite{RT} or~\cite[Chapter~11]{MS}), one can see that an $l$-fold  product is then obtained by extending linearly from the formula
\begin{equation}\label{e:expand}
a_1*_\zeta\cdots *_\zeta a_l=\sum_{A\in H_2(M)}\sum_{k=0}^{\infty}\frac{1}{k!}\sum_{j=1}^K{\rm GW}_{0,l+k+1,A}^{I}(a_1\ldots,a_l,c_j,\underbrace{\zeta,\ldots,\zeta}_k)T^{\langle [\omega],A\rangle}c^j
\end{equation}
for $l\geq2$, where $I=\{1,\ldots,l+1\}$, and
$\{c_j\}_{j=1}^K$ is a homogeneous basis for $H_*(M;\Q)$ with
left-dual basis $\{c^j\}$.
Indeed, move the point class on $\overline{\mathcal M}_{0,l+1}$ to a
maximally nodal tree and apply the splitting axiom. Distributing the
free even insertions over its vertices gives the multinomial factors
which combine with the vertex factorials to yield $1/k!$ in
\eqref{e:expand}. No ordinary $(l+k+1)$-point invariant without this
domain constraint is being substituted for the mixed invariant.

As a vector space over $\C$, the big quantum homology $H_*(M;\Lambda)$ is isomorphic to $H_*(M;\C)\otimes_\C\Lambda$. Each element $a\in H_*(M;\Lambda)$ can be written as $a=\sum_{g\in\Gamma_\omega}a_gT^g$ with $a_g\in H_*(M;\C)$ and the property that for any $C>0$,
$$\#\{g\in\Gamma_\omega\;|\;a_g\neq 0, g<C\}<\infty.$$
This latter finiteness condition gives rise to a  valuation function $I_\nu:H_*(M;\Lambda)\to\Gamma_\omega\cup\{-\infty\}$ defined by
\begin{equation}\label{e:val}
I_\nu(a)=\max\{-g|a_g\neq 0\},\qquad I_\nu(0)=-\infty.
\end{equation}

There exists a pairing $\Delta:H_*(M;\Lambda)\times H_*(M;\Lambda)\to \Lambda$ defined by \[
\Delta\bigg(\sum a_gT^g,\sum b_hT^h\bigg)=\sum a_g\circ b_h\cdot T^{g+h}
\]
where $\circ$ stands for the Poincar\'{e} intersection pairing.
It is easy to verify that the pairing is well-defined, and nondegenerate in the sense that $\Delta(a,b)=0$ for all $b$ then $a=0$. Then, by taking  coefficients of the zero-order term in $\Delta(a,b)$ with respect to the formal variable $T$, one can define a $\C$-valued  nondegenerate paring ${\textstyle\prod}:H_*(M;\Lambda)\times H_*(M;\Lambda)\to \C$ by
\begin{equation}\label{e:pair}
\prod\bigg(\sum a_gT^g,\sum b_hT^h\bigg)=\sum a_g\circ b_{-g}
\end{equation}
(see~\cite[p. 1359]{Us} or~\cite[Chapter~15]{FOOO}).

\subsubsection{Deformed quantum cuplength}
Schwarz~\cite[Section~2.2]{Sc2} generalized the notion of classical cuplength to small quantum cohomology, called \textit{quantum cuplength}, and employed it to obtain some estimates for the number of 1-periodic solutions of Hamiltonian equations. Now we further generalize this invariant to our deformed quantum homology as follows.

Given $\zeta\in  \oplus_{i=0}^{n-1}H_{2i}(M;\Lambda_0)$, for homology classes $\alpha_1,\ldots,\alpha_k\in H_*(M;\Lambda)$, using the splitting property of the Gromov invariants,  the $k$-fold big quantum product can be represented as a formal series in $H_*(M;\Lambda)$
\[
\alpha_1*_\zeta\cdots *_\zeta\alpha_k=\sum (\alpha_1*_\zeta\cdots *_\zeta\alpha_k)_g T^g.
\]

\begin{df}\label{df:bcl}
For $g\in\Gamma_\omega$, the $\zeta$-deformed quantum cuplength is
\[
\bcl(g)=\max\bigl\{k+1\;\big|\;k\in\Z_{\geq0},\ 
 a_1,\ldots,a_k\in H_{<2n}(M;\C),\ 
 (a_1*_\zeta\cdots *_\zeta a_k)_g\ne0\bigr\}.
\]
The empty product is $[M]$, and the value is defined to be zero if
the displayed set is empty. The following lemma shows that every
nonempty set in this definition has a maximum.
\end{df}

\begin{lem}\label{lem:qcl-finite}
For each $g\in\Gamma_\omega$ with $g\geq0$, there is an integer $C_g\geq0$, depending only on
$(M,\omega)$, $g$, and a fixed compatible almost complex structure,
such that
\[
 (a_1*_\zeta\cdots *_\zeta a_k)_g\ne0,\qquad
 a_i\in H_{<2n}(M;\C)
 \quad\Longrightarrow\quad k\leq2n+2C_g.
\]
This bound holds for every
$\zeta\in\bigoplus_{i=0}^{n-1}H_{2i}(M;\Lambda_0)$.
Consequently, $\bcl(g)$ is finite and its defining maximum is attained
whenever its defining set is nonempty. Moreover, $\bcl(g)=0$ for $g<0$.
\end{lem}

\begin{proof}
Fix an $\omega$-compatible almost complex structure $J$. Let
$\mathcal A_g$ consist of zero and the classes in $H_2(M)$ represented
by genus-zero $J$-holomorphic stable maps of area at most $g$.
This is a finite set. Indeed, the integral of any fixed smooth
closed two-form over such a map is bounded in absolute value by a
constant times its $\omega$-area. Applying this observation to a basis
of $H^2(M;\R)$ bounds all coordinates in the integral lattice
$H_2(M;\Z)/{\rm torsion}$. Set
\[
 C_g=\max\bigl(\{0\}\cup
       \{\langle c_1(TM),A\rangle:A\in\mathcal A_g\}\bigr).
\]

Consider a nonzero coefficient as in the statement, first with
$k\geq2$. Expand the ordinary inputs into homogeneous components
and write $\zeta=D+B$ as above, expressing $B$ in homogeneous ordinary
homology classes with coefficients in $\Lambda_0$. After summing the divisor insertions,
formula~\eqref{e:expand} expresses this coefficient in terms of mixed
invariants with $k$ ordinary inputs, one output-testing class $c$,
and some number $r$ of non-divisorial bulk insertions $b_1,\ldots,b_r$.
A nonzero coefficient requires at least one nonzero such invariant.
All exponents coming from bulk coefficients and divisor exponentials
are nonnegative. Its curve class therefore satisfies
$0\leq\omega(A)\leq g$ and belongs to $\mathcal A_g$: nonvanishing
of the invariant implies that its underlying stable-map moduli space
for the fixed $J$ is nonempty. In real codimensions, its dimension constraint is
\[
 \sum_{i=1}^k\operatorname{codim}(a_i)
 +\operatorname{codim}(c)
 +\sum_{j=1}^r\bigl(\operatorname{codim}(b_j)-2\bigr)
 =2n+2\langle c_1(TM),A\rangle.
\]
Every ordinary input has codimension at least one, the output-testing
class has nonnegative codimension, and every non-divisorial bulk
insertion has codimension at least four. Thus
$k\leq2n+2\langle c_1(TM),A\rangle\leq2n+2C_g$.
The bound is independent of the number and choice of the ordinary
inputs and of $\zeta$. The cases $k=0,1$ satisfy it as well because $n\geq1$.
If $g<0$, no coefficient is possible, since both curve areas and bulk
exponents are nonnegative. Finally, a nonempty bounded set of positive
integers has a largest element, proving the assertion about the maximum.
\end{proof}

\begin{rmk}
The zero-energy term is the ordinary intersection product:
\[
 a_1*_\zeta\cdots *_\zeta a_k
 =a_1\cap\cdots\cap a_k+\sum_{g>0}b_gT^g.
\]
Indeed, area zero forces a constant map. In class zero, every primary
invariant with more than three insertions vanishes, and the divisor
weight is one. The same conclusion for iterated products follows by
induction. Hence $\bcl(0)=\cl(M)$, with the convention that cuplength
includes the unit.
\end{rmk}

\subsubsection{Morse theoretical big quantum homology}

The $\zeta$-deformed quantum homology ring  $QH_*(M,\omega)_\zeta$ has a Morse theoretical interpretation which we describe as follows. By Thom's rational realization theorem~\cite[Th\'{e}or\`{e}me II.29]{Th}, after multiplying the elements of a homogeneous rational basis by suitable nonzero integers, we may choose a basis $\{\beta_i\}_{i=1}^m$ of $\oplus_{i=0}^{n-2}H_{2i}(M;\Q)$ represented by smooth maps
\begin{equation}\label{e:delta}
f_i:\Delta_i\longrightarrow M
\end{equation}
 (i.e., $f_{i*}[\Delta_i]=\beta_i$). Here, each $\Delta_i$, $i=1,\ldots, m$, is a smooth closed oriented manifold of dimension $2d(i)$. For simplicity, hereafter we denote
$$\delta(I):=\sum^k_{j=1}(2n-2d(i_j)-2).$$

Given $\zeta\in\oplus_{i=0}^{n-1}H_{2i}(M,\Lambda_0)$, we let $\theta\in \Omega^2(M;\Lambda_0)$ be a closed Novikov-ring-valued two-form which is Poincar\'{e} dual to the degree-$(2n-2)$ component of $\zeta$. Then we have
\begin{equation}\label{e:zeta}
\zeta-{\rm PD_M}[\theta]=\sum_{i=1}^ma_i\beta_i\in \bigoplus_{i=0}^{n-2}H_{2i}(M;\Lambda_0)
\end{equation}
where $a_i\in\Lambda_0$ for all $i\in\{1,\ldots,m\}$. For any $I=(i_1,\ldots,i_k)\in \{1,\ldots, m\}^k$, we set
\[
a_I=a_{i_1}\ldots a_{i_k}.
\]

Choose a generic Morse function $f$ whose critical points avoid the
fixed images $f_i(\Delta_i)$, and a generic Riemannian metric $\rho$
such that $(f,\rho)$ is Morse--Smale and its stable and unstable
manifolds are transverse to the maps $f_i$.
A grading
\[{\rm ind}_f:{\rm Crit}(f)\longrightarrow \{0,\ldots, 2n\}\]
is given by the Morse index. Denote
\[
\widehat{\mathcal{M}}_{x,y}(f,\rho)=\big\{\gamma\in C^\infty(\R, M)\;\big|\;\dot{\gamma}+{\rm grad}_\rho f(\gamma)=0,\;\gamma(-\infty)=x,\;\gamma(+\infty)=y\big\}.
\]
For $x\ne y$, we have $\dim \widehat{\mathcal{M}}_{x,y}(f,\rho)={\rm ind}_f(x)-{\rm ind}_f(y)$ and the translation action of $\R$ is free. If ${\rm ind}_f(x)={\rm ind}_f(y)+1$ then the moduli space
\[\mathcal{M}_{x,y}(f,\rho)=\widehat{\mathcal{M}}_{x,y}(f,\rho)/\R\]
is a compact 0-dimensional manifold.
Equipping unstable manifolds with  arbitrary orientations, together with the orientation of $M$, gives rise to orientations on $\widehat{\mathcal{M}}_{x,y}(f,\rho)$, which, compared with the canonical orientation by the flow, assign a sign to every element $[\gamma]\in \mathcal{M}_{x,y}(f,\rho)$ provided that ${\rm ind}_f(x)={\rm ind}_f(y)+1$. Then we obtain a Morse complex
\[
CM_k(f;\Z)=\sum_{\substack{p\in{\rm Crit}(f),\\{\rm ind}_f(p)=k}}\Z \langle p \rangle
\]
with differential \[\partial:CM_{k+1}(f;\Z)\to CM_k(f;\Z),
\quad
\partial x=\sum_{{\rm ind}_f(x)={\rm ind}_f(y)+1}\sharp_{\rm alg} \mathcal{M}_{x,y}(f,\rho)\cdot y
\]
where $\sharp_{\rm alg} \mathcal{M}_{x,y}(f,\rho)\in\Z$ is the number of the negative gradient flowlines counted with sign. The homology of the complex $(CM(f;\Z),\partial)$ is called  \textit{Morse homology} of $(f,\rho)$ with integer coefficients, naturally isomorphic to the singular homology  $H_*(M;\Z)$.

We now consider a complex with $\Lambda$-coefficients
\[
CM_*(f;\Lambda)=CM_*(f;\Z)\otimes_{\Z}\Lambda
\]
with differential $\partial^{f}:CM(f;\Lambda)\to CM(f;\Lambda)$ given by $\partial\otimes_{\Z} {\rm id}$. The homology of this complex is independent of the choice of $\rho$, and naturally isomorphic to $H_*(M;\Lambda)$. We denote it by $HM_*(f;\Lambda)$.

The $\zeta$-deformed
quantum multiplication on $H_*(M;\Lambda)$ has the following Morse theoretical description. Let $(f,\rho^0), (g,\rho^1), (h,\rho^2)$ be three Morse-Smale pairs as above.
Let $J$ be an almost complex structure on $M$ which smoothly depends on $z\in\C P^1$. Given a spherical class $A\in H_2(M)$, we consider the space of $J$-holomorphic spheres
\[
\mathcal{M}_A(J)=\{u:\C P^1\to M\;|\;du(z)\circ j=J(z,u(z))du(z),\;[u]=A\}.
\]
Now we define
\[
\mathcal{M}_k(A,J)=\big\{(u,\vec{z})\in \mathcal{M}_A(J)\times (\C P^1)^k\big\},
\]
and denote by $\overline{\mathcal{M}}_k(A,J)$ the compactification of this space with evaluation maps $ev_i\;(1\leq i\leq k)$ at the $i$-th marked point. For the marked space $(\C P^1,0,-1,\infty)$, each element $(u,\vec{z})$ in $\mathcal{M}_k(A,J)$ carries the natural evaluation maps $ev_0,ev_{-1},ev_\infty$ given by $ev_0(u,\vec{z})=u(0)$, $ev_{-1}(u,\vec{z})=u(-1)$ and $ev_\infty(u,\vec{z})=u(\infty)$. For $p\in\crit (f),q\in\crit (g),r\in\crit(h)$, let $\iota^0:W^u(p)\to M$, $\iota^1:W^u(q)\to M$, $\iota^2:W^s(r)\to M$ denote the inclusion maps. The fiber product
\[
\overline{\mathcal{M}}_k(A,J)_{(ev_0,ev_{-1},ev_\infty,ev_1,\ldots,ev_k)}\times _{(\iota^0,\iota^1,\iota^2,f_{i_1},\ldots,f_{i_k})}W^u(p)\times W^u(q)\times W^s(r)\times \Delta_{i_1}\times\ldots\times \Delta_{i_k}
\]
has a Kuranishi structure (cf.~\cite{FO,FOOO1,FOOO2}) with corners of dimension
$$\ind_f(p)+\ind_g(q)-\ind_h(r)+2\langle c_1(TM),A\rangle-2n-\delta(I).$$
Then where the associated multisections are denoted by $\mathsf{s}_{P,I}$, the
$\zeta$-deformed quantum product
\[*_\zeta^{\rm Morse}:CM_*(f;\Lambda)\times CM_*(g;\Lambda)\longrightarrow CM_*(h;\Lambda)\]
is defined by extending linearly from
$p*_\zeta^{\rm Morse}q$ for $p\in\crit(f)$ and $q\in \crit(g)$ given by
\[
p*_\zeta^{\rm Morse}q=\sum_{k=0}^\infty
\frac{1}{k!}\sum_{\substack{A\in H_2(M),\\r\in\crit(h)}}\sum_{\substack{I\in\{1,\ldots,m\}^k,\\\ind_f(p)+\ind_g(q)-\ind_h(r)=2n+\delta(I)-2c_1(A)}}|\mathsf{s}_{P,I}^{-1}(0)|\exp\big(\int_A\theta\big)a_IT^{\int_A\omega}r
\]
where $|\mathsf{s}_{P,I}^{-1}(0)|$ denotes the signed sum of the rational multiplicities of the points of the zero-dimensional vanishing locus of
$\mathsf{s}_{P,I}$. On homology level (after identifying the Morse homology with the $\zeta$-deformed homology), the map $*_\zeta^{\rm Morse}$ induces a map from $H_*(M;\Lambda)\otimes H_*(M;\Lambda)$ to $H_*(M;\Lambda)$ by sending $a\otimes b$ to
\[
\sum_{k=0}^\infty\frac{1}{k!}\sum_{j=1}^K\sum_{A\in H_2(M)}\sum_{I\in\{1,\ldots,m\}^k}{\rm GW}_{0,k+3,A}(a,b,c_j,a_{i_1}\beta_{i_1},\ldots,a_{i_k}\beta_{i_k})\exp\big(\int_A\theta\big)T^{\int_A\omega}c^j
\]
where$\{c_j\}_{j=1}^K$ is a homogeneous basis for $H_*(M;\Q)$ with dual basis $\{c^j\}$.
By (\ref{e:zeta}), we see that $\zeta={\rm PD_M}[\theta]+\sum^m_{i=1}a_i\beta_i$. It follows from the divisor axiom of Gromov-Witten invariants~\cite[Chapter~7]{MS} that the last formula is
\[
\sum_{k=0}^\infty\frac{1}{k!}\sum_{j=1}^K\sum_{A\in H_2(M)}{\rm GW}_{0,k+3,A}(a,b,c_j,\underbrace{\zeta,\ldots,\zeta}_k)T^{\int_A\omega}c^j
\]
which coincides with the $\zeta$-deformed quantum multiplication $*_\zeta$ on $H_*(M;\Lambda)$.

 \begin{rmk}
In the construction of big quantum homology ring,  one can avoid using Kuranishi structures by imposing the assumption that $(M,\omega)$ is {\it strongly semipositive} (i.e., for any $A\in\pi_2(M)$ with $2-n\leq\langle c_1(TM),A\rangle<0$ we have $\int_A\omega\leq 0$) and instead follow the constructions of Hofer and
Salamon~\cite{HS}. See~\cite[Section~3.2.2]{Us}. This will not affect our applications in Section~\ref{sect:app} since those examples listed there meet the latter condition. The generalization to all closed symplectic manifolds requires more subtle technical arguments. However, these arguments are less involved than the technical hurdles overcome in the construction of $A_\infty$--algebras associated to appropriate Lagrangian submanifolds~\cite{FOOO1}; therefore one
can adapt the methods from that work or~\cite{FOOO2}.

 \end{rmk}

\subsection{Deformed Hamiltonian Floer homology}
Following Usher~\cite{Us} (see also~\cite{FOOO}), we review the construction of deformed Hamiltonian Floer homology with additional algebraic structures.

\subsubsection{Deformed Floer complexes}
Let $\rH=C^\infty(S^1\times M,\R)$. For $H\in\rH$, we denote by $\{\varphi_H^t\}_{t\in[0,1]}$ the Hamiltonian isotopy of $H$ which is given by integrating the time-dependent vector field $X_{H_t}$, where $H_t=H(t,\cdot)$ and $X_{H_t}$ is determined uniquely by $-dH_t=\omega(X_{H_t},\cdot)$. Denote by ${\rm Per}(H)$ the set of contractible $1$-periodic orbits of the Hamiltonian flow of $H$. We say that $H$ is \textit{normalized} if $\int_M H_t\omega^n=0$ for all $t\in S^1$.

We consider the set $\mathcal{L}_{pt}(M)$ of all pairs $(x,\bar{x})$ where $x:S^1\to M$ is a loop and $\bar{x}:D^2\to M$ a disc with $\bar{x}|_{\partial D^2}=x$, called a \textit{capping} of $x$. We denote by $\pi_2(x)$ the set of homotopy classes of such $\bar{x}'s$.  We identify $(x,\bar{x})$ and $(x',\bar{x}')$ if $x=x'$ and $\bar{x}$ is homotopic to $\bar{x}'$ relative to the boundary $x$.

Given $H\in\rH$, we consider the action functional
$\ah:\mathcal{L}_{pt}(M)\longrightarrow \R$
given by
\[
\ah(x,\bar{x})=-\int_{D}\bar{x}^*\omega+\int^1_0H(t,x(t))dt.
\]
Let ${\rm Crit}(\ah)$ denote the set of critical points of $\ah$. It is easy to see that
\[{\rm Crit}(\ah)=\big\{(x,\bar{x})\in\mathcal{L}_{pt}(M)\;\big|\;\dot{x}(t)=X_{H_t}(x(t))\big\}.\]
For each $H\in\rH$, we define the \textit{action spectrum} of $H$ as
\[
{\rm Spec}(H)=\big\{\ah(x,\bar{x})\;\big|\;(x,\bar{x})\in {\rm Crit}(\ah)\big\}.
\]
It is shown that  for any $(M,\omega)$ and for any $H\in\rH$, the set ${\rm Spec}(H)$ has measure zero in $\R$ (see Oh~\cite{Oh} for the general case  and  Schwarz~\cite{Sc1} for aspherical $(M,\omega)$). Moreover, for any one-parameter smooth family of normalized periodic Hamiltonians $H^s:S^1\times M\to \R, s\in[0,1]$ with
\[
\varphi_{H^s}^0\equiv id,\quad
\varphi_{H^s}^1\equiv \varphi_{H^0}^1\quad \hbox{for all}\;s\in[0,1],
\]
we have ${\rm Spec}(H^s)={\rm Spec}(H^0)$;  see~\cite[Corollary 10.3]{FOOO}.

We say that $H$ is \textit{nondegenerate} if for each $p\in{\rm Fix}(\varphi_H^1)$, $1$ is not the eigenvalue of the linearization $d_p\varphi_H^1:T_pM\to T_pM$. Clearly, for a nondegenerate Hamiltonian $H$, the cardinality of the set ${\rm Per}(H)$ of one-periodic solutions of $\dot{x}(t)=X_{H_t}(x(t))$ is finite. In this case, we  have an integer grading
\[\mu_H:{\rm Crit}(\ah)\longrightarrow\Z\]
induced by the Conley-Zehnder index. More specifically, in this paper, the Conley-Zehnder index is taken to be the negative of the one defined in~\cite{CZ}. Consequently, for a $C^2$-small Morse function $f$, each constant Hamiltonian one-periodic orbit $q$ (together with  its constant capping $\bar{q}$) satisfies $\mu_H(q,\bar{q})=n-{\rm ind}_q(f)$, where ${\rm ind}_q(f)$ denotes the Morse index of $f$ at  $q$.

We introduce an equivalence relation $\sim$ on ${\rm Crit}(\ah)$ such that $(x_1,\bar{x}_1)\sim (x_2,\bar{x}_2)$ if and only if $x_1=x_2$ and  \[\int_{D^2}\bar{x}_1^*\omega=\int_{D^2}\bar{x}_2^*\omega.\]
Denote by $\mathcal{O}(H)$ the set of such equivalence classes. The equivalence class of $(x,\bar{x})\in{\rm Crit}(\ah)$ is denoted by $[x,\bar{x}]$. Then we have the maps
\[{\rm Crit}(\ah)\longrightarrow\mathcal{O}(H)\longrightarrow{\rm Per}(H)\]
by $(x,\bar{x})\mapsto [x,\bar{x}]$ and $[x,\bar{x}]\mapsto x$.

We define the Floer complex $CF_*(H)$ of a nondegenerate Hamiltonian $H\in\rH$ as the set of elements having the formal sum
\begin{equation}\label{e:chain}
\sum_{[x,\bar{x}]\in \mathcal{O}(H)} \alpha_{[x,\bar{x}]}[x,\bar{x}],\; \alpha_{[x,\bar{x}]}\in\C
\end{equation}
with the property that for every $C\in\R$,
\[\#\big\{[x,\bar{x}]|\alpha_{[x,\bar{x}]}\neq 0, \ah([x,\bar{x}])>C\big\}<\infty.\]
The Novikov field $\Lambda$ acts on
$CF_*(H)$ as follows: for each $\lambda\in \Gamma_\omega$, we pick an arbitrary class $A_\lambda\in\pi_2(M)$ satisfying $\langle [\omega],A_\lambda\rangle=\lambda$ and define the action
\[
\big(\sum a_iT^{\lambda_i}\big)\cdot\big(\sum \beta_{[x,\bar{x}]}[x,\bar{x}]\big)=\sum a_i\beta_{[x,\bar{x}]}[x,\bar{x}\sharp A_{\lambda_i}].
\]
It is easy to verify that this action is well-defined. Consequently, $CF_*(H)$ is a vector space over $\Lambda$ of finite dimension equal to $\sharp{\rm Per}(H)$. More precisely, for each $x\in {\rm Per}(H)$, if we pick a lifting $[x,\bar{x}]\in\mathcal{O}(H)$, then any element $c$ of $CF_*(H)$ is uniquely written as a finite sum
\begin{equation}\label{e:sum}
c=\sum_{x\in {\rm Per}(H)} a_x[x,\bar{x}],\;a_x\in\Lambda.
\end{equation}

\begin{df} The filtration function $\ell_H$ on $CF_*(H)$ is defined by
\[
\ell_H(c)=\max\big\{\ah([x,\bar{x}])\;\big|\;\alpha_{[x,\bar{x}]}\neq 0\big\}
\]
for any $c=\sum \alpha_{[x,\bar{x}]}[x,\bar{x}]$ of the form (\ref{e:chain}).
\end{df}

Equivalently, after fixing a lifting $x\to[x,\bar{x}]$ for each $x\in{\rm Per}(H)$, one can define the value of $\ell_H$ at $c$ given in the form (\ref{e:sum}) by
\[\ell_H(c)=\max\{\ah([x,\bar{x}])+\nu(a_x)|a_x\neq 0\}
\]
where $\nu$ is the valuation map on $\Lambda$ (cf.  Definition~$\ref{df:Nov}$(2)). The function $\ell_H$ gives rise to a filtration on $CF_*(H)$ by
\[
CF_*^\lambda(H)=\{c\in CF_*(H)\;|\;\ell_H(c)\leq \lambda\}\quad\hbox{for any}\;\lambda\in\R.
\]

\subsubsection{Gradings}
As a vector space over $\Lambda$, $CF_*(H)$ carries a natural $\Z_2$-grading which we define as follows.

\[
CF_0(H)=\bigoplus_{\substack{x\in{\rm Per}(H)\\n-\mu_H(x,\bar{x})=0\;{\rm mod}\;2}} \Lambda[x,\bar{x}],
\]
\[
CF_1(H)=\bigoplus_{\substack{x\in{\rm Per}(H)\\n-\mu_H(x,\bar{x})=1\;{\rm mod}\;2}} \Lambda[x,\bar{x}].
\]
Note that the parity of $\mu_H(x,\bar{x})$ depends only on the orbit $x\in{\rm Per}(H)$ and not on its capping $\bar{x}$. Indeed,  for any two cappings $w$ and $u$
of $x$, we have
\[
\mu_H(x,w)-\mu_H(x,u)=2\langle c_1(TM),[w\#
\bar{u}]\rangle,
\]
where $w\# \bar{u}:S^2\to M$ denotes the disc $u$ with opposite orientation so that $w\#
\bar{u}$ is a $2$-sphere.

\subsubsection{Maslov index of homotopy classes of cylinders}

Let $H_{\pm}\in\rH$ and $x_{\pm}\in{\rm Per}(H_{\pm})$. We denote by $\pi_2(x_-,x_+)$ the set of path components of the space of continuous maps $u:[-\infty,+\infty]\times S^1\to M$ satisfying $u(\{\pm\infty\}\times S^1)=x_\pm$.

Given $H_i\in\rH$ and $x_i\in{\rm Per}(H_i),i=0,1,2$, there is the natural gluing map
\[
\pi_2(x_0,x_1)\times \pi_2(x_1,x_2)\longrightarrow \pi_2(x_0,x_2)
\]
\[(C_0,C_1)\longmapsto C_0\sharp C_1\]
where $C_0\sharp C_1$ is the equivalence class of a gluing map $w$ of any two representatives $u,v$ of $C_0$ and $C_1$. More precisely, let $u,v:[-\infty,+\infty]\times S^1\to M$ be two continuous maps which represent $C_0$ and $C_1$ respectively, and define $w:[-\infty,+\infty]\times S^1\to M$
by putting
$w(s,t)=u(-\log (-s),t)$ for $s\leq0$ and
$w(s,t)=v(\log (s),t)$ for $s\geq 0$.  Here $\log:[0,\infty]\to[-\infty,\infty]$ is an extension of the natural logarithm.

Similarly, one can define a concatenation map
\[\pi_2(x_0)\times\pi_2(x_0,x_1)\longrightarrow\pi_2(x_1),\quad (\omega,C)\longmapsto \omega\sharp C.\]
For $x_{\pm}\in{\rm Per}(H_{\pm})$, we now use the Conley-Zehnder index to define a map $\mu$ on $\pi_2(x_-,x_+)$ as follows. For $C\in \pi_2(x_-,x_+)$, we pick an arbitrary representative $u:[-\infty,+\infty]\times S^1\to M$ of $C$, and pick any disc $w$ such that $w|_{\partial D^2}=x_-$ (note that $x_-$ is contractible by our assumption). We then glue the negative end of $u$ to the boundary of $w$ along $x_-$ to obtain a disc $w\sharp u:D^2\to M$ with $w\sharp u|_{\partial D^2}=x_+$. Then we define
\[
\mu(C)=\mu_{H_+}(x_+,w\sharp u)-\mu_{H_-}(x_-,w).
\]
Clearly, $\mu(C)$ is independent of the choice of $u$ and $w$, and has the property
\[\mu(C_0\sharp C_1)=\mu(C_0)+\mu(C_1)\quad \hbox{for all}\;  C_0\in \pi_2(x_0,x_1), C_1\in \pi_2(x_1,x_2).\]

\subsubsection{Deformed Floer boundary maps}

\begin{df}
We say that $H\in\rH$ is strongly nondegenerate if for each $p\in{\rm Fix}(\varphi_H^1)$, $1$ is not an eigenvalue of the linearization $d_p\varphi_H^1:T_pM\to T_pM$, and each image $f_i(\Delta_i)$, $i=1,\ldots, m$, does not intersect the set $\{\varphi_H^t(p)|t\in[0,1]\}$.
\end{df}

By the Sard-Smale theorem, one can show that the space of strongly nondegenerate Hamiltonians is residual in $\rH$ in the $C^l$-topology for all $2\leq l\leq\infty$. See~\cite[Appendix A]{Us}.

Let $J=\{J_t\}_{t\in S^1}$ be an $S^1$-family of $\omega$-compatible almost complex structures on $M$, and $H\in\rH$ a strongly nondegenerate Hamiltonian. Given $x_\pm\in{\rm Per}(H)$, we consider the solution $u:\R\times S^1\to M$ of the Floer equation:
\[
\overline{\partial}_{J,H}u:=\partial_su+J(t,u(s,t))\big(\partial_t u-X_H(t,u(s,t))\big)=0
\]
with the asymptotic boundary condition $u(s,t)=x_\pm(t)$ for $s\to\pm \infty$. Then $u$ extends continuously to a map $u:[-\infty,\infty]\times S^1\to M$ with $u(\pm\infty,t)=x_\pm(t)$, which represents some class $C\in\pi_2(x_-,x_+)$.  We denote by $\widehat{\mathcal{M}}(x_-,x_+,C;H,J)$ the moduli space of such solutions with
\[\int_{\R\times S^1}\big|\partial_s u\big|^2_{J_t}dsdt<\infty\]
where $|\cdot|_{J_t}$ is the norm associated to an $S^1$-family of metrics induced by $J$ for all $t$.

Consider the product space
\[\widehat{\mathcal{M}}_k(x_-,x_+,C;H,J)=\widehat{\mathcal{M}}(x_-,x_+,C;H,J)\times (\R\times S^1)^k.\]
Note that $\R$ acts on each factor by $s$-translation, we have a quotient
\[\mathcal{M}_k(x_-,x_+,C;H,J)=\widehat{\mathcal{M}}_k(x_-,x_+,C;H,J)/\R.\]
Using the constructions of Kuranishi structures  from~\cite[Theorem~19.44]{FO}, one can show that the
compactification of $\mathcal{M}_k(x_-,x_+,C;H,J)$, denoted by $\overline{\mathcal{M}}_k(x_-,x_+,C;H,J)$, admits an oriented Kuranishi structure with corners; see~\cite[Section~3.2.2]{Us}. Futhermore, one has evaluation maps
\[
\operatorname{ev}_1,\ldots, \operatorname{ev}_k:\overline{\mathcal{M}}_k(x_-,x_+,C;H,J)\longrightarrow M
\]
given by evluating at the $i$-th marked point, i.e. $\operatorname{ev}_i([u,z_1,\ldots,z_k])=u(z_i)$ for all $[u,z_1,\ldots,z_k]$ in the dense subset $\mathcal{M}_k(x_-,x_+,C;H,J)$ and  $i\in\{1,\ldots,k\}$.
Given an ordered set $I=(i_1,\ldots,i_k)$ with each $i_j\in\{1,\ldots,m\}$,  we consider the fiber product
\[\widetilde{\mathcal{M}}(x_-,x_+,C;\Delta_I):=\overline{\mathcal{M}}_k(x_-,x_+,C;H,J)_{\operatorname{ev}_1,\ldots,\operatorname{ev}_k}\times_{f_{i_1},\ldots,f_{i_k}}(\Delta_{i_1}\times\cdots\times \Delta_{i_k}).\]
It follows from~\cite[Appendix A1]{FOOO1} or  \cite[Section~3]{FOOO2} that there is an oriented Kuranishi structure with corners on $\widetilde{\mathcal{M}}(x_-,x_+,C;\Delta_I)$ which has the (virtual) dimension
\[
\dim \widetilde{\mathcal{M}}(x_-,x_+,C;\Delta_I)=\mu(C)-\sum^k_{j=1}(2n-2d(i_j)-2)-1.
\]
When $\mu(C)=\delta(I)+1$, we let $\mathsf{s}_{C,I}$ denote the multisection associated to Kuranishi structure on $\widetilde{\mathcal{M}}(x_-,x_+,C;\Delta_I)$, and denote $|\mathsf{s}_{C,I}^{-1}(0)|$ the signed sum of the rational multiplicities of the points of the zero-dimensional vanishing locus of
$\mathsf{s}_{C,I}$.

Following~\cite[Proposition~3.5]{Us} (see also~\cite[Chapter~6]{FOOO}), we will deform the standard Floer differential by a general homology class $\zeta\in\oplus_{i=0}^{n-1}H_{2i}(M;\Lambda_0)$.   For any $[x_-,\bar{x}_-],[x_+,\bar{x}_+]\in\mathcal{O}(H)$, we denote
\[
\mathsf{n}_{(H,J);k}\big([x_-,\bar{x}_-],[x_+,\bar{x}_+]\big)=\sum_{\substack{C\in \pi_2(x_-,x_+)\\ [x_+,\bar{x}_+]=[x_+,\bar{x}_-\sharp C]}}\sum_{\substack{I\in \{1,\ldots, m\}^k\\ \mu(C)=\delta(I)+1}}\frac{1}{k!}|\mathsf{s}_{C,I}^{-1}(0)|\exp\big(\int_C\theta\big)a_I.
\]
Here $\bar{x}_-\sharp C$ stands for the concatenation of $\bar{x}$ and any representative $u$ of the class $C$. Clearly, the equivalence class $[x_+,\bar{x}_-\sharp C]$ does not depend on the choice of $u$. Now we define the map $\partial_\zeta^{H,J}: CF_*(H)\to CF_*(H)$
by extending $\Lambda$-linearly from
\[
\partial_\zeta^{H,J}[x_-,\bar{x}_-]=\sum_{[x_+,\bar{x}_+]\in\mathcal{O}(H)}\sum_{k=0}^\infty \mathsf{n}_{(H,J);k}\big([x_-,\bar{x}_-],[x_+,\bar{x}_+]\big)[x_+,\bar{x}_+].
\]
Then Propositions~3.5 and 3.6 in~\cite{Us} imply the following.
\begin{prop}
For any $\zeta\in\oplus_{i=0}^{n-1}H_{2i}(M;\Lambda_0)$,  the $\partial_\zeta^{H,J}$ as above is a well-defined endomorphism and satisfies
\[\big(\partial_\zeta^{H,J}\big)^2=0.\]
\end{prop}

With this preparation regarding the complex, now we can define the $\zeta$-\textit{deformed Floer homology} of $(CF_*(H),\partial_\zeta^{H,J})$ as follows.
\begin{df}
The homology of the complex $(CF_*(H),\partial_\zeta^{H,J})$ is defined by
\[HF_\zeta(H)=\frac{{\rm ker}\;\partial_\zeta^{H,J} }{{\rm im}\;\partial_\zeta^{H,J}}.\]
\end{df}
It is shown in~\cite[Proposition~3.8]{Us} that for a strongly nondegenerate $H\in\rH$,  $HF_\zeta(H)$ is independent of the choice of the almost complex structure $J$. This is the reason why we suppress the $J$ from the notation of the deformed Floer homology. Note that the differential $\partial_\zeta^{H,J}$  preserves the $\C$-space $CF_*^\lambda(H)$, the corresponding homology is denoted by $HF_\zeta^\lambda(H)$, called \textit{filtered Floer homology} of $H$.

\subsubsection{Pair of pants products}\label{subsec:concatenation}
For any Hamiltonian $H:S^1\times M\to \R$, one can modify $H$ in a standard way such that the associated Hamiltonian isotopy is constant near $[1]\in S^1$, namely, letting  $H'(t,x)=\chi'(t)H(\chi(t),x)$ for some smooth function $\chi:[0,1]\to [0,1]$ with the property that $\chi=0$ near $t=0$ and $\chi=1$ near $t=1$. In what follows, we always assume that for some small $\epsilon>0$, $H(t,x)\equiv 0$ for all $x\in M$ for $|t|<\epsilon$. Given two such Hamiltonians $H,K$, we define a new Hamiltonian $H*K:S^1\times M\to\R$ by
\[
H*K(t,x)=
\begin{cases}
2H(2t,x)&0\leq t\leq 1/2,\\
2K(2t-1,x)&1/2\leq t\leq 1.\\
\end{cases}
\]

We now assume that $H,K$ and $H*K$ are  strongly nondegenerate Hamiltonians.
Then there exists a deformed pair of pants product on chain level, denoted by
\[\bar{*}_\zeta^{\rm PP}:CF_*(H)\otimes CF_*(K)\to CF_*(H*K),\]
which resembles the construction of the product on Hamiltonian Floer homology in~\cite{PSS} (see also~\cite{Lu}). The precise definition of the product $\bar{*}_\zeta^{\rm PP}$ is not used in this paper, and hence we will not go into details about it. Here we only mention that $\bar{*}_\zeta^{\rm PP}$ is a chain map that restricts to a map
\[\bar{*}_\zeta^{\rm PP}:CF^{\lambda}(H)\otimes CF^{\mu}(K)\to CF^{\lambda+\mu}(H*K)\]
for all $\lambda,\mu\in\R$.  On the homology level, it induces the map
\[*_\zeta^{\rm PP}:HF_\zeta^{\lambda}(H)\otimes HF_\zeta^{\mu}(K)\to HF_\zeta^{\lambda+\mu}(H*K).\]
See~\cite[Proposition~3.11]{Us}.

\subsection{PSS isomorphisms}\label{sec:pss}

Following~\cite[Section~3.3.2]{Us}, we will define a deformed Piunikhin-Salamon-Schwarz isomorphism from the Morse complex $(CM(f;\Lambda),\partial^f)$ to the $\zeta$-deformed Floer complexes $(CF_*(H),\partial_\zeta^{H,J})$. Isomorphisms of this type have been studied in~\cite{PSS,Lu}.

For a generic Morse-Smale pair $(f,\rho)$ and a strongly nondegenerate Hamiltonian $H$ on $M$, the deformed PSS map
\[
\psi^{\pss}_\zeta: (CM(f;\Lambda),\partial^f)\longrightarrow (CF_*(H),\partial_\zeta^{H,J})
\]
is defined by counting ``spiked discs'' with incidence conditions corresponding to $I=(i_1,\ldots,i_k)$ as follows.

Let $\chi:\R\to [0,1]$ be a smooth cutoff function satisfying $\chi(s)=0$ for $s\leq 0$, $\chi(s)=1$ for $s\geq 1$, and $\chi'(s)\geq 0$ for all $s$. Given an $S^1$-dependent family of almost complex structures $J$, let $\{J_s\}_{s\in[0,1]}$ be a two-parameter family of almost complex structures such that $J_{0,t}=J_0$,  $J_{1,t}=J_t$ and $J_{s,0}=J_0$ for all $s,t\in[0,1]$. Here $J_0$ is a time-independent almost complex structure compatible with $\omega$. We define the $\R\times S^1$-dependent pair $(H_\chi,J_\chi)$ by
\[
H_\chi(s,t)=\chi(s)H_t,\quad J_\chi(s,t)=J_{\chi(s),t}.
\]

Given $\gamma\in{\rm Per}(H)$ and $C\in\pi_2(\gamma)$, we choose any disc $u$ representing $C$ and denote $\mu(C)=\mu_H(\gamma,u)$. Let $\mathcal{M}^{\pss}(\gamma)$ denote the space of the solutions $u:\R\times S^1\to M$ of the Floer equation
\begin{equation}\label{e:sfl}
\partial_su+J_\chi\big(\partial_tu-X_{H_\chi}(u)\big)=0
\end{equation}
satisfying asymptotic boundary condition
$u(+\infty,t)=\gamma(t)$ for all $t$ and having finite energy
\[
E(u)=\int_{\R\times S^1}\big|\partial_s u\big|^2_{J_\chi}dsdt.
\]
Note that since $H_\chi\equiv0$ and $J_\chi\equiv J_0$ for $s<0$, the equation (\ref{e:sfl}) restricts to $J_0$-holomorphic curve on $(-\infty,0)\times S^1$ with finite energy. By the removable singularity theorem, the limit $\lim_{s\to-\infty}u(s,t)$ exists and is independent of $t$. So one can define the map $ev_{-\infty}:\mathcal{M}^{\pss}(\gamma)\to M$ by
\[
\operatorname{ev}_{-\infty}(u)=\lim_{s\to-\infty}u(s,t).
\]
We now define
\[
\mathcal{M}^{\pss}_k(\gamma,C)=\big\{(u,\vec{z})\in\mathcal{M}^{\pss}(\gamma)\times(\R\times S^1)^k\;\big|\;[u]=C\big\}.
\]
Denote its standard compactification by $\overline{\mathcal{M}}^{\pss}_k(\gamma,C)$. The evaluation maps $ev_i, i\in\{1,\ldots,k\}$ at the marked points, as well as the map $ev_{-\infty}$, extend naturally to this compactification.

Given $p\in\crit(f)$ and $I=(i_1,\ldots,i_k)\in\{1,\ldots,m\}^k$, let $\iota_p:W^u(p)\hookrightarrow M$ denote the inclusion of the unstable manifold of $p$. We consider the fiber product:
\[
\mathcal{M}^{\pss}_k(\gamma,C,p;\Delta_I)=\overline{\mathcal{M}}^{\pss}_k(\gamma,C)_{\operatorname{ev}_{-\infty},\operatorname{ev}_1,\ldots,\operatorname{ev}_k}\times_{\iota_p,f_{i_1},\dots,f_{i_k}}\big(W^u(p)\times\Delta_{i_1}\times\cdots\times\Delta_{i_k}\big).
\]
This space carries a Kuranishi structure with corners, and  its virtual dimension is
\[
\mu(C)+\ind_f(p)-n-\delta(I).
\]
Let $\mathsf{s}_{\pss}$ denote the corresponding multisection. Then the PSS map $$\psi^{\pss}_{\zeta,H}: (CM(f;\Lambda),\partial^f)\longrightarrow (CF_*(H),\partial_\zeta^{H,J})$$
is defined by extending $\Lambda$-linearly from
\[
\psi^{\pss}_{\zeta,H} p=\sum_{k=0}^\infty\frac{1}{k!}\sum_{\substack{\gamma\in{\rm Per}(H),\\C=[u]\in\pi_2(\gamma)}}\sum_{\substack{I\in\{1,\ldots,m\}^k,\\\mu(C)+\ind_f(p)=n+\delta(I)}}|\mathsf{s}_{\pss}^{-1}(0)|\exp\big(\int u^*\theta\big)a_I[\gamma,u].
\]
By~\cite[Proposition~3.9]{Us}, the deformed PSS map is a chain map
inducing a $\Lambda$-module isomorphism
\[
 \Psi^{\rm PSS}_{\zeta,H}:HM(f;\Lambda)\longrightarrow HF_\zeta(H).
\]
It preserves the $\Z_2$-gradings and, under the canonical continuation
identifications, is independent of the cutoff, metric, and almost
complex structures. Its compatibility with multiplication is the
separate statement~\cite[Proposition~3.12]{Us}: after identifying
Morse homology with $H_*(M;\Lambda)$, the following diagram commutes:
\[
\begin{CD}
 H_*(M;\Lambda)\otimes H_*(M;\Lambda)
 @>{*_\zeta}>> H_*(M;\Lambda)\\
 @V{\Psi^{\rm PSS}_{\zeta,H}\otimes\Psi^{\rm PSS}_{\zeta,K}}VV
 @VV{\Psi^{\rm PSS}_{\zeta,H*K}}V\\
 HF_\zeta(H)\otimes HF_\zeta(K)
 @>{*_\zeta^{\rm PP}}>> HF_\zeta(H*K).
\end{CD}
\]
This compatibility, together with the filtration estimate for the
pair-of-pants product, yields the spectral triangle inequality below.

\section{Spectral invariants}\label{Sec:spec}

\subsection{Minmax critical value selector}
Fix a ground field $\F$. Let $X$ be a closed, connected smooth
manifold of positive dimension, oriented over $\F$. We write $H_*(X)$
for singular homology with coefficients in $\F$, and $[X]$ for the
fundamental class with the chosen orientation.
Let $f\in C^\infty(X,\R)$ be a smooth function. For any $\lambda\in\R$, we define $$X^\lambda:=\{x\in X\;|\;f(x)<\lambda\}.$$

To a non-zero singular homology class $a\in H_*(X)\setminus\{0\}$, we associate a  numerical invariant defined by
$$c_{LS}(a,f)=\inf\{\lambda\in\R\;|\;a\in{\rm im}(i^\lambda_*)\},$$
where $i^\lambda_*:H_*(X^\lambda)\to H_*(X)$ is the map induced by the inclusion $i^\lambda:X^\lambda\to X$.
The function $$c_{LS}:H_*(X)\setminus\{0\}\times C^\infty(X,\R)\to\R$$ is called the \textit{minmax critical value selector} of $f$. The following  properties are well-known from classical Lusternik--Schnirelmann theory; see, e.g.,~\cite{HZ,Cha,CLOT,GG}.

\begin{prop}\label{pp:minmax}
	The minmax critical value selector $c_{LS}$ satisfies the following properties.
\begin{enumerate}
	\item[{\rm 1.}] $c_{LS}(a,f)$ is a critical value of $f$, and $c_{LS}(k a,f)=c_{LS}(a,f)$ for any nonzero $k\in\F$.
	\item[{\rm 2.}]  $c_{LS}(a,f)$ is Lipschitz continuous in $f$ with respect to the $C^{0}$-distance.
	\item[{\rm 3.}] Let $[pt]$ and $[X]$ denote the point class and the fundamental class, respectively. Then
$$c_{LS}([pt],f)=\min f\leq c_{LS}(a,f)\leq\max f= c_{LS}([X],f).$$
	\item[{\rm 4.}] If $b\in H_{<\dim X}(X)$, $a\cap b\neq0$, and $c_{LS}(a\cap b,f)=c_{LS}(a,f)$, then the set
     $$\Sigma=\{x\in\crit(f)\;|\;f(x)= c_{LS}(a,f)\}$$
     is homologically non-trivial; that is, for every open neighborhood $V$ of $\Sigma$, the induced map $i_*:H_k(V)\to H_k(X)$ by the inclusion $i:V\hookrightarrow X$ is nonzero for some $k>0$. In particular, such a set $\Sigma$ must be infinite.
 \end{enumerate}

\end{prop}

\subsection{Deformed Hamiltonian spectral invariants}
In this section, we follow the works of  Schwarz~\cite{Sc1} and Oh~\cite{Oh3} to define the spectral invariants and review their properties which will be useful in this paper.

For $\zeta\in   \oplus_{i=0}^{n-1}H_{2i}(M;\Lambda_0)$
and a strongly nondegenerate Hamiltonian $H\in\rH$, as in the previous section we have the PSS isomorphism
\[\Psi^{\rm PSS}_{\zeta,H}: H_*(M;\Lambda)\longrightarrow HF_\zeta(H).\]
Given $a\in H_*(M;\Lambda)\setminus\{0\}$, we put
\[
\rho_\zeta(a,H)=\inf\big\{\lambda\;|\;\Psi^{\rm PSS}_{\zeta,H}(a)\in \operatorname{im}( i^\lambda)\big\}
\]
where $i^\lambda:HF_\zeta^\lambda(H)\to HF_\zeta(H)$ is the natural map induced by the inclusion of $CF_*^\lambda(H)$ to $CF_*(H)$. Then $-\infty<\rho_\zeta(a,H)<+\infty$, and
for any two strongly nondegenerate Hamiltonians $H,K\in \rH$,   $\rho_\zeta(a,\cdot)$ satisfies
$$\int^1_0\min\limits_M(H_t-K_t)dt\leq \rho_\zeta(a,H)-\rho_\zeta(a,K)\leq \int^1_0\max\limits_M(H_t-K_t)dt.$$
This has been proved in~\cite{Us,FOOO}; see also~Oh~\cite{Oh3} for the undeformed case. Recall that the Hofer norm on the space of continuous functions $H:S^1\times M\to\R$ is given by
\[
\|H\|=\int^1_0\big(\max_M H(t,x)-\min_M H(t,x)\big)dt.
\]
Hence, if $H,K$ are both normalized,  then
\[|\rho_\zeta(a,H)-\rho_\zeta(a,K)|\leq \|H-K\|.\]
Consequently, one can extend $\rho_\zeta(a,\cdot)$ by continuity to the set of all normalized continuous functions $H:S^1\times M\to\R$. Note that adding a time-dependent smooth function $c:S^1\to\R$ to $H$ does not change the Hamiltonian flow of $H$ but only shifts the value of $\ah$ at each critical point by the constant $\int^1_0c(t)dt$, so we get
$$\rho_\zeta(a,H+c)=\rho_\zeta(a,H)+\int^1_0c(t)dt.$$ for all nonzero classes $a\in H(M,\Lambda)$. Then one can further extend $\rho_\zeta(a,\cdot)$ to a map on $C^0(S^1\times M,\R)$. The numbers $\rho_\zeta(a,H)$ are called \textit{deformed spectral invariants} of an element $H\in C^0(S^1\times M,\R)$ (if $a=0$, we put $\rho_\zeta(0,H)=-\infty$).

\begin{prop}[{\cite[Proposition~3.13]{Us} or \cite[Theorem~7.8]{FOOO}}]\label{prop:property}
The deformed spectral invariant $$\rho_\zeta:H_*(M;\Lambda)\times C^\infty(S^1\times M,\R)\to [-\infty,\infty)$$ has the following properties. In the spectrality and duality
statements the input $a$ is nonzero; as before, $\rho_\zeta(0,H)=-\infty$.
\begin{itemize}
    \item[{\rm (P1)}] {\rm Normalization:}  $I_\nu(a)+\int^1_0\min H(t,\cdot)dt\leq \rho_\zeta(a,H)\leq I_\nu(a)+\int^1_0\max H(t,\cdot)dt$ for all $H\in\rH$. In particular, $\rho_\zeta(a,0)=I_\nu(a)$.
    \item[{\rm (P2)}] {\rm Spectrality:} For $a\ne0$, if $H$ is strongly nondegenerate, then there is $c\in CF_*(H)$ such that $[c]=\Psi^{\rm PSS}_{\zeta,H}(a)$ and $\rho_\zeta(a,H)=\ell_H(c)$, and hence $\rho_\zeta(a,H)\in\spec(H)$. If $M$ is rational, for any $H\in\rH$ one also has $\rho_\zeta(a,H)\in\spec(H)$.
    \item[{\rm (P3)}] {\rm Homotopy invariance:} $\rho_\zeta(a,H)=\rho_\zeta(a,K)$, when $\varphi_H=\varphi_K$ in the universal covering of the group of  Hamiltonian diffeomorphisms $\widetilde{{\rm Ham}}(M,\omega)$, and $H,K$ are normalized.
    \item[{\rm (P4)}] {\rm Quantum shift:} $\rho_\zeta(\lambda a,H)=\rho_\zeta(a,H)+\nu(\lambda)$ for $a\ne0$ and $\lambda\in\Lambda\setminus\{0\}$.
    \item[{\rm (P5)}] {\rm Projective invariance:} $\rho_\zeta(w a,H)=\rho_\zeta(a,H)$ for any $0\neq w\in\C$.
    \item[{\rm (P6)}] {\rm Triangle inequality:} $\rho_\zeta(a*_\zeta b,H*K)\leq \rho_\zeta(a,H)+\rho_\zeta(b,K)$ for any $H,K\in \rH$.
    \item[{\rm (P7)}] {\rm Valuation inequality:}
    $\rho_\zeta(a+b,H)\leq\max\{\rho_\zeta(a,H),\rho_\zeta(b,H)\}$. And if $\rho_\zeta(a,H)\neq\rho_\zeta(b,H)$, then this equality holds.
    \item[{\rm (P8)}] {\rm Duality:} For the Hamiltonian $\overline{H}$ given by $\overline{H}(t,x)=-H(t,\varphi_H^t(x))$, one has
    \[\rho_\zeta(a,\overline{H})=-\inf\big\{\rho_\zeta(b,H)\;|\;b\in H_*(M;\Lambda),\; {\textstyle\prod}(b,a)\neq 0\big\}.\]

\end{itemize}

\end{prop}
\begin{rmk}
The properties in Proposition~\ref{prop:property} follow by adapting arguments in the undeformed case, see for instance~\cite{Oh3}. Here we only mention that the spectrality property for a strongly nondegenerate $H$ follows from~\cite[Corollary 1.3]{Us1} straightforwardly, and then for arbitrary $H$ on a rational symplectic manifold follows from~\cite[Lemma~7.2]{Oh2}, and that the second statement in the (P7)  follows from Remark~\ref{rmk:nonArchi}.
\end{rmk}

\begin{rmk}
For strongly semipositive $(M,\omega)$, the construction and properties in Proposition~\ref{prop:property} can be obtained by the approach of~\cite[Section~3.2]{Us} without Kuranishi structures. The time-independent reduction used below in Lemma~\ref{lem:computation} requires an additional equivariant perturbation argument. As explained in~\cite[Remark~3.16]{Us}, strong semipositivity alone does not remove the need for virtual techniques in that step. For general closed symplectic manifolds we use the virtual constructions of~\cite{FO,FOOO}, with the compatible bulk-deformation framework of~\cite{Us}.

\end{rmk}

\begin{df}\label{df:zetanorm}
Given $\zeta\in   \oplus_{i=0}^{n-1}H_{2i}(M;\Lambda_0)$, for any $H\in\rH$, we define
\[
\gamma_\zeta(H):= \rho_\zeta([M],H)+\rho_\zeta([M],\overline{H})
\]
and
\[
    \gamma_\zeta(\varphi):=\inf\{\gamma_\zeta(H)\;|\;\varphi=\varphi_H^1,\;H\in\rH\}.
\]
\end{df}
\begin{rmk}
According to the (P1), (P3) and (P6) as above, we see that $\gamma_\zeta$ satisfies
\begin{equation}\label{e:hofer}
0\leq \gamma_\zeta(\varphi)\leq \|\varphi\|:=\inf\{\|H\|\;|\;\varphi=\varphi_H^1,\;H\in\rH\}.
\end{equation}
We use the term \textit{$\zeta$-spectral norm} for this quantity.
Only its definition and the inequalities~\eqref{e:hofer} are used below;
no other norm property is needed in our arguments.
\end{rmk}

Spectral invariants are notoriously difficult to compute in general. However, for a $C^2$-small autonomous Hamiltonian $H:M\to\R$, the deformed Hamiltonian spectral invariants $\rho_\zeta$ of $H$ can be computed in terms of the minmax critical selector $c_{LS}$ for $H$.

\begin{lem}\label{lem:computation}
Let $H:M\to\R$ be a $C^2$-small smooth function (i.e., the $C^2$-norm $\|H\|_{C^2}$ is sufficiently small). Then $\rho_\zeta(a,H)=c_{LS}(a,H)$ for all $a\in H_*(M,\C)\setminus \{0\}$.
\end{lem}

\begin{proof}
First assume that $H$ is Morse and sufficiently $C^2$-small, and that its critical points avoid the representatives of the non-divisorial bulk classes. Then $H$ is slow and flat in the terminology of~\cite{Us}: its contractible periodic orbits of period at most one are constant, and the linearized flow at each critical point has no nonconstant periodic orbit of period at most one. Choose a time-independent compatible almost complex structure $J$ such that $(H,\rho_J)$ is Morse--Smale, where $\rho_J(\cdot,\cdot)=\omega(\cdot,J\cdot)$; see~\cite[Theorem~8.1]{SZ}.

We use the Morse--Floer reduction with bulk in the proof of~\cite[Theorem~3.15, pp.~1364--1365]{Us}, based on~\cite[Mainlemma~22.4]{FO}. With compatible virtual perturbation data, angular rotation of the plane components reduces the zero-dimensional PSS counts to the fixed configurations.
The virtual dimension argument for these configurations, including their spherical bubbles, excludes nonconstant spherical components and non-divisorial bulk marked points from the relevant counts. Thus only the Morse configurations in the trivial spherical class remain. The divisor weights on these configurations equal one, since their images are one-dimensional.

Use $H$ itself as the Morse function in the PSS construction, with the same metric $\rho_J$. The remaining PSS configurations from $q\in\crit(H)$ to $p\in\crit(H)$ are the spiked gradient trajectories determined by
\[
 \dot\gamma(s)=-\chi(s)\operatorname{grad}_{\rho_J}H(\gamma(s)),
 \qquad \gamma(0)\in W^u(q;H),
 \qquad \lim_{s\to+\infty}\gamma(s)=p.
\]
Their zero-dimensional count is the oriented count of $W^u(q;H)\cap W^s(p;H)$ when $\ind_H(q)=\ind_H(p)$. For $p\ne q$, a nonempty intersection would contain a nonconstant gradient trajectory and hence a one-dimensional flow orbit, contrary to the Morse--Smale dimension formula. For $p=q$, the only configuration is constant and has sign $+1$ with compatible orientation conventions. Consequently,
\[
 \psi^{\pss}_{\zeta,H}(q)=[q,0].
\]
Thus the $\Lambda$-linear bijection
\[
 \mathcal I:CM(H;\Lambda)\longrightarrow CF_*(H),
 \qquad q\longmapsto[q,0],
\]
is the PSS chain map for these choices.  The chain-map identity gives $\partial_\zeta^{H,J}\mathcal I=\mathcal I\partial^H$, so $\mathcal I$ is a chain isomorphism inducing the PSS isomorphism.

Equip $CM(H;\Lambda)$ with the filtration
\[
 L_H\!\left(\sum_q b_q q\right)
 =\max_{b_q\ne0}\{H(q)+\nu(b_q)\},
 \qquad L_H(0)=-\infty.
\]
Since $\mathcal A_H([q,0])=H(q)$, the map $\mathcal I$ preserves this filtration. It follows that
\[
 \rho_\zeta(a,H)
 =\inf\{L_H(c)\mid \partial^Hc=0,\ [c]=a
                  \text{ in }HM(H;\Lambda)\}.
\]
For $a\in H_*(M;\C)\setminus\{0\}$, extending the coefficients to $\Lambda$ does not change this minmax value. Indeed, let $P_0:CM(H;\Lambda)\to CM(H;\C)$ extract the coefficient of $T^0$ from each Morse coefficient. The Morse differential has coefficients in $\C$, so $P_0$ is a chain map. If $c$ represents $a$ over $\Lambda$, then $P_0c$ represents $a$ over $\C$, and
\[
 L_H(P_0c)\le L_H(c).
\]
The Morse minmax characterization of $c_{LS}$ gives $c_{LS}(a,H)\le\rho_\zeta(a,H)$. Conversely, every Morse cycle over $\C$ representing $a$ is also such a cycle over $\Lambda$, with the same filtration, giving the opposite inequality.

Finally, a general sufficiently $C^2$-small smooth $H$ admits $C^2$-small Morse approximations whose critical points avoid the bulk representatives. Apply the equality just proved and pass to the limit using the $C^0$-continuity of both $\rho_\zeta(a,\cdot)$ and $c_{LS}(a,\cdot)$.
\end{proof}

\subsection{Hamiltonian Lusternik--Schnirelmann inequality}

As we mentioned before, as a vector space over $\C$ the big quantum homology $H_*(M;\Lambda)$ is isomorphic to $H_*(M;\C)\otimes_\C\Lambda$. We set
\[
\widehat{H}_*(M;\Lambda)=H_{<2n}(M;\C)\otimes_\C\Lambda.
\]

\begin{thm}\label{thm:LjS}
Let $(M^{2n},\omega)$ be a closed rational symplectic manifold. Let $H\in\rH$, and let $\alpha,\beta\in H_*(M;\Lambda)$ with $\alpha\neq 0$. Then we have
\begin{equation*}
\rho_\zeta(\alpha *_\zeta \beta, H)\leq \rho_\zeta (\alpha,H)+I_\nu(\beta).
\end{equation*}
Furthermore, if $\beta\in\widehat{H}_*(M,\Lambda)$ is non-zero, then either
\[
\rho_\zeta(\alpha *_\zeta \beta, H)< \rho_\zeta (\alpha,H)+I_\nu(\beta)
\]
or the set ${\rm Fix}(\varphi_H^1)$ is homologically non-trivial in $M$ (and hence infinite).
\end{thm}

Next we postpone the proof of Theorem~\ref{thm:LjS} to the end of this section, and employ this theorem to give the following corollary which will be used in Section~\ref{sec:mainpf}.

\begin{cor}\label{cor:specineq}
Let $(M^{2n},\omega)$ be a closed rational symplectic manifold. Suppose that there exist $k$ homology classes $u_i\in \widehat{H}_{*}(M;\Lambda)$, $i=1,\ldots,k$ and $\beta\in H_*(M;\Lambda)$ such that $\beta*_\zeta u_1*_\zeta\cdots*_\zeta u_k\neq 0$ and $I_\nu(u_i)\leq 0$. Let $H\in\rH$ such that the ${\rm Fix}(\varphi_H^1)$ is a finite set. Then the action functional $\ah$ has $k+1$ critical points $(x_0,\overline{x}_0),\ldots,(x_{k},\overline{x}_{k})$ such that
\[
    \rho_\zeta(\alpha_k,H)=\ah(x_{k},\overline{x}_{k})<\ldots<\rho_\zeta(\alpha_0,H)=\ah(x_{0},\overline{x}_{0})
\]
where $\alpha_0=\beta$ and $\alpha_i=\alpha_{i-1}*_\zeta u_i$ for $i\in \{1,\ldots,k\}$.
\end{cor}

\begin{proof}
    The condition that $\beta*_\zeta u_1*_\zeta\cdots*_\zeta u_k\neq 0$ implies that each $\alpha_i$ is nonzero.
    Since ${\rm Fix}(\varphi_H^1)$ is finite and $I_\nu(u_i)\leq 0$ for any $i\in \{1,\ldots,k\}$, by Theorem~\ref{thm:LjS} we obtain immediately
    \[
        \rho_\zeta(\alpha_i,H)=\rho_\zeta(\alpha_{i-1}*_\zeta u_i,H)<\rho_\zeta(\alpha_{i-1},H)+I_\nu(u_i)\leq \rho_\zeta(\alpha_{i-1},H).
    \]
So we get
\[
\rho_\zeta(\alpha_k,H)<\ldots< \rho_\zeta(\alpha_0,H).
\]
Combining this with Proposition~\ref{prop:property} (P2) yields $k+1$ critical points $(x_0,\overline{x}_0),\ldots,(x_{k},\overline{x}_{k})$ of $\ah$ as required. \end{proof}

\noindent\textbf{Proof of Theorem~\ref{thm:LjS}.}
The argument we will make to show the theorem is an adaption of Howard~\cite{How}; see also~\cite{Sc1,GG,BHS,Go}. The key idea of the proof is to use Lemma~\ref{lem:computation} to transform the desired inequality on the spectral invariant $\rho_\zeta$  to the 4th property of $c_{LS}$ in Proposition~\ref{pp:minmax}.

First, by Proposition~\ref{prop:property}~(P6), we see that
\[
\rho_\zeta(\alpha *_\zeta \beta,H*0)\leq \rho_\zeta(\alpha,H)+\rho_\zeta(\beta,0).
\]
This, together with by Proposition~\ref{prop:property}~ (P1) and (P3), implies
\[
\rho_\zeta(\alpha *_\zeta \beta,H*0)\leq \rho_\zeta(\alpha,H)+I_\nu (\beta)
\]
where we have used the fact that $H*0$ and $H$ have the same mean value, that is,  $\int^1_0\int (H*0)_t\omega^n dt=\int^1_0\int H_t\omega^n dt$ and $\varphi_{H*0}^1=\varphi_H^1$ in $\widetilde{{\rm Ham}}(M,\omega)$.

Next, we prove the second statement. It suffices to show that if $\alpha\in H_*(M;\Lambda)$, $\beta\in \hat{H}(M;\Lambda)$ are two nonzero classes with the property
\begin{equation}\label{e:condition}
\rho_\zeta(\alpha *_\zeta \beta, H)=\rho_\zeta(\alpha,H)+I_\nu(\beta),
\end{equation}
then ${\rm Fix}(\varphi_H^1)$ is homologically  nontrivial in $M$. Fix an arbitrary open neighborhood $U$ of
${\rm Fix}(\varphi_H^1)$ and choose an open set $V$ such that
${\rm Fix}(\varphi_H^1)\subset V\subset\overline V\subset U$.
If the fixed-point set is all of $M$, the conclusion is immediate;
otherwise we may also choose $\overline V\ne M$.  We will show that
$H_k(U;\mathbb C)\to H_k(M;\mathbb C)$ is nonzero for some $k>0$.

Pick a $C^2$-small function $f:M\to\R$ such that $f|_{\overline{V}}\equiv 0$ and $f|_{M\setminus \overline{V}}<0$.  We shall prove the following

\begin{clm}\label{clm:spec}
For sufficiently small $\varepsilon>0$, we have
\[
\rho_\zeta(\alpha *_\zeta \beta, H*\varepsilon f)=\rho_\zeta(\alpha *_\zeta \beta, H).
\]
\end{clm}
\noindent\textbf{Proof of Claim~\ref{clm:spec}.} Consider now a family of spectral invariants $\rho_\zeta(\alpha *_\zeta \beta, H*s\varepsilon f)$ for $s\in[0,1]$. Here to guarantee that the concatenated Hamiltonians $H*s\varepsilon f$ make sense, one can assume that for some small $\epsilon>0$, $H=f\equiv0$ for $|t|<\epsilon$ by the reparameterization as before (cf. Section~\ref{subsec:concatenation}). In the notation below, $\varphi_f^\tau$ denotes the autonomous flow of the original function $f$; these time reparameterizations preserve time-one maps and action spectra. Since $\rho_\zeta(\alpha *_\zeta \beta, H)$ is continuous in $H$ in $C^0$-topology and the action spectra $\spec(H*s\varepsilon f)$ have measure zero for all $s$, we only need to prove that for $\varepsilon>0$ small enough, the sets $\spec(H*s\varepsilon f)$ are the same for all $s\in[0,1]$. Note that
\[
{\rm Fix}(\varphi^1_{H*\varepsilon f})={\rm Fix}(\varphi_H^1)
\]
provided that $\varepsilon$ is sufficiently small. In fact, we have $\varphi^1_{H*\varepsilon f}=\varphi^1_{\varepsilon f}\circ\varphi^1_H$ and $\varphi^t_{\varepsilon f}|_V\equiv Id$. Clearly,
${\rm Fix}(\varphi^1_{H*\varepsilon f})\supseteq {\rm Fix}(\varphi_H^1)$. Suppose now that $p\in M$ is a fixed point of $\varphi^1_{H*\varepsilon f}$, and hence $\varphi_H^1(p)=\varphi_f^{-\varepsilon}(p)$. Since the set $M\setminus V$ is compact and $q\neq\varphi_H^1(q)$ for any $q\in M\setminus V$ (recall that $V$ contains ${\rm Fix}(\varphi_H^1)$), one may choose $\varepsilon>0$ small enough such that $\varphi_f^{-\varepsilon}(q)\neq\varphi_H^1(q)$ for any $q\in M\setminus V$.  Therefore, $p\in V$, and hence $p=\varphi_H^1(p)$. So ${\rm Fix}(\varphi^1_{H*\varepsilon f})\subseteq {\rm Fix}(\varphi_H^1)$.

We now prove that $\spec(H*\varepsilon f)=\spec(H)$ for any given small $\varepsilon>0$. For each $x\in {\rm Fix}(\varphi_H^1)$ whose orbit $x(t)=\varphi_H^t(x)$ is contractible, choose a capping $\overline{x}:D^2\to M$ such that $\overline{x}(e^{2\pi it})=x(t)$. It is easy to see that
\[
\varphi_{H*\varepsilon f}^t(x)=
\begin{cases}
\varphi_H^{2t}(x)&0\leq t\leq 1/2,\\
x&1/2\leq t\leq 1.\\
\end{cases}
\]

Corresponding to each critical point $(x,\overline{x})$ of $\ah$, we denote by $x'(t)=\varphi_{H*\varepsilon f}^t(x)$ and let $\overline{x}':D^2\to M$ be a topological disc given by
\[
\overline{x}'(re^{2\pi it})=
\begin{cases}
\overline{x}(re^{4\pi it})&(r,t)\in [0,1]\times [0,1/2],\\
\overline{x}(r)&(r,t)\in [0,1]\times [1/2,1].\\
\end{cases}
\]
Clearly, on the boundary of this disc, $\overline{x}'$ restricts as the periodic orbit $x'$. So we get a bijection
$$\mathcal{J}:\crit(\ah)\longrightarrow\crit(\mathcal{A}_{H*\varepsilon f}),\quad (x,\overline{x})\mapsto (x',\overline{x}').$$
A straightforward computation yields $\mathcal{A}_{H*\varepsilon f}(\mathcal J(x,\overline x))=\ah(x,\overline x)$. This gives $\spec(H*\varepsilon f)=\spec(H)$. Similarly, one can show that
$\spec(H*s\varepsilon f)=\spec(H)$ for all $s$. This completes the proof of the claim.

Now we are in a position to finish the proof. By our assumption, for $\beta\in\widehat{H}(M;\Lambda)$,
choose a homogeneous $\C$-basis $\{\beta_i\}_{i=1}^N$ of $H_{<2n}(M;\C)$ and write $\beta=\sum_{i=1}^N\beta_i\otimes k_i$. Omitting zero coefficients in the following maxima, we have $I_\nu(\beta)=\max_i\nu(k_i)$. By Proposition~\ref{prop:property} (P1), (P4) and (P7),  for given sufficiently small $\varepsilon>0$, the following inequalities hold:
\begin{equation}\label{e:f}
\begin{aligned}
\rho_\zeta (\beta,\varepsilon f)&\leq  \max_i\big\{\rho_\zeta(\beta_i\otimes k_i,\varepsilon f)\big\}\\
&= \max_i\big\{\rho_\zeta(\beta_i,\varepsilon f)+\nu(k_i)\big\}\\
&\leq \max_i\big\{\rho_\zeta(\beta_i,\varepsilon f)\big\}+I_\nu(\beta).
\end{aligned}
\end{equation}
Proposition~\ref{prop:property} (P6) and Claim~\ref{clm:spec} give
\begin{equation}\label{e:H}
\begin{aligned}
\rho_\zeta (\alpha *_\zeta\beta,H)
&=  \rho_\zeta(\alpha *_\zeta\beta,H*\varepsilon f)\\
&\leq \rho_\zeta(\alpha,H)+\rho_\zeta(\beta,\varepsilon f)\\
&\leq \rho_\zeta(\alpha,H)+\max_i\big\{\rho_\zeta(\beta_i,\varepsilon f)\big\}+I_\nu(\beta),
\end{aligned}
\end{equation}
where in the last inequality we have used (\ref{e:f}). On the one hand, from (\ref{e:H}) we deduce that there exists $i_0\in\N$ such that
\[
\rho_\zeta(\beta_{i_0},\varepsilon f)\geq \rho_\zeta (\alpha *_\zeta\beta,H)-\rho_\zeta(\alpha,H)-I_\nu(\beta).
\]
Then, by (\ref{e:condition}) and Lemma~\ref{lem:computation}, we get $$c_{LS}(\beta_{i_0},\varepsilon f)=\rho_\zeta(\beta_{i_0},\varepsilon f)\geq 0.$$
On the other hand, it follows from Proposition~\ref{pp:minmax}.3 that
\[
c_{LS}(\beta_{i_0},\varepsilon f)=c_{LS}([M]\cap\beta_{i_0},\varepsilon f)\leq c_{LS}([M],\varepsilon f)=\max \varepsilon f=0.
\]
Consequently, we obtain
\[
c_{LS}([M]\cap\beta_{i_0},\varepsilon f)=c_{LS}([M],\varepsilon f)=0
\]
for some nonzero class $\beta_{i_0}\in H_{<2n}(M;\C)$. Proposition~\ref{pp:minmax}.4, applied to the zero level set
$\overline V$ and its open neighborhood $U$, implies that
$H_k(U;\mathbb C)\to H_k(M;\mathbb C)$ is nonzero for some $k>0$.
Since $U$ was arbitrary, this proves the second statement. \qed

\section{Nonvanishing Gromov-Witten invariants}\label{sec:gw}

We begin with the following lemma whose proof is essentially due to Usher~\cite[Lemma~4.2]{Us} (see also~\cite[Lemma~A.3]{Mc} or \cite[Lemma~7.1]{Lu1}).

\begin{lem}\label{lem:kfoldprod}
    Let $(M^{2n},\omega)$ be a closed symplectic manifold which admits a nonzero mixed Gromov-Witten invariant of the form:
    \[
    {\rm GW}_{0,l+k+1,A}^{\{1,\ldots,l+1\}}(a_1,\ldots,a_l,[pt],b_1,\ldots,b_k)
    \]
    where $A\in H_2(M)$, $a_1,\ldots, a_l\in H_{*<2n}(M;\Q)$ with $2\leq l\in\N$, and the classes $b_1,\ldots,b_k$ are rational homology classes of even degree. Then there is an open dense set of choices of the deformation parameter $\zeta\in\oplus_{i=0}^{n-1}H_{2i}(M;\C)$ such that the $l$-fold product $a_1 *_\zeta\cdots *_\zeta a_l = \lambda[M]+\beta$ with $\beta\in H_{*<2n}(M;\Lambda_0)$ satisfies $\nu(\lambda)\geq -\langle [\omega],A\rangle$.
\end{lem}
\begin{proof} First we show that the class $A$ appearing in our nonzero
Gromov-Witten invariant is nonzero.
According to the fundamental class axiom of mixed Gromov-Witten invariants (see~\cite[Chapter~7]{MS}), none of the classes $b_1,\ldots, b_k$ can be a multiple of the fundamental class.  If $A=0$, then the dimension formula~(\ref{e:mixdim}) yields
\[
 \sum_{i=1}^l(2n-\deg(a_i))+2n+\sum_{j=1}^k(2n-2-\deg(b_j))=2n
\]
contradicting our assumption that $2n-\deg(a_i)\geq 1$ for all $i$.

By the divisor axiom of the mixed Gromov-Witten invariants, one may assume that $b_1,\ldots, b_k\in \oplus_{i=0}^{n-2}H_{2i}(M;\Q)$.
Let $\xi_1,\ldots,\xi_N$ be a homogeneous basis for the $\Z$-module $\oplus_{i=0}^{n-1} H_{2i}(M;\mathbb{Z})/{\rm torsion}$. We may order them so that $\xi_1,\ldots,\xi_s$ form a basis for $H_{2n-2}(M;\mathbb{Z})/{\rm torsion}$ for some $0\leq s\leq N$. Empty lists and products are allowed when $s=N$. Then by the multilinearity and symmetry properties of the Gromov-Witten invariants, we may further assume that our nonzero Gromov-Witten invariant is of the form
\begin{equation}\label{e:ngw}
{\rm GW}_{0,l+1+\sum_j\alpha_j,A}^{\{1,\ldots,l+1\}}(a_1,\ldots,a_l,[pt],\underbrace{\xi_{s+1},\ldots,\xi_{s+1}}_{\alpha_{s+1}},\ldots,\underbrace{\xi_{N},\ldots,\xi_{N}}_{\alpha_{N}})
\end{equation}
for some $\alpha=(\alpha_{s+1},\ldots,\alpha_N)\in\Z_{\geq0}^{N-s}$ with $\sum_{j=s+1}^N\alpha_j\leq k$.

For $\vec{x}=(x_1,\ldots,x_s)\in\C^s$ and $\vec{y}=(y_{s+1},\ldots,y_N)\in\C^{N-s}$, we put
\[
\zeta(x,y)=\sum_{i=1}^sx_i\xi_i+\sum_{j=s+1}^Ny_j\xi_j.
\]
Let us consider the class $a_1 *_\zeta\cdots *_\zeta a_l$ as a function of $(\vec{x},\vec{y})$. Take a homogeneous basis  $\{c_j\}_{j=1}^K$ for $H_*(M;\Q)$ with dual basis $\{c^j\}_{j=1}^K$ with respect to the Poincar\'{e} homology intersection pairing $\circ$. According to the symmetry properties of the Gromov-Witten invariants, after expanding out the formula for $a_1 *_\zeta\cdots *_\zeta a_l$, we see that the coefficient of the class $[M]$ is an expression of the form
\[
\sum_{g\in\Gamma_\omega}\sum_{\beta=(\beta_{s+1},\ldots,\beta_N)}f_{g,\beta}(\vec{x})y_{s+1}^{\beta_{s+1}}\cdots y_N^{\beta_N}T^g
\]
For each fixed $g$, only finitely many curve classes contribute,
and the dimension formula~\eqref{e:mixdim} bounds the number of
non-divisorial insertions. Thus $f_{g,\beta}$ is identically zero
for all but finitely many $\beta$. To examine the coefficient at
$g=\omega(A)$ and $\beta=\alpha$, put
\[
 C_B={\rm GW}_{0,l+1+\sum_j\alpha_j,B}^{\{1,\ldots,l+1\}}
 (a_1,\ldots,a_l,[pt],\xi_{s+1}^{[\alpha_{s+1}]},\ldots,
 \xi_N^{[\alpha_N]}),
\]
where $\xi_i^{[\alpha_i]}$ denotes $\alpha_i$ repeated insertions.
Then
\[
 \left(\prod_{i=s+1}^N\alpha_i!\right)
 f_{\omega(A),\alpha}(\vec x)
 =\sum_{\substack{B\in H_2(M)\\\omega(B)=\omega(A)}}
 C_B\exp\!\left(\sum_{i=1}^s x_i(\xi_i\circ B)\right).
\]
Gromov compactness bounds the number of contributing curve classes
at this area, and the dimension constraint bounds the number of
non-divisorial insertions.

We notice that the coefficient of $\prod_{i=1}^s(e^{x_i})^{\xi_i\circ A}$ in the above sum is nonzero because of our assumed nonzero Gromov-Witten~(\ref{e:ngw}).   Now $H_{2n-2}(M;\mathbb{Z})/{\rm torsion}$ has $\{\xi_i\}_{i=1}^s$ as a basis, the map $B\mapsto (\xi_1\circ B,\ldots,\xi_s\circ B)$ is injective. Consequently,  in the expansion of  $a_1 *_\zeta\cdots *_\zeta a_l$, the coefficient of $T^{\langle [\omega],A\rangle}[M]$ is a Laurent polynomial in the variables $e^{x_i}$ and a polynomial in the variables $y_j$ having a nonzero coefficient multiplying
\[
\bigg(\prod_{i=1}^s(e^{x_i})^{\xi_i\circ A}\bigg)\prod_{j=s+1}^Ny_j^{\alpha_j}.
\]
Then it follows from the fact that this polynomial does not vanish identically that there is an open dense set of choices of $(\vec{x},\vec{y})$ at which the coefficient of $T^{\langle [\omega],A\rangle}[M]$ is not zero. This proves that for such  choice of $(\vec{x},\vec{y})$,  writing $a_1 *_{\zeta(\vec{x},\vec{y})}\cdots *_{\zeta(\vec{x},\vec{y})}a_l=\lambda[M]+\beta$  with $\beta\in H_{*<2n}(M;\Lambda_0)$, we have $\nu(\lambda)\geq -\langle [\omega],A\rangle$.
\end{proof}

Before finishing this section, we notice that
the argument used in the proof of Lemma~\ref{lem:kfoldprod} also yields the following
\begin{lem}\label{lem:mixnon0}
Let $(M^{2n},\omega)$ be a closed symplectic manifold which admits a nonzero mixed Gromov-Witten invariant of the form:
\[
{\rm GW}_{0,l+k+1,A}^{\{1,\ldots,l+1\}}(a,\ldots,a,b_0,\ldots,b_k)
\]
where $A\in H_2(M)$, the class $a\in H_{*<2n}(M;\Q)$ occurs $l$ times for some $l>2n$, and $b_0,\ldots,b_k$ (with $k\geq 0$) are  rational homology classes of even degree . Then there is an open dense set of possible choices of the deformation parameter $\zeta\in\oplus_{i=0}^{n-1}H_{2i}(M;\C)$ such that
$ (\underbrace{a *_\zeta\cdots *_\zeta a}_l)_{\langle[\omega],A\rangle}\neq 0$, and in particular the $l$-fold product of $a$ w.r.t. $*_\zeta$ is nonzero.

\end{lem}

\begin{proof}
First note that the class $A$ appearing in the mixed Gromov--Witten invariant cannot be zero.  Indeed, by the fundamental class axiom, none of the classes $b_1,\ldots,b_k$ can be a nonzero multiple of $[M]$.  If $A=0$, the dimension formula~(\ref{e:mixdim}) gives
\[
 l(2n-\deg(a))+2n-\deg(b_0)+\sum_{j=1}^k(2n-2-\deg(b_j))=2n,
\]
which contradicts $l>2n$ and $2n-\deg(a)\geq1$.

As in the proof of Lemma~\ref{lem:kfoldprod}, use the divisor axiom to absorb all divisor insertions into exponential factors and choose a homogeneous integral basis $\xi_1,\ldots,\xi_N$ of $\oplus_{i=0}^{n-1}H_{2i}(M;\Z)/{\rm torsion}$ such that $\xi_1,\ldots,\xi_s$ span $H_{2n-2}(M;\Z)/{\rm torsion}$.  By multilinearity, the given nonzero invariant has a nonzero component in which the remaining insertions are basis elements $\xi_{s+1},\ldots,\xi_N$ with multiplicities $\alpha_{s+1},\ldots,\alpha_N$.  For
\[
\zeta(\vec x,\vec y)=\sum_{i=1}^s x_i\xi_i+\sum_{j=s+1}^N y_j\xi_j,
\]
consider the coefficient of the period $\langle[\omega],A\rangle$ in the product $a*_\zeta\cdots *_\zeta a$.  Pairing this coefficient with the dual of $b_0$ gives a polynomial in $e^{x_i}$ and $y_j$.  The monomial
\[
\bigg(\prod_{i=1}^s(e^{x_i})^{\xi_i\cap A}\bigg)\prod_{j=s+1}^Ny_j^{\alpha_j}
\]
has nonzero coefficient by the assumed nonzero mixed invariant, and the injectivity of $B\mapsto(\xi_1\cap B,\ldots,\xi_s\cap B)$ on $H_2(M;\Z)/{\rm torsion}$ separates this monomial from the contributions of other classes with the same symplectic area.  Hence this polynomial is not identically zero.  For an open dense set of parameters $(\vec x,\vec y)$, the coefficient of $T^{\langle[\omega],A\rangle}$ in the $l$-fold product is nonzero.  This proves $(a*_\zeta\cdots *_\zeta a)_{\langle[\omega],A\rangle}\neq0$ for such $\zeta$, and the lemma follows.
\end{proof}

\section{Proofs of the fixed-point estimates}\label{sec:mainpf}
We first record the consequence of the minimum-area hypothesis used in the proof of Theorem~\ref{thm:onept}.

\begin{lem}\label{lem:min-area-valuation}
Assume the hypotheses of Theorem~\ref{thm:onept}.  Let $\zeta$ be chosen as in Lemma~\ref{lem:kfoldprod}, and write
\[
a_1 *_\zeta\cdots *_\zeta a_l=\lambda[M]+\beta,
\qquad \beta\in H_{*<2n}(M;\Lambda_0).
\]
Then
\[
\nu(\lambda)=-\langle[\omega],A\rangle,
\]
and either $\beta=0$, or
\[
I_\nu(\beta)\leq\nu(\lambda).
\]
\end{lem}
\begin{proof}
Expand the left-hand side by formula~(\ref{e:expand}).  The coefficient of any homogeneous basis element $c^j$ is a Novikov series whose terms are determined by mixed Gromov--Witten invariants of the form
\[
{\rm GW}_{0,l+i+1,B}^{\{1,\ldots,l+1\}}(a_1,\ldots,a_l,c_j,b_1,\ldots,b_i),
\]
where the additional insertions come from the components of the bulk parameter $\zeta$.  By the minimal-area assumption~(\ref{e:infarea}), every nonzero such term has symplectic area at least $\langle[\omega],A\rangle$.  Lemma~\ref{lem:kfoldprod} gives a nonzero contribution to the fundamental-class coefficient with area $\langle[\omega],A\rangle$.  The same minimal-area assumption excludes every contribution of smaller area to that coefficient.  Hence
\[
\nu(\lambda)=-\langle[\omega],A\rangle.
\]
All terms contributing to $\beta$ have valuation at most $-\langle[\omega],A\rangle$.  Thus, if $\beta\neq0$, then $I_\nu(\beta)\leq\nu(\lambda)$.
\end{proof}

\begin{proof}[Proof of Theorem~\ref{thm:onept}]
Let $\varphi=\varphi_H^1$ for some $H\in\rH$. Without loss of generality, we may assume that the number of elements in ${\rm Fix}(\varphi_H^1)$ is finite, otherwise, nothing needs to be proved.  The proof of the theorem is divided into two parts.

\noindent (1) Since our symplectic manifold $(M,\omega)$ carries the nonzero mixed Gromov-Witten invariant, by Lemma~\ref{lem:kfoldprod}, there is a class $\zeta\in\oplus_{i=0}^{n-1}H_{2i}(M;\C)$ such that
\begin{equation}\label{e:a1al}
    a_1 *_\zeta\cdots *_\zeta a_l = \lambda[M]+\beta
\end{equation}
with $\beta\in H_{*<2n}(M;\Lambda_0)$ and $\nu(\lambda)\geq -\langle [\omega],A\rangle$. Set $\alpha_0=[M]$ and
\[
\alpha_j=a_1*_\zeta\cdots *_\zeta a_j,\quad j=1,\ldots, l.
\]
It follows from Corollary~\ref{cor:specineq} that there exist elements $\tilde{x}_0=(x_0,\overline{x}_0),\ldots,\tilde{x}_l=(x_{l},\overline{x}_{l})\in\crit(\ah)$ such that
\begin{equation}\label{e:lactions}
    \rho_\zeta(\alpha_l,H)=\ah(x_{l},\overline{x}_{l})<\ldots<\rho_\zeta(\alpha_0,H)=\ah(x_{0},\overline{x}_{0}).
\end{equation}
If $\beta\ne0$, Theorem~\ref{thm:LjS} gives
\begin{gather}
\rho_\zeta(\beta,H)=\rho_\zeta([M]*_\zeta\beta,H)<\rho_\zeta([M],H)+I_\nu(\beta)\label{e:comparisonVal}
\end{gather}
By our hypothesis, the class $A$ minimizes $\langle [\omega],A\rangle$ among all $B\in H_2(M)$ with a nonzero mixed Gromov-Witten invariant of the form
\[
{\rm GW}_{0,l+1+k,B}^{1,\dots,l+1}(a_1,\ldots,a_l,c,b_1,\ldots,b_k)
\]
for any $c\in H_*(M;\Q)$ and even-degree classes $b_j$.  By the definition of $l$-fold  products (cf.~(\ref{e:expand})), this implies
\[
I_\nu(\beta)\leq-\langle[\omega],A\rangle=\nu(\lambda)
\]
because the coefficient at area $\omega(A)$ is nonzero and no
smaller-area term can occur. Combining this with~\eqref{e:comparisonVal} yields
\[
    \rho_\zeta(\beta,H)<\rho_\zeta(\lambda [M],H).
\]
This inequality also holds when $\beta=0$, since $\lambda\ne0$ and $\rho_\zeta(0,H)=-\infty$.
Hence, by Proposition~\ref{prop:property} (P7), we get
\[
\rho_\zeta(\alpha_l,H)=\rho_\zeta(\lambda[M]+\beta,H)=\rho_\zeta([M],H)+\nu(\lambda)=\rho_\zeta([M],H)-\langle [\omega],A \rangle,
\]
equivalently,
\begin{equation}\label{e:normbound}
\rho_\zeta(\alpha_0,H)-\rho_\zeta(\alpha_l,H)\leq\langle [\omega],A \rangle.
\end{equation}

\noindent (2) For estimating $\#{\rm Fix}(\varphi^1_H)$, we set
\[
\mathcal{E}=\big\{\ah(\widetilde{x}_i)|i=0,\ldots,l\big\},
\]
and define an equivalence relation on $\mathcal{E}$ as follows: for $a,b\in\mathcal{E}$,  $a\sim b$ if and only if $a-b$ is a symplectic period, i.e., $a-b\in\Gamma_\omega$. Denote by $\widehat{\mathcal{E}}$ the set of such equivalence classes. Hence, the cardinality $\#\widehat{\mathcal{E}}$ gives a lower bound for the number of distinct Hamiltonian periodic orbits $x_i$, and hence for $\#{\rm Fix}(\varphi_H^1)$.

Let $k=\langle [\omega], A \rangle/p$ (recall that $p\in\Gamma_\omega$ is the positive generator of $\Gamma_\omega$).   According to (\ref{e:lactions}) and (\ref{e:normbound}), there is at most one equivalence class containing $k+1$ elements  in $\mathcal{E}$, and each remaining class contains at most $k$ elements. Therefore,
\[
\#{\rm Fix}(\varphi^1_H)\geq \#\widehat{\mathcal{E}}\geq 1+\bigg\lceil\frac{l+1-k-1}{k}\bigg\rceil=\bigg\lceil\frac{pl}{\langle [\omega],A \rangle}\bigg\rceil.
\]
This completes the proof of the estimate (\ref{e:thm1}). \end{proof}

\begin{proof}[Proof of Corollary~\ref{cor:hbar}]
Put $E=\langle[\omega],A\rangle=\hbar^{\rm eff}(M,\omega)>0$. We use deformation parameters only in $H_{2n-2}(M;\C)$. Choose an integral basis $\xi_1,\ldots,\xi_s$ of $H_{2n-2}(M;\Z)/{\rm torsion}$ and write $\zeta=\sum_{r=1}^s x_r\xi_r$.

By the divisor axiom, the coefficient of $T^E[M]$ in
\[
P_\zeta=a_1*_\zeta\cdots *_\zeta a_l
\]
is
\[
\sum_{\substack{B\in H_2(M)\\
\langle[\omega],B\rangle=E}}
{\rm GW}_{0,l+1,B}^{\{1,\ldots,l+1\}}
(a_1,\ldots,a_l,[pt])\,
\exp\!\left(\sum_{r=1}^s x_r(\xi_r\circ B)\right).
\]
Only finitely many terms are nonzero, by Gromov compactness. The intersection pairings with the $\xi_r$ distinguish classes in $H_2(M)$, and the coefficient corresponding to $A$ is nonzero. Hence this exponential polynomial is not identically zero. Choose $\zeta$ for which it is nonzero.

The splitting axiom expresses every nonzero coefficient of $P_\zeta$ as a sum of products of three-point Gromov--Witten invariants, multiplied by divisor weights. Consequently, each curve class contributing to $P_\zeta$ belongs to $K^{\rm eff}(M,\omega)$. Its coefficient of $T^0$ is $a_1\cap\cdots\cap a_l=0$. Thus every nonzero term of $P_\zeta$ has symplectic area at least $E$. Writing
\[
P_\zeta=\lambda[M]+\beta,
\qquad
\beta\in H_{*<2n}(M;\Lambda_0),
\]
we therefore have
\[
\nu(\lambda)=-E,\qquad I_\nu(\beta)\leq-E.
\]

Suppose that $\varphi=\varphi_H^1$ has finitely many fixed points; otherwise the estimate is immediate. Put $\alpha_0=[M]$ and $\alpha_j=a_1*_\zeta\cdots *_\zeta a_j$. All these classes are nonzero. Corollary~\ref{cor:specineq} gives
\[
\rho_\zeta(\alpha_l,H)<\cdots<
\rho_\zeta(\alpha_0,H).
\]
If $\beta\ne0$, Theorem~\ref{thm:LjS} implies
\[
\rho_\zeta(\beta,H)
<\rho_\zeta([M],H)+I_\nu(\beta)
\leq\rho_\zeta([M],H)-E
=\rho_\zeta(\lambda[M],H).
\]
The valuation and quantum-shift properties consequently give
\[
\rho_\zeta(\alpha_l,H)
=\rho_\zeta([M],H)-E.
\]
The same equality is immediate when $\beta=0$. The action-spectrum counting argument in the proof of Theorem~\ref{thm:onept} now yields
\[
\#{\rm Fix}(\varphi)\geq
\left\lceil\frac{pl}{E}\right\rceil.
\]
Finally, the dimension axiom gives
\[
\sum_{i=1}^l(2n-\deg(a_i))
=2\langle c_1(TM),A\rangle.
\]
Hence
\[
l\geq
\frac{2\langle c_1(TM),A\rangle}
{\max_i\{2n-\deg(a_i)\}},
\]
which proves the stated estimate.\end{proof}

\begin{proof}[Proof of Theorem~\ref{thm:Arnoldspec}]  Given $\varphi\in\ham(M,\omega)$, take $H\in\rH$ such that $\varphi=\varphi^1_H$. Without loss of generality, we may assume that the number of elements in ${\rm Fix}(\varphi_H^1)$ is finite, otherwise, nothing needs to be proved.

The proof proceeds in two steps.

\noindent\textbf{Step 1.} By the definition of cuplength (cf. Section~\ref{sec:1}) and the Poincar\'{e} duality, there exist $u_i\in H_{<2n}(M,\C)$, $i=1\ldots,k$ such that $u_1\cap\cdots\cap u_k=[pt]$ and $\cl(M)=k+1$.  Take classes $\alpha_0,\dots,\alpha_k$ in $H_*(M;\Lambda)$ given by
\[
\alpha_0=[M]\quad\hbox{and}\quad \alpha_{i}=\alpha_{i-1} *_\zeta u_{i},\;i=1\ldots,k.
\]
Then by Theorem~\ref{thm:LjS}  we have
\[
\rho_\zeta(\alpha_i,H)<\rho_\zeta(\alpha_{i-1},H)
\]
for each $i\in\{1,\ldots,k\}$. Note that the big quantum product is a quantum deformation of the homological intersection product in $H_*(M;\C)$. In particular, we have
\[
\alpha_k=[pt]+\hbox{h.o.t.}
\]
where  h.o.t. stands for higher-order terms coming from deformations by $J$-holomorphic spheres. This implies that each $\alpha_i$ is nonzero, and that ${\textstyle\prod}(\alpha_k,[M])\neq 0$. Thus, it follows from Propositions~\ref{prop:property}(P8) that
\[
\rho_\zeta(\alpha_k,H)\geq -\rho_\zeta([M],\overline{H}).
\]
Consequently, we get
\begin{equation}\label{e:gabound}
\rho_\zeta(\alpha_0,H)-\rho_\zeta(\alpha_k,H)\leq\gamma_\zeta(H).
\end{equation}

\noindent\textbf{Step 2.}
The next goal is to prove that
\begin{equation}\label{e:fixpts}
\#{\rm Fix}(\varphi)\geq  \frac{ {\rm cuplength}(M;\C)}{\lfloor \gamma_\zeta(H)/p\big\rfloor+1}.
\end{equation}
According to Corollary~\ref{cor:specineq}, there exist $k+1$ elements $\widetilde{x}_i=(x_i,\overline{x}_i)\in\crit(\ah)$  such that
\begin{equation}\label{e:critvalue}
\rho_\zeta(\alpha_k,H)=\ah(\widetilde{x}_k)<\ldots<\rho_\zeta(\alpha_0,H)=\ah(\widetilde{x}_0).
\end{equation}
As in the second part of the proof of Theorem~\ref{thm:onept}, we let $\mathcal{E}=\{\ah(x_{i},\overline{x}_{i})\;|\;i=0,\ldots,k\}$, and consider the set $\widehat{\mathcal{E}}$ consisting of equivalence classes. It follows from (\ref{e:gabound}) that each equivalence class in $\widehat{\mathcal{E}}$ contains at most $$\bigg\lfloor\frac{\gamma_\zeta(H)}{p}\bigg\rfloor+1$$
elements in $\mathcal{E}$. So we get
\[
    \#{\rm Fix}(\varphi_H^1)\geq \#\widehat{\mathcal{E}}\geq \frac{k+1}{\big\lfloor\gamma_\zeta(H)/p\big\rfloor+1}.
\]
Take Hamiltonians generating $\varphi$ for which
$\gamma_\zeta(H)$ decreases to $\gamma_\zeta(\varphi)$.
The floor is right-continuous: for all sufficiently close terms,
$\lfloor\gamma_\zeta(H)/p\rfloor=\lfloor\gamma_\zeta(\varphi)/p\rfloor$.
It follows that
\[
    \#{\rm Fix}(\varphi_H^1)\geq \frac{k+1}{\big\lfloor\gamma_\zeta(\varphi)/p\big\rfloor+1}.
\]
\end{proof}

\begin{proof}[Proof of Theorem~\ref{thm:bcl}] The case $\bcl(g)=0$ is immediate. Suppose $\bcl(g)=k+1$ for some $g\in\Gamma_\omega$ and $\zeta\in\oplus_{i=0}^{n-1}H_{2i}(M;\Lambda_0)$.  Let $\varphi=\varphi_H^1$ for some $H\in\rH$, and assume $ \#{\rm Fix}(\varphi_H^1)$ is finite. Then there exist homology classes $\alpha_1,\ldots,\alpha_k\in H_{<2n}(M;\C)$ such that
$(\alpha_1*_\zeta\cdots *_\zeta\alpha_k)_{g}\neq 0$
.
Hence,
\[
\alpha_1*_\zeta\cdots *_\zeta\alpha_k=\sum_{h\in\Gamma_\omega} (\alpha_1*_\zeta\cdots *_\zeta\alpha_k)_h T^h\neq 0.
\]
Set $\beta=\alpha_1*_\zeta\cdots *_\zeta\alpha_k$. Then $I_\nu(\beta)\geq -g$.  By Proposition~\ref{prop:property} (P1), we obtain that
\begin{equation}\label{e:P1}
I_\nu(\beta)+\int^1_0\min H_t dt\leq \rho_\zeta(\beta,H)\quad\hbox{and}\quad \rho_\zeta([M],H)\leq \int^1_0\max H_t dt.
\end{equation}
According to Corollary~\ref{cor:specineq}, there exist $k+1$ critical points $(x_0,\overline{x}_0),\ldots,(x_{k},\overline{x}_{k})$ of $\ah$ such that
\[
    \rho_\zeta(\beta,H)=\ah(x_{k},\overline{x}_{k})<\ldots<\rho_\zeta([M],H)=\ah(x_{0},\overline{x}_{0}).
\]
This, together with~(\ref{e:P1}), gives
\begin{equation}\label{e:actionbd}
-g+\int^1_0\min H_t dt\leq \ah(x_{k},\overline{x}_{k})<\ldots <  \ah(x_{0},\overline{x}_{0})\leq \int^1_0\max H_t dt.
\end{equation}
As in the proof of Theorem~\ref{thm:onept}(2),  let $\mathcal{E}=\{\ah(x_{i},\overline{x}_{i})\;|\;i=0,\ldots,k\}$ and consider the set $\widehat{\mathcal{E}}$ consisting of equivalence classes under the relation $a\sim b$ iff $a-b\in\Gamma_\omega$. It follows from (\ref{e:actionbd}) that each equivalence class in $\widehat{\mathcal{E}}$ contains at most $$\bigg\lfloor\frac{\|H\|+g}{p}\bigg\rfloor+1$$
elements in $\mathcal{E}$. So we get
\[
    \#{\rm Fix}(\varphi_H^1)\geq \#\widehat{\mathcal{E}}\geq \bigg\lceil\frac{k+1}{\big\lfloor(\|H\|+g)/p\big\rfloor+1} \bigg\rceil\geq \bigg\lceil\frac{p(k+1)}{\|H\|+g+p}\bigg\rceil.
\]
For fixed $g$ and $\zeta$, this holds for every Hamiltonian generating
$\varphi$. Choose generators whose Hofer lengths decrease to
$\|\varphi\|$. Passing first to the real-valued inequality and then
using that $\#\operatorname{Fix}(\varphi)$ is an integer gives the
ceiling with $\|\varphi\|$ in place of $\|H\|$. Taking the two
suprema proves the theorem. \end{proof}

\begin{proof}[Proof of Theorem~\ref{thm:schtype}] According to Lemma~\ref{lem:mixnon0}, for  each integer $l> 2n$, there exists a class $\zeta_l\in\oplus_{i=0}^{n-1}H_{2i}(M;\C)$ such that $(\underbrace{a *_{\zeta_l}\cdots *_{\zeta_l}a}_l)_{g_l}\neq 0$ for $g_l=\langle[\omega],A_l\rangle$, and hence ${\rm qcl}_{\zeta_l}(g_l)\geq l+1$.
Moreover, for each such $l$, by (\ref{e:ngw}), we have the nonzero Gromov-Witten invariant
\begin{equation}\label{e:gwl}
    {\rm GW}_{0,l+1+K,A_l}^{\{1,\ldots,l+1\}}(\underbrace{a,\ldots,a}_l,b^l,\eta_1^l,\ldots,\eta_{K}^l)
    \end{equation}
for some homogeneous class $b^l\in H_*(M;\Q)$ of even degree, and homogeneous classes $\eta_1^l\ldots,\eta_K^l\in \oplus_{i=0}^{n-2}H_{2i}(M;\Q)$ with $0\leq K=K(l)\leq k(l)$ (corresponding to $K=0$ there are no inserted classes $\eta_i^l$). Then it follows from Theorem~\ref{thm:bcl} that for any $\varphi\in\ham(M,\omega)$,
\begin{equation}\label{e:l}
\#{\rm Fix}(\varphi)\geq
    \limsup_l\bigg\lceil \frac{p(l+1)}{p+g_l+\|\varphi\|} \bigg\rceil.
\end{equation}
We now estimate the right hand side of (\ref{e:l}). Since the invariant (\ref{e:gwl}) does not vanish, according to the dimension formula~(\ref{e:mixdim}), we obtain
\begin{equation}\label{e:ldeg}
 l(2n-\deg(a))+2n-\deg(b^l)+\sum_{j=1}^K(2n-\deg(\eta^l_j))=2n+2\langle c_1(TM),A_l\rangle+2K.
\end{equation}
Put $\delta=2n-\deg(a)>0$. Rearranging~\eqref{e:ldeg} gives
\[
 2\langle c_1(TM),A_l\rangle
 =l\delta-\deg(b^l)
   +\sum_{j=1}^{K(l)}(2n-2-\deg(\eta_j^l)).
\]
The degrees of the insertions lie between $0$ and $2n$, and $K(l)$
is bounded. Thus
\[
 \frac{\langle c_1(TM),A_l\rangle}{l}\longrightarrow\frac{\delta}{2}.
\]
Write $\omega|_{\pi_2(M)}=\lambda c_1(TM)|_{\pi_2(M)}$ with
$\lambda>0$. Then $g_l/l\to\lambda\delta/2$ and $p=\lambda N$.
Applying the real-valued inequality underlying~\eqref{e:l} and
taking the limit gives
\[
 \#\operatorname{Fix}(\varphi)
 \geq\lim_{l\to\infty}\frac{p(l+1)}{p+g_l+\|\varphi\|}
 =\frac{2N}{\delta}.
\]
This proves the theorem.\end{proof}

\begin{proof}[Proof of Theorem~\ref{thm:factorization}]
Without loss of generality, we may assume that $\varphi=\varphi_H^1$ for some $H\in\rH$ and that $\sharp {\rm Fix}(\varphi_H^1)$ is finite. Set $\alpha_0=[M]$ and $\alpha_i=\alpha_{i-1}*_\zeta u_i$ for $i\in \{1,\ldots,l\}$.
By Corollary~\ref{cor:specineq}, there exist $l+1$ critical points $(x_0,\overline{x}_0),\ldots,(x_{l},\overline{x}_{l})$ of $\ah$ such that
\begin{equation}\label{e:fact}
    \rho_\zeta(\alpha_l,H)=\ah(x_{l},\overline{x}_{l})<\ldots<\rho_\zeta([M],H)=\ah(x_{0},\overline{x}_{0}).
\end{equation}
According to our assumption, $\alpha_l=zT^{gp}[M]+\beta$ with $\beta\in H_{<2n}(M;\Lambda)$ and $I_\nu(\beta)\leq -gp$. If $\beta\ne0$, Theorem~\ref{thm:LjS} gives
\[
\rho_\zeta(\beta,H)=\rho_\zeta([M]*_\zeta\beta,H)<\rho_\zeta([M],H)+I_\nu(\beta).
\]
Hence, using Proposition~\ref{prop:property} (P4) and (P5), we see that
\[
\rho_\zeta(\beta,H)<\rho_\zeta([M],H)-gp=  \rho_\zeta(zT^{gp}[M],H).
\]
The same comparison is immediate when $\beta=0$.
This, together with Proposition~\ref{prop:property} (P7), implies that
\[
 \rho_\zeta(\alpha_l,H)= \rho_\zeta(zT^{gp}[M],H)=\rho_\zeta([M],H)-gp.
\]
Consequently,
\begin{equation}\label{e:factbd}
\rho_\zeta([M],H)- \rho_\zeta(\alpha_l,H)= gp.
\end{equation}
Now we let $\mathcal{E}=\{\ah(x_{i},\overline{x}_{i})\;|\;i=0,\ldots,l\}$, and consider the set $\widehat{\mathcal{E}}$ consisting of equivalence classes as in the proof of Theorem~\ref{thm:onept}. Then by (\ref{e:fact}) and (\ref{e:factbd}) there is at most one equivalence class containing $g+1$ elements  in $\mathcal{E}$, and each remaining class contains at most $g$ elements. So we obtain
\[
\#{\rm Fix}(\varphi^1_H)\geq 1+\bigg\lceil\frac{l+1-g-1}{g}\bigg\rceil=\bigg\lceil\frac{l}{g}\bigg\rceil.
\]
The proof is complete.\end{proof}

\section{Generalized flag manifolds}\label{sec:flags}

In this section we prove Theorem~\ref{thm:flags} and give explicit applications to coadjoint orbits. Throughout, let $X=G/P$ be a positive-dimensional generalized flag manifold, where $G$ is a connected complex semisimple group and $P$ is parabolic, and put $D=\dim_{\mathbb C}X$. Replacing $G$ by its simply connected cover and $P$ by its inverse image leaves $G/P$ unchanged, so we may assume that $G$ is simply connected. Let $\omega$ be a K\"ahler form with $\omega(\pi_2(X))=p\mathbb Z$, $p>0$.

We use cohomological notation in the geometric arguments: $P_X\in H^{2D}(X;\mathbb Q)$ denotes the point class and $1\in H^0(X;\mathbb Q)$ the unit. Under Poincar\'e duality these correspond to $[pt]$ and $[X]$, respectively, in quantum homology. In particular, positive cohomological degree corresponds to homological degree strictly less than $2D$. Write $*$ for the undeformed small quantum product.

\subsection{Invariants with two point insertions and positivity}

Let $C_1,\ldots,C_\rho$ be the primitive Schubert curve classes. They form an integral basis of $H_2(X;\mathbb Z)$, and every effective integral curve class is a nonnegative integral combination of them. We write $B\leq A$ if $A-B$ is such a combination. Since $\omega$ is K\"ahler, $B\leq A$ implies $\omega(B)\leq\omega(A)$. Before specializing the Novikov variables, write $q^A$ for the monomial associated with an effective curve class $A$. The symplectic quantum product is obtained by the specialization
\begin{equation}\label{eq:flag-area-specialization}
q^A\longmapsto T^{\omega(A)}.
\end{equation}
The integrable complex structure of $X$ is compatible with every K\"ahler form under consideration. Thus the genus-zero invariants used below are the invariants of the homogeneous variety, with Novikov exponents weighted by $\omega$.

We recall two facts about quantum Schubert calculus. First, all structure constants of the small quantum product in the Schubert basis are nonnegative integers. Second, the theorem of Fulton--Woodward relates a connected curve joining two opposite Schubert varieties to a nonzero three-point invariant in a degree no larger than that of the curve. Both statements hold for arbitrary $G/P$; see~\cite[Sections~3, 6 and~9]{FWo}.

\begin{lem}\label{lem:flag-threepoint-area}
The invariant $\Theta(X,\omega)$ in Conjecture~\ref{thm:mainthm} is finite and satisfies
\[
\Theta(X,\omega)=
\min\big\{\omega(A)\mid
\langle P_X,b,P_X\rangle_{0,3,A}>0,
\quad b\text{ a positive-degree Schubert class}\big\}.
\]
Here the brackets denote cohomological Gromov--Witten invariants.
\end{lem}

\begin{proof}
Two opposite point Schubert varieties can be joined by a chain of $T$-invariant rational curves, where $T$ is a maximal torus of $G$. Such a chain can be obtained from successive simple reflections in a reduced expression of the longest Weyl group element, after omitting constant steps in $G/P$. Theorem~9.1 of~\cite{FWo} supplies a nonzero three-point invariant with two point insertions. In particular, the defining set for $\Theta$ is nonempty. Its areas belong to $p\mathbb N$, so its infimum is a minimum.

Choose a class $A$ realizing this minimum, allowing any number of additional primary insertions in the definition of $\Theta$. A nonzero invariant with two point insertions implies the existence of a genus-zero stable map in class $A$ through two general points of $X$. The diagonal action of $G$ on $X\times X$ has an open orbit consisting of pairs in opposite position. Translating the map by the connected group $G$, we may take these points to be the opposite point Schubert varieties without changing the curve class. The reduced image is a connected curve whose class is at most $A$ in the effective order.

The equivalent conditions (8) and (10) in the proof of \cite[Theorem~9.1]{FWo} give an effective class $B\leq A$ and a Schubert class $b$ such that
\[
\langle P_X,b,P_X\rangle_{0,3,B}>0.
\]
The class $B$ is nonzero because $D>0$ and the degree-zero invariant with two point insertions vanishes. The fundamental-class axiom excludes $b=1$, so $b$ has positive degree. Thus $\omega(B)\leq\omega(A)=\Theta$. Conversely, this three-point invariant is allowed in the definition of $\Theta$, so $\omega(B)\geq\Theta$. This proves the claim.
\end{proof}

An element of the multivariable small quantum cohomology is called \emph{Schubert nonnegative} if all coefficients in the basis $\{q^AC_w\}$ are nonnegative real numbers. Quantum Schubert positivity implies that such elements are closed under addition and multiplication. Specialization~\eqref{eq:flag-area-specialization} preserves nonnegativity: coefficients of classes with the same symplectic area add and cannot cancel.

\begin{lem}\label{lem:flag-positive-powers}
Suppose $a_1,\ldots,a_l\in H^{>0}(X;\mathbb Q)$ are Schubert nonnegative and the coefficient of $q^A1$ in $U=a_1*\cdots*a_l$ is positive, where $\omega(A)>0$. Then every Hamiltonian diffeomorphism of $(X,\omega)$ satisfies
\begin{equation}\label{eq:flag-positive-power-bound}
\#\operatorname{Fix}(\varphi)
\geq\left\lceil\frac{pl}{\omega(A)}\right\rceil.
\end{equation}
Here the area $\omega(A)$ need not be the smallest area occurring in $U$.
\end{lem}

\begin{proof}
Let $c>0$ be the coefficient of $q^A1$ in $U$. Write $U=cq^A1+R$, with $R$ Schubert nonnegative. Since $1$ is the quantum unit, the coefficient of $q^{mA}1$ in $U^{*m}$ is at least $c^m>0$ for every $m\geq1$. After specialization, the coefficient of $T^{m\omega(A)}1$ is still positive.

Under Poincar\'e duality, $U^{*m}$ is a product of $ml$ ordinary homology classes of degree less than $2D$. Definition~\ref{df:bcl} gives
\[
{\rm qcl}_0\big(m\omega(A)\big)\geq ml+1.
\]
For a fixed $\varphi\in\operatorname{Ham}(X,\omega)$, Theorem~\ref{thm:bcl}, with deformation parameter zero, therefore implies
\[
\#\operatorname{Fix}(\varphi)
\geq\frac{p(ml+1)}{p+m\omega(A)+\|\varphi\|}.
\]
If the fixed-point set is infinite there is nothing to prove. Otherwise, let $m\to\infty$ and use the integrality of its cardinality to obtain~\eqref{eq:flag-positive-power-bound}.
\end{proof}

\subsection{Chern numbers and Schubert lines}

\begin{prop}\label{prop:flag-chern-bound}
Let $A\ne0$ be an effective integral class such that
$\langle P_X,b,c\rangle_{0,3,A}\ne0$ for some
$b,c\in H^*(X;\mathbb Q)$. Then
\begin{equation}\label{eq:flag-chern-bound}
\#\operatorname{Fix}(\varphi)\geq
\left\lceil
\frac{p\langle c_1(TX),A\rangle}{\langle [\omega],A\rangle}
\right\rceil
\qquad
\text{for every }\varphi\in\operatorname{Ham}(X,\omega).
\end{equation}
\end{prop}

\begin{proof}
Expanding $b$ and $c$ in the Schubert basis, we can choose Schubert classes $\sigma,\tau$ such that
\[
N_A:=\langle P_X,\sigma,\tau\rangle_{0,3,A}>0.
\]
Indeed, at least one of the corresponding Schubert invariants is nonzero, and all are nonnegative. Let $r$ and $s$ be the complex codimensions of $\sigma$ and $\tau$. The fundamental-class axiom excludes a unit insertion for $A\ne0$, so $r,s\geq1$. The dimension constraint gives
\begin{equation}\label{eq:flag-chern-dimension}
\langle c_1(TX),A\rangle=r+s.
\end{equation}

Choose an integral ample divisor class $h$ that is a positive linear combination of Schubert divisor classes. If $\sigma^\vee$ and $\tau^\vee$ are the Schubert classes dual to $\sigma$ and $\tau$ under the intersection pairing, the corresponding coefficients in $h^r$ and $h^s$ are
\[
a_\sigma=\int_X h^r\smile\sigma^\vee>0,
\qquad
a_\tau=\int_X h^s\smile\tau^\vee>0.
\]
These are the top self-intersection numbers of an ample class on nonzero effective Schubert cycles of dimensions $r$ and $s$.

The terms of curve class zero in $h^{*r}$ and $h^{*s}$ are the ordinary cup powers $h^r$ and $h^s$. Quantum Schubert positivity therefore gives
\[
h^{*r}=a_\sigma\sigma+R_r,
\qquad
h^{*s}=a_\tau\tau+R_s,
\]
with $R_r,R_s$ Schubert nonnegative. The coefficient of $q^A1$ in $\sigma*\tau$ is $N_A$. Consequently, its coefficient in $h^{*(r+s)}$ is at least $a_\sigma a_\tau N_A>0$. Applying Lemma~\ref{lem:flag-positive-powers} to the $r+s$ ordinary divisor factors, and using~\eqref{eq:flag-chern-dimension}, proves the estimate.
\end{proof}

The next observation ensures that the preceding proposition applies to every primitive Schubert curve class. It also removes any restriction on the number of primary insertions in the subsequent application to effective classes.

\begin{lem}\label{lem:flag-schubert-lines}
Let $h$ be the sum of the Schubert divisor classes dual to $C_1,\ldots,C_\rho$, and set $c_i=\langle c_1(TX),C_i\rangle$. Then $c_i\geq2$ and
\begin{equation}\label{eq:flag-line-invariant}
\langle P_X,h^{c_i-1},h\rangle_{0,3,C_i}>0
\qquad (1\leq i\leq\rho),
\end{equation}
where the powers in the insertions are ordinary cup powers.
\end{lem}

\begin{proof}
The class $h$ is the first Chern class of a very ample line bundle $L$, and $h(C_i)=1$ for every $i$. Thus the Schubert curves are projective lines in the embedding defined by $L$; see~\cite[Section~3, especially p.~538]{Pan}. Every stable map of class $C_i$ has exactly one nonconstant component, mapped isomorphically onto a projective line, because its total $L$-degree is one. All other components, if any, are contracted.

The tangent bundle of $X$ is globally generated. For an embedded line $f:\mathbb P^1\to X$, the bundle $f^*TX$ therefore splits into line bundles of nonnegative degrees. Its tangent subbundle is $T\mathbb P^1=\mathcal O(2)$, so $c_i=\deg(f^*TX)\geq2$.

Fix a point $x\in X$ and consider
\[
Z_i=\operatorname{ev}_1^{-1}(x)
\subset\overline{\mathcal M}_{0,2}(X,C_i).
\]
The target $X$ is convex. Hence the stable-map space is a smooth proper Deligne--Mumford stack of pure dimension $D+c_i-1$, its virtual fundamental class is its ordinary fundamental class, and its boundary has codimension one; see~\cite[Theorems~2 and~3 and Section~7]{FP}. In this degree-one, two-marked case all automorphisms are trivial: a contracted component, if present, must carry both marks and its attaching node. Thus the moduli space is a smooth projective variety. The evaluation map $\operatorname{ev}_1$ is smooth: the infinitesimal action of $G$ maps onto $T_xX$ at every point of its fiber. The same argument applies to each boundary stratum. It follows that $Z_i$ is nonempty, projective and of pure dimension $c_i-1$, with a dense open locus of irreducible domains. Nonemptiness follows by translating a Schubert line through $x$.

On this open locus, $\operatorname{ev}_2:Z_i\to X$ is injective: two distinct image points $x$ and $y$ determine a unique projective line, and the two marked points then determine the stable map up to isomorphism. Thus the pushforward $(\operatorname{ev}_2)_*[Z_i]$ is a nonzero effective cycle of dimension $c_i-1$. Since $h$ is ample, the projection formula gives
\[
\langle P_X,h^{c_i-1}\rangle_{0,2,C_i}
=\int_{Z_i}\operatorname{ev}_2^*h^{c_i-1}>0.
\]
Finally, the divisor axiom and $h(C_i)=1$ identify this number with
$\langle P_X,h^{c_i-1},h\rangle_{0,3,C_i}$.
\end{proof}

\begin{proof}[Proof of Theorem~\ref{thm:flags}]
Put $a_i=\omega(C_i)>0$ and $c_i=\langle c_1(TX),C_i\rangle$.
Proposition~\ref{prop:flag-chern-bound} and
Lemma~\ref{lem:flag-schubert-lines} give
\[
\#\operatorname{Fix}(\varphi)
\geq
B_{\max}(X,\omega):=
\left\lceil p\max_{1\leq i\leq\rho}\frac{c_i}{a_i}\right\rceil.
\]
This is~\eqref{eq:flag-uniform-bound}.
For any nonzero effective integral class $A=\sum_i d_iC_i$, with $d_i\geq0$, we have
\begin{equation}\label{eq:flag-ratio-average}
\frac{\langle c_1(TX),A\rangle}{\omega(A)}
=\sum_i\frac{d_i a_i}{\omega(A)}\frac{c_i}{a_i}
\leq\max_i\frac{c_i}{a_i}.
\end{equation}
This proves~\eqref{eq:flag-effective-bound} for every such $A$.
In particular, it applies to any class supporting a nonzero primary genus-zero Gromov--Witten invariant, irrespective of the number of insertions: a nonzero invariant for the integrable complex structure requires an effective class.

For the cup-length estimate choose $A_0$ of area $\Theta(X,\omega)$ using Lemma~\ref{lem:flag-threepoint-area}. If the third insertion in the resulting three-point invariant has complex codimension $s$, the dimension constraint gives
\[
\langle c_1(TX),A_0\rangle=D+s\geq D+1.
\]
Moreover, $X$ has no odd cohomology and an ample divisor class has nonzero $D$th cup power. With the convention of this paper,
\begin{equation}\label{eq:flag-cuplength}
\operatorname{cuplength}(X;\mathbb C)=D+1.
\end{equation}
Therefore
\[
\#\operatorname{Fix}(\varphi)
\geq\left\lceil
\frac{p\langle c_1(TX),A_0\rangle}{\Theta(X,\omega)}
\right\rceil
\geq\left\lceil\frac{p(D+1)}{\Theta(X,\omega)}\right\rceil.
\]
This proves Conjecture~\ref{thm:mainthm} for $X$, as well as the asserted effective-class version of the Chern-number bound. Finally, if $\omega=\lambda c_1(TX)$ on $\pi_2(X)$, then $p=\lambda N$ and $p c_i/a_i=N$ for every $i$, which proves the monotone assertion.
\end{proof}

\begin{rmk}\label{rmk:flag-undeformed-method}
The proof uses the undeformed quantum product, quantum Schubert positivity and the asymptotic quantum-cuplength estimate. Its input for each primitive Schubert class is the ordinary three-point invariant~\eqref{eq:flag-line-invariant}. The extension to other effective classes follows from the weighted average~\eqref{eq:flag-ratio-average}; it does not identify higher-point Gromov--Witten invariants with coefficients of an iterated small quantum product. The bound $B_{\max}$ is the maximum of the Chern-number bounds over all nonzero effective integral classes, since each $C_i$ is itself effective.
\end{rmk}

\begin{rmk}\label{rmk:flag-bound-comparison}
Write
\[
B_\Theta=\left\lceil\frac{p(D+1)}{\Theta(X,\omega)}\right\rceil,
\qquad
B(A)=\left\lceil\frac{p\langle c_1(TX),A\rangle}{\omega(A)}\right\rceil.
\]
For a class $A_0$ supplied by Lemma~\ref{lem:flag-threepoint-area},
\begin{equation}\label{eq:flag-bound-comparison}
B_{\max}\geq B(A_0)\geq B_\Theta.
\end{equation}
These inequalities need not be strict, even when $\omega$ is nonmonotone. In particular,
\[
B(A_0)>B_\Theta
\quad\Longleftrightarrow\quad
\frac{p\langle c_1(TX),A_0\rangle}{\Theta(X,\omega)}
>
\left\lceil\frac{p(D+1)}{\Theta(X,\omega)}\right\rceil.
\]
Thus a larger numerator before rounding need not improve the integer bound. The comparison uses a class realizing $\Theta$; it does not assert $B(A)\geq B_\Theta$ for every other class $A$. We retain the $\Theta$-estimate to record Conjecture~\ref{thm:mainthm} and examples in which its integer bound agrees with the stronger theorem.
\end{rmk}

\begin{rmk}\label{rmk:flag-monotone-comparison}
Suppose $\omega=\lambda c_1(TX)$ on $\pi_2(X)$, with $\lambda>0$, and let $N$ be the minimal Chern number. Then $p=\lambda N$ and Theorem~\ref{thm:flags} gives
\[
B_{\max}=N,
\qquad
\#\operatorname{Fix}(\varphi)\geq N.
\]
This recovers Schwarz's estimate from a quantum nonnilpotent divisor~\cite[Corollary~1.3]{Sc2}, recalled in Corollary~\ref{cor:nonnilp}. Indeed, the proof of Proposition~\ref{prop:flag-chern-bound} exhibits a positive $q^{C_i}1$ coefficient in $h^{*c_i}$. Positivity implies $h^{*mc_i}\ne0$ for every $m\geq1$, so $h$ is quantum nonnilpotent. Its Poincar\'e dual has homological degree $2D-2$, and Corollary~\ref{cor:nonnilp} also gives the lower bound $N$.
\end{rmk}

\begin{ex}\label{ex:flag-classwise-comparison}
A Chern-number bound for an individual nonminimal class need not dominate the $\Theta$ estimate. Let $X=\mathbb P^1\times\mathbb P^3$, let $F_1,F_2$ be the line classes in its factors, and choose a product K\"ahler form with $\omega(F_1)=3$ and $\omega(F_2)=1$. Write $h_1,h_2$ for the
hyperplane classes pulled back from the factors and $P_X$ for the point class in cohomology. Then $p=1$, $d=4$, and $\Theta=4$. Indeed, a connected curve through two general points has positive degree in both factors, while degree $F_1+F_2$ supports the nonzero three-point invariant $\langle P_X,P_X,h_1h_2\rangle_{F_1+F_2}=1$.

For $A=3F_1+F_2$ one has
\[
 \big\langle P_X,P_X,
       \underbrace{h_1h_2,\ldots,h_1h_2}_{5\text{ insertions}}
 \big\rangle_A=1.
\]
To see this, the projection to $\mathbb P^3$ is the line through
the two point constraints. Each of the other five constraints
specifies a value in $\mathbb P^1$ at the intersection of that
line with a general hyperplane. Thus a curve satisfying the
constraints is the graph of a degree-three rational function with
seven prescribed general values. These are seven independent
linear conditions on a pair of cubic homogeneous polynomials up
to common scale, and their unique solution is coprime and has
degree three. Reducible maps do not contribute: if the component
projecting onto the line has degree $e<3$ in $\mathbb P^1$, at
most $3-e$ constraints can lie on vertical components, whereas a
degree-$e$ rational function satisfies at most $2e+1$ general
value conditions. These inequalities cannot account for all seven
constraints. Homogeneity gives the usual enumerative
interpretation of the invariant. Nevertheless,
\[
 \left\lceil\frac{p\langle c_1(TX),A\rangle}{\omega(A)}\right\rceil
 =\left\lceil\frac{10}{10}\right\rceil=1
 <2=\left\lceil\frac{p(d+1)}{\Theta}\right\rceil.
\]
There is no conflict with Theorem~\ref{thm:flags}: optimizing over
the basic classes gives
$\lceil\max\{2/3,4\}\rceil=4$, which dominates both estimates.
Thus comparisons with the original conjecture must distinguish a
specified class from an optimal choice of class.
\end{ex}

\subsection{Explicit estimates in type A}

\subsubsection{\texorpdfstring{Coadjoint orbits of $U(n)$}{Coadjoint orbits of U(n)}}
Let $\lambda_1\geq\cdots\geq\lambda_n$ be real numbers which are not
all equal. Write their distinct values as
$\mu_1>\cdots>\mu_s$, with multiplicities $k_1,\ldots,k_s$, and set
$n_i=k_1+\cdots+k_i$. The coadjoint orbit through
$\sqrt{-1}\operatorname{diag}(\lambda_1,\ldots,\lambda_n)$ is
\[
 X=U(n)/(U(k_1)\times\cdots\times U(k_s))
   =F(n_1,\ldots,n_{s-1};n).
\]
Normalize its Kostant--Kirillov--Souriau form $\omega_\lambda$ by
\[
 \omega_\lambda(C_i)=\mu_i-\mu_{i+1}=:a_i>0,
 \qquad 1\leq i<s,
\]
where $C_i$ is the Schubert sphere at the boundary between the $i$th
and $(i+1)$st blocks. These spheres form an integral basis of
$H_2(X;\mathbb Z)$ and generate its effective cone. Thus the form is
rational precisely when its adjacent gaps are commensurable, and
its minimal symplectic period is determined by
\begin{equation}\label{e:flag-period}
 p\mathbb Z=\sum_{i=1}^{s-1}a_i\mathbb Z.
\end{equation}
In particular, $p$ is the positive generator of the group of gaps,
which need not be the smallest gap. The dimension and basic Chern
numbers are
\begin{equation}\label{e:flag-dimension}
 d=\dim_{\mathbb C}X=\sum_{1\leq i<j\leq s}k_i k_j
   =\frac12\left(n^2-\sum_{i=1}^s k_i^2\right),
\end{equation}
\begin{equation}\label{e:flag-chern}
 c_1(TX)[C_i]=k_i+k_{i+1}=:c_i.
\end{equation}
The summands in \eqref{e:flag-dimension} are the dimensions of
$\operatorname{Hom}(\mathbb C^{k_i},\mathbb C^{k_j})$ for all $i<j$,
not just adjacent pairs.

Applying Theorem~\ref{thm:flags} to the basic sphere classes gives
the following bound for every
$\varphi\in\operatorname{Ham}(X,\omega_\lambda)$:
\begin{equation}\label{e:U}
 \#\operatorname{Fix}(\varphi)\geq B_{\max}:=
 \left\lceil
 p\max_{1\leq i<s}\frac{k_i+k_{i+1}}{\mu_i-\mu_{i+1}}
 \right\rceil.
\end{equation}
This is the largest bound furnished by the Chern-number estimate
as the positive effective class varies. Indeed, if
$A=\sum_i m_iC_i$, with $m_i\geq0$ and $A\ne0$, then
\[
 \frac{\langle c_1(TX),A\rangle}{\omega_\lambda(A)}
 =\sum_i\frac{m_i a_i}{\sum_jm_j a_j}\frac{c_i}{a_i}.
\]
This weighted-average observation concerns the estimates obtained
from the theorem; it does not assert that \eqref{e:U} is an optimal
fixed-point bound.

For comparison with Conjecture~\ref{thm:mainthm}, the type-$A$
cascade formula in Appendix~\ref{sec:coadjoint} gives
\begin{equation}\label{e:flag-theta}
 \Theta(X,\omega_\lambda)
 =\sum_{j=1}^{\lfloor n/2\rfloor}(\lambda_j-\lambda_{n+1-j})
 =\sum_{i=1}^{s-1}w_i a_i,
 \qquad w_i:=\min\{n_i,n-n_i\}.
\end{equation}
The coefficient of $\lambda_a-\lambda_{a+1}$ in the first sum is
$\min\{a,n-a\}$, which explains the second expression, including
repeated eigenvalues. The first expression also occurs in the
Hofer--Zehnder capacity calculation of Caviedes Castro~\cite{CC}.
Its identification here with $\Theta$ uses the minimal degree in
the quantum product of two point classes. That degree is
$A_0=\sum_iw_iC_i$ and supports a nonzero three-point invariant with
two point insertions. Consequently the class $A_0$ gives
\begin{equation}\label{e:U-chern}
 B(A_0):=
 \left\lceil
 \frac{p\sum_{i=1}^{s-1}w_i(k_i+k_{i+1})}
 {\sum_{i=1}^{s-1}w_i(\mu_i-\mu_{i+1})}
 \right\rceil.
\end{equation}
Since $H^{\mathrm{odd}}(X;\mathbb C)=0$ and a K\"ahler class has
nonzero $d$th power, our cup-length convention gives
$\operatorname{cuplength}(X;\mathbb C)=d+1$. Thus the original estimate from Conjecture~\ref{thm:mainthm} is
\begin{equation}\label{e:U-theta}
 B_\Theta:=
 \left\lceil\frac{p(d+1)}{\Theta(X,\omega_\lambda)}\right\rceil,
 \qquad
 B_{\max}\geq B(A_0)\geq B_\Theta.
\end{equation}
The last inequality follows from
$\langle c_1(TX),A_0\rangle\geq d+1$.
Neither inequality in \eqref{e:U-theta} need be strict, even for nonmonotone forms.The following examples distinguish strict improvements from cases where the original estimate agrees with the stronger theorem.

\begin{enumerate}
\item Let $X=F(1,5;6)$ and
$\lambda=(3,2,2,2,2,0)$. The blocks are $(1,4,1)$ and
\[
 (a_1,a_2)=(1,2),\qquad (c_1,c_2)=(5,5),\qquad
 p=1,\qquad d=9.
\]
The form is nonmonotone. Here $A_0=C_1+C_2$, so
$\Theta=3$ and $\langle c_1(TX),A_0\rangle=10=d+1$.
The original estimate and the Chern estimate for this class realizing $\Theta$ both give
\[
 B_\Theta=B(A_0)=\left\lceil\frac{10}{3}\right\rceil=4.
\]
The basic class $C_1$, however, gives the stronger conclusion
\[
 \#\operatorname{Fix}(\varphi)\geq B_{\max}=5.
\]
Thus restricting the strengthened theorem to a class realizing $\Theta$ can miss its best bound.

\item Let $X=F(2,4;5)$ and $\lambda=(3,3,1,1,0)$. Then
\[
 (k_1,k_2,k_3)=(2,2,1),\qquad
 (a_1,a_2)=(2,1),\qquad (c_1,c_2)=(4,3),\qquad
 p=1,\qquad d=8.
\]
Here $A_0=2C_1+C_2$, $\Theta=5$, and
$\langle c_1(TX),A_0\rangle=11$. Hence
\[
 B_\Theta=\left\lceil\frac95\right\rceil=2,
 \qquad
 B(A_0)=\left\lceil\frac{11}5\right\rceil=3,
 \qquad B_{\max}=3.
\]
Thus the Chern estimate already improves the original bound in
the class realizing $\Theta$, and
$\#\operatorname{Fix}(\varphi)\geq3$.

\item Let $X=U(3)/U(1)^3$ and $\lambda=(3,2,0)$. In this case
\[
 (a_1,a_2)=(1,2),\qquad (c_1,c_2)=(2,2),\qquad
 p=1,\qquad d=3,\qquad \Theta=3.
\]
The form is again nonmonotone, but all three estimates agree:
\[
 B_{\max}=B(A_0)=B_\Theta=2.
\]
This gives a nonmonotone example where optimizing the Chern-number
estimate does not strictly improve the original $\Theta$ estimate.
\end{enumerate}

\subsubsection{Monotone partial flag manifolds}
If $\omega=\tau c_1(TX)$ on $\pi_2(X)$, with $\tau>0$, then
$a_i=\tau(k_i+k_{i+1})$ and
\[
 N=\gcd(k_1+k_2,\ldots,k_{s-1}+k_s),\qquad p=\tau N.
\]
Consequently every ratio in \eqref{e:U} equals $N$ after
multiplication by $p$, and Theorem~\ref{thm:flags} gives
\begin{equation}\label{e:monotone-flag-bound}
 \#\operatorname{Fix}(\varphi)\geq N.
\end{equation}
For the normalization $[\omega]=c_1(TX)$, one may equivalently take
$\mu_i=n-n_{i-1}-n_i+c$, where $n_0=0$, $n_s=n$, and $c$ is
arbitrary; see~\cite{NNU}. A common positive rescaling of the form
does not change the bound.

The five families considered here therefore have the following
estimates. The table records the bounds from
\eqref{e:monotone-flag-bound}, rather than the potentially smaller
values obtained by evaluating $B_\Theta$.
\begin{center}
\small
\renewcommand{\arraystretch}{1.2}
\begin{tabular}{lccc}
\toprule
Manifold & Blocks & $\dim_{\mathbb C}X$ & Lower bound $N$\\
\midrule
$\mathbb{CP}^{n-1}$ & $(1,n-1)$ & $n-1$ & $n$\\
$\operatorname{Gr}_{\mathbb C}(k,n)$ & $(k,n-k)$ & $k(n-k)$ & $n$\\
$F(1,n-1;n)$ & $(1,n-2,1)$ & $2n-3$ & $n-1$\\
$U(mk)/U(k)^m$ & $(k,\ldots,k)$ & $k^2m(m-1)/2$ & $2k$\\
$U(n)/U(1)^n$ & $(1,\ldots,1)$ & $n(n-1)/2$ & $2$\\
\bottomrule
\end{tabular}
\end{center}
Here $n\geq2$, $1\leq k<n$ in the Grassmannian row,
$n\geq3$ in the incidence row, and $m\geq2$, $k\geq1$ in the
equal-block row. All five bounds agree with Schwarz's estimate
from a quantum nonnilpotent divisor~\cite[Corollary~1.3]{Sc2}, as
explained in Remark~\ref{rmk:flag-monotone-comparison}. They are
lower bounds for every Hamiltonian diffeomorphism; no assertion
that they are attained is made.

\section{Negatively monotone manifolds}\label{sec:negative-monotone}

Assume throughout this section that $M$ is closed and connected, with $\dim M=2n$, and that
\begin{equation}\label{eq:negative-monotonicity}
 [\omega]|_{\pi_2(M)}=-\kappa\,c_1(TM)|_{\pi_2(M)},\qquad
 \kappa>0,\qquad c_1(TM)(\pi_2(M))=N\mathbb Z,\quad N\geq1.
\end{equation}
Then $M$ is rational with minimal symplectic period $p=\kappa N$. We use the undeformed quantum product $*=*_0$ and spectral invariant $\rho_0$. In contrast to the flag-manifold argument, no nonzero Gromov--Witten invariant of positive area is assumed.

\subsection{Grading and spectral carriers}

Since $\ker\omega=\ker c_1$ on $\pi_2(M)$, the equivalence of cappings used in Section~\ref{Sec:pre} preserves the integer grading. We use homological Floer degrees compatible with the PSS isomorphism: an ordinary homology class of degree $j$ maps to Floer degree $j$. Thus the degree of a nondegenerate capped orbit is $|\bar x|=n-\mu_H(\bar x)$ in the Conley--Zehnder convention of this paper. With $\bar x\# A$ denoting recapping by $A\in\pi_2(M)$, our conventions give
\begin{equation}\label{eq:negative-recapping}
 \mathcal A_H(\bar x\# A)=\mathcal A_H(\bar x)-\omega(A),\qquad
 |\bar x\# A|=|\bar x|-2\langle c_1(TM),A\rangle.
\end{equation}
Accordingly, set $\deg T^{mp}=2mN$ for $m\in\mathbb Z$. The small quantum product has degree $-2n$ in this grading. For an isolated capped orbit, write $\mathrm{HF}^{\mathrm{loc}}_j(H,\bar x;\mathbb C)$ for local Floer homology in the same homological grading. All local statements below are invariant under an overall grading shift.

\begin{lem}\label{lem:negative-graded-carrier}
Suppose that $H$ has finitely many contractible one-periodic orbits.
If $a\ne0$ is a homogeneous quantum homology class of degree $j$,
there is a capped orbit $\bar x$ such that
\[
 \mathcal A_H(\bar x)=\rho_0(a,H),\qquad
 \mathrm{HF}^{\mathrm{loc}}_j(H,\bar x;\mathbb C)\ne0.
\]
\end{lem}

\begin{proof}
This is the graded spectral-carrier property; its formulation in the
weakly monotone setting is~\cite[Section~6.2.1, (AS10)]{GG}. We explain
the filtered argument in the virtual Floer framework used in this
paper, so that no semipositivity assumption is needed here.

Put $c=\rho_0(a,H)$. Rationality and finiteness of the orbit set imply
that the action spectrum is discrete and that only finitely many
capped orbits have action $c$. Choose $\epsilon>0$ so that $c$ is
the only spectral value in $[c-\epsilon,c+\epsilon]$. Decrease
$\epsilon$ further so that the window width is smaller than both the
sphere-bubbling energy and the energy required for a Floer trajectory
to leave the chosen disjoint isolating neighborhoods of these orbits.
For sufficiently small nondegenerate perturbations, the Floer complex
in this action window is the direct sum of their local complexes.
Indeed, trajectories contributing to the window have arbitrarily
small energy as the perturbation and window shrink; they cannot
bubble or connect distinct isolating neighborhoods. The virtual
perturbations used for the global theory may be chosen compatibly
with the regular local problems, where no sphere bubbling occurs.
Consequently,
\[
 HF_j^{(c-\epsilon,c+\epsilon)}(H)
 \cong\bigoplus_{\mathcal A_H(\bar x)=c}
             \mathrm{HF}^{\mathrm{loc}}_j(H,\bar x;\mathbb C).
\]
Here the left-hand side is defined by continuation from small
nondegenerate perturbations.

The PSS image of $a$ lifts to $HF_j^{(-\infty,c+\epsilon)}(H)$ but
not to $HF_j^{(-\infty,c-\epsilon)}(H)$. By the long exact sequence
for the two filtration levels, any such lift has nonzero image in
$HF_j^{(c-\epsilon,c+\epsilon)}(H)$. At least one summand in the
displayed decomposition is therefore nonzero, as required.
\end{proof}

\begin{lem}\label{lem:negative-local-width}
For an isolated capped one-periodic orbit $\bar x$, let
\[
 S(\bar x)=\{j\in\mathbb Z:
       \mathrm{HF}^{\mathrm{loc}}_j(H,\bar x;\mathbb C)\ne0\}.
\]
If $S(\bar x)$ is nonempty, then
\[
 \max S(\bar x)-\min S(\bar x)<2n.
\]
\end{lem}

\begin{proof}
For this proof, we use the Conley--Zehnder grading of
\cite{GGlocal}. The homological grading used elsewhere in this
paper is obtained by adding $n$, which does not affect the
diameter of the support.

Write $\Delta=\Delta_H(\bar x)$. By the mean-index support
estimate~\cite[Section~3, (LF5)]{GGlocal},
\[
 \operatorname{supp}\mathrm{HF}^{\mathrm{loc}}_*(H,\bar x;\mathbb C)
 \subset [\Delta-n,\Delta+n].
\]
Suppose, for a contradiction, that both endpoints belong to the
support. Then (LF5) forces $\bar x$ to be totally degenerate.
Consequently,
\[
 \Delta=2r,\qquad r\in\mathbb Z,
\]
by~\cite[Section~2, (MI8)]{GGlocal}.

Working in a local model compatible with the chosen capping,
we may represent the orbit by a constant orbit $p$.
Choose a local Hamiltonian loop $G$ fixing $p$ whose linearization
has Maslov index $\mu=-r$. Such a loop exists: in Darboux
coordinates, rotations of one symplectic coordinate plane realize
every integer Maslov index. Set $K=G\#H$. Since the loop has
time-one map equal to the identity, $K$ and $H$ have the same
time-one germ at $p$, and $p$ remains isolated.

The mean index and the local Floer grading transform
simultaneously according to
\[
 \Delta_K(p)=\Delta_H(p)+2\mu,
 \qquad
 \mathrm{HF}^{\mathrm{loc}}_{j+2\mu}(K,p;\mathbb C)
 \cong
 \mathrm{HF}^{\mathrm{loc}}_j(H,p;\mathbb C);
\]
see~\cite[Section~2, (MI7), and Section~3, (LF6)]{GGlocal}.
Taking $\mu=-r$ therefore gives
\[
 \Delta_K(p)=0,
 \qquad
 \mathrm{HF}^{\mathrm{loc}}_n(K,p;\mathbb C)
 \cong
 \mathrm{HF}^{\mathrm{loc}}_{2r+n}(H,p;\mathbb C)
 \ne0.
\]
Thus $p$ is a symplectically degenerate maximum of $K$.
The assumed nonvanishing at the lower endpoint also gives
\[
 \mathrm{HF}^{\mathrm{loc}}_{-n}(K,p;\mathbb C)
 \cong
 \mathrm{HF}^{\mathrm{loc}}_{2r-n}(H,p;\mathbb C)
 \ne0.
\]

On the other hand, the local model and perturbation constructed
in~\cite[Proposition~5.2 and its proof]{GGlocal} compute the local
Floer homology of this symplectically degenerate maximum using
a local complex with a single nondegenerate generator of degree
$n$. Hence
\[
 \mathrm{HF}^{\mathrm{loc}}_j(K,p;\mathbb C)=0
 \qquad\text{for every }j\ne n.
\]
Since $n\geq1$, this contradicts the nonvanishing in degree $-n$.
The two endpoints therefore cannot both belong to the support,
and its diameter is strictly less than $2n$.

The normalization above changes only the local Hamiltonian
path; it does not require a change of capping or an extension
of the local loop to a global Hamiltonian loop. The argument
applies to capped isolated orbits on arbitrary symplectic
manifolds and with coefficients in $\mathbb C$.
\end{proof}

\subsection{Counting the orbits in a strict spectral chain}

\begin{proof}[Proof of Theorem~\ref{thm:negative-monotone}]
If $\operatorname{Fix}(\varphi)$ is infinite, there is nothing to
prove. Otherwise choose a Hamiltonian $H$ generating $\varphi$.
All its one-periodic orbits are isolated. We first prove
\eqref{eq:negative-degree-bound}.

Let $\delta_i=\deg\eta_i\geq d$, put $u_i=\operatorname{PD}(\eta_i)$,
and form the quantum products
\[
 a_0=[M],\qquad a_j=u_1*\cdots*u_j,\qquad 1\leq j\leq k.
\]
Every ordinary cup-product prefix is nonzero, since otherwise the
full product would vanish. It is the zero-area coefficient of the
corresponding quantum product. Thus $a_j\ne0$, and its total
homological degree is
\begin{equation}\label{eq:negative-chain-degrees}
 D_j=2n-\sum_{i=1}^j\delta_i,\qquad D_0=2n.
\end{equation}
Theorem~\ref{thm:LjS}, with $\zeta=0$, gives the strict chain
\[
 c_0:=\rho_0(a_0,H)>c_1:=\rho_0(a_1,H)>
 \cdots>c_k:=\rho_0(a_k,H).
\]
By Lemma~\ref{lem:negative-graded-carrier}, choose capped orbits
$\bar x_j$ such that
\begin{equation}\label{eq:negative-chain-carriers}
 \mathcal A_H(\bar x_j)=c_j,\qquad
 \mathrm{HF}^{\mathrm{loc}}_{D_j}(H,\bar x_j;\mathbb C)\ne0.
\end{equation}

Suppose that the same underlying orbit $x$ occurs at the indices
$i_1<\cdots<i_r$. Use $\bar x_{i_1}$ as its reference capping and
write $\bar x_{i_s}=\bar x_{i_1}\# A_s$. Since the action values
strictly decrease,~\eqref{eq:negative-recapping} gives
\[
 \omega(A_s)=m_sp,\qquad
 0=m_1<m_2<\cdots<m_r,\qquad m_s\in\mathbb Z.
\]
In particular, $m_r\geq r-1$, and negative monotonicity implies
$\langle c_1(TM),A_s\rangle=-m_sN$. Transferring
\eqref{eq:negative-chain-carriers} to the reference capping, we find
\[
 \mathrm{HF}^{\mathrm{loc}}_{D_{i_s}-2m_sN}
                   (H,\bar x_{i_1};\mathbb C)\ne0.
\]
The difference between the first and last of these degrees is
\[
 D_{i_1}-\bigl(D_{i_r}-2m_rN\bigr)
 =\sum_{t=i_1+1}^{i_r}\delta_t+2m_rN
 \geq(r-1)(d+2N).
\]
Lemma~\ref{lem:negative-local-width} therefore yields
\[
 (r-1)(d+2N)<2n,\qquad
 r\leq\left\lceil\frac{2n}{2N+d}\right\rceil.
\]
Each underlying orbit can thus account for at most the latter number
of the $k+1$ spectral carriers. Counting these carriers proves
\eqref{eq:negative-degree-bound}. The orbits counted are contractible,
so the same estimate holds for the total number of fixed points.

Choose a nonzero product realizing the cup-length $L$ and take
homogeneous components with nonzero product. With $k=L-1$ and $d=1$,
\eqref{eq:negative-degree-bound} gives the first term in
\eqref{eq:negative-monotone-bound}. On the other hand,
$[\omega]^n\ne0$, so we may take $\eta_i=[\omega]$, $k=n$, and $d=2$.
This gives the second term, since
$\lceil2n/(2N+2)\rceil=\lceil n/(N+1)\rceil$.
\end{proof}

\begin{cor}\label{cor:negative-borderline}
Under the hypotheses of Theorem~\ref{thm:negative-monotone}, if
$N\geq n-1$, then
\[
 \#\operatorname{Fix}(\varphi)\geq n+1.
\]
If also $H^1(M;\mathbb C)=0$, this is the cup-length bound, since
$\operatorname{cuplength}(M;\mathbb C)=n+1$.
\end{cor}

\begin{proof}
The second denominator in~\eqref{eq:negative-monotone-bound} is one
when $N\geq n-1$. If $H^1(M;\mathbb C)=0$, every positive-degree
cohomology class has degree at least two, so a nonzero product has
at most $n$ factors. The product $[\omega]^n\ne0$ attains this length.
\end{proof}

\begin{rmk}\label{rmk:negative-aspherical}
The assumption $N\geq1$ means that $c_1(TM)(\pi_2(M))$ is nonzero.
If this subgroup is zero, the negative proportionality relation
forces $[\omega]|_{\pi_2(M)}=0$ as well. This is the symplectically
aspherical case, for which the full estimate
$\#\operatorname{Fix}(\varphi)\geq\operatorname{cuplength}(M;\mathbb C)$
is classical~\cite{Fl3,Ho,RO}. Depending on convention, the minimal
Chern number of the zero subgroup is denoted by $\infty$ or by $0$;
it is not a finite positive $N$ to be substituted into
\eqref{eq:negative-monotone-bound}, and there is no positive
minimal symplectic period $p$ in this case.
\end{rmk}

\appendix

\section{Coadjoint orbits of compact simple Lie groups}\label{sec:coadjoint}

Let $K$ be a compact connected simple Lie group, $T\subset K$ a
maximal torus, and $\mathfrak R\subset\mathfrak t^*$ its root system.
Fix positive roots $\mathfrak R^+$ and simple roots $\Delta$, and use
an invariant inner product to define
$\alpha^\vee=2\alpha/(\alpha,\alpha)$. Passing to the simply
connected cover does not change a coadjoint orbit. We may therefore
use its complexification $G$ and identify the orbit through a
dominant $\lambda\in\mathfrak t^*$ with $X=G/P$, where
\[
 \Delta_P=\{\alpha\in\Delta:
                 \langle\lambda,\alpha^\vee\rangle=0\}.
\]
Write $\mathfrak R_P^+=\mathfrak R^+\cap\mathbb Z\Delta_P$ and
\[
 2\rho=\sum_{\gamma\in\mathfrak R^+}\gamma,
 \qquad
 2\rho_P=\sum_{\gamma\in\mathfrak R_P^+}\gamma.
\]
The complex dimension and degree lattice are
\begin{equation}\label{e:orbit-dimension-lattice}
 d=|\mathfrak R^+|-|\mathfrak R_P^+|,\qquad
 H_2(X;\mathbb Z)\cong Q^\vee/Q_P^\vee,
\end{equation}
where $Q^\vee=\bigoplus_{\alpha\in\Delta}\mathbb Z\alpha^\vee$
and $Q_P^\vee=\bigoplus_{\alpha\in\Delta_P}\mathbb Z\alpha^\vee$.
The classes of $\alpha^\vee$ for
$\alpha\in\Delta\setminus\Delta_P$ are the Schubert sphere classes
$C_\alpha$ and form a basis of the effective degree monoid.

Normalize the Kostant--Kirillov--Souriau form by
\[
 a_\alpha:=\omega_\lambda(C_\alpha)
          =\langle\lambda,\alpha^\vee\rangle>0,
 \qquad \alpha\in\Delta\setminus\Delta_P.
\]
For a nontrivial rational orbit, its minimal symplectic period is
the positive generator determined by
\begin{equation}\label{e:orbit-period}
 p\mathbb Z=
 \sum_{\alpha\in\Delta\setminus\Delta_P}a_\alpha\mathbb Z.
\end{equation}
For example, if $a_\alpha=t b_\alpha$ for $t>0$ and positive
integers $b_\alpha$, then $p=t\gcd_\alpha b_\alpha$. The first Chern
numbers of the basic classes are
\begin{equation}\label{e:orbit-chern}
 c_\alpha:=\langle c_1(TX),C_\alpha\rangle
 =\langle2\rho-2\rho_P,\alpha^\vee\rangle.
\end{equation}
Thus Theorem~\ref{thm:flags} gives the explicit estimate
\begin{equation}\label{e:orbit-basic-bound}
 \#\operatorname{Fix}(\varphi)\geq B_{\max}:=
 \left\lceil
 p\max_{\alpha\in\Delta\setminus\Delta_P}
 \frac{\langle2\rho-2\rho_P,\alpha^\vee\rangle}
      {\langle\lambda,\alpha^\vee\rangle}
 \right\rceil.
\end{equation}
In particular, for a monotone orbit this is the minimal Chern
number $N=\gcd_\alpha c_\alpha$.

To compare \eqref{e:orbit-basic-bound} with the original estimate from Conjecture~\ref{thm:mainthm}, let $\mathcal B_{\mathfrak R}$ be Kostant's cascade of
strongly orthogonal positive roots. It is constructed recursively
by selecting the highest root and repeating in each irreducible
component of the perpendicular root subsystem. For a reducible
root system take the union of the cascades of its components. Put
\begin{equation}\label{e:cascade-degree}
 D_X=\left[\sum_{\beta\in\mathcal B_{\mathfrak R}}\beta^\vee\right]
      \in Q^\vee/Q_P^\vee,
 \qquad
 E_\lambda=\sum_{\beta\in\mathcal B_{\mathfrak R}}
                       \langle\lambda,\beta^\vee\rangle.
\end{equation}
The brackets denote projection to the quotient lattice. A coroot
class that vanishes in that quotient contributes zero area.

The minimal-degree theorem for the quantum product of two point
classes and B\"arligea's cascade formula~\cite{Barligea} identify
$D_X$ as its unique minimal degree. More precisely, $D_X$ occurs
with nonzero Schubert coefficient in
$\operatorname{pt}*\operatorname{pt}$, and every other degree in
that product differs from $D_X$ by an effective degree. The
projection formula $D_{G/P}=(D_{G/B})_P$ holds for every parabolic
$P$.

\begin{thm}\label{thm:orbits}
For a nontrivial rational coadjoint orbit with the conventions
above,
\begin{equation}\label{e:orbit-theta}
 \Theta(X,\omega_\lambda)=E_\lambda.
\end{equation}
Every $\varphi\in\operatorname{Ham}(X,\omega_\lambda)$ satisfies
\begin{equation}\label{e:adjorbit}
\begin{aligned}
 \#\operatorname{Fix}(\varphi)
 &\geq B_{\max}\geq
 \left\lceil\frac{p\langle c_1(TX),D_X\rangle}{E_\lambda}\right\rceil
 \\
 &\geq
 \left\lceil
 \frac{p\bigl(|\mathfrak R^+|-|\mathfrak R_P^+|+1\bigr)}
 {E_\lambda}\right\rceil,
\end{aligned}
\end{equation}
where $B_{\max}$ is given by \eqref{e:orbit-basic-bound}.
\end{thm}

\begin{proof}
The degree-$D_X$ coefficient in the quantum product of two point
classes gives a nonzero three-point invariant with two point
insertions. Therefore
$\Theta(X,\omega_\lambda)\leq\omega_\lambda(D_X)=E_\lambda$.
Conversely, any nonzero invariant in the definition of $\Theta$
gives a genus-zero stable curve in its degree $A$ through two
general points. The Fulton--Woodward theorem~\cite{FWo} supplies
an effective degree $A'\leq A$ occurring in the small quantum
product of two point classes. Minimality gives
$D_X\leq A'\leq A$, where each difference is effective. Positivity
of a K\"ahler class implies
$\omega_\lambda(A)\geq E_\lambda$ and proves
\eqref{e:orbit-theta}.

The first inequality in \eqref{e:adjorbit} is
\eqref{e:orbit-basic-bound}. Writing $D_X$ as a nonnegative
combination of the $C_\alpha$ shows that
$\langle c_1(TX),D_X\rangle/E_\lambda$ is an area-weighted average
of $c_\alpha/a_\alpha$, which proves the next inequality. Finally,
the dimension constraint for the nonzero three-point invariant in
degree $D_X$ and the fundamental-class axiom imply
$\langle c_1(TX),D_X\rangle\geq d+1$. This proves the last
inequality.
\end{proof}

The classical types have the following coordinate formulas for
$\Theta$. Here the $e_i$ are orthonormal. Type $A_{n-1}$ uses roots
$e_i-e_j$ and $\lambda_1\geq\cdots\geq\lambda_n$; adding a common
constant to the $\lambda_i$ has no effect. Types $B_\ell$ and
$C_\ell$ use, respectively,
\[
 \{\pm e_i,\ \pm e_i\pm e_j\},\qquad
 \{\pm2e_i,\ \pm e_i\pm e_j\},
\]
with $\lambda_1\geq\cdots\geq\lambda_\ell\geq0$. Type $D_\ell$
uses roots $\pm e_i\pm e_j$ and
$\lambda_1\geq\cdots\geq\lambda_{\ell-1}\geq|\lambda_\ell|$.
All areas use coroots: $(e_i)^\vee=2e_i$ in type $B$, whereas
$(2e_i)^\vee=e_i$ in type $C$.
\[
\begin{array}{c|c|c}
 \text{type}&|\mathfrak R^+|&\Theta(X,\omega_\lambda)\\ \hline
 A_{n-1}&n(n-1)/2&
 \displaystyle\sum_{j=1}^{\lfloor n/2\rfloor}
                   (\lambda_j-\lambda_{n+1-j})\\[6pt]
 B_\ell&\ell^2&
 \displaystyle2\sum_{j=1}^{\lceil\ell/2\rceil}\lambda_{2j-1}\\[6pt]
 C_\ell&\ell^2&\displaystyle\sum_{j=1}^\ell\lambda_j\\[6pt]
 D_\ell&\ell(\ell-1)&
 \displaystyle2\sum_{j=1}^{\lfloor\ell/2\rfloor}\lambda_{2j-1}
\end{array}
\]
These formulas remain valid on chamber walls, where the dimension
is $|\mathfrak R^+|-|\mathfrak R_P^+|$ and the Chern numbers must
be computed using $2\rho-2\rho_P$. The root formulas
\eqref{e:orbit-period} and \eqref{e:orbit-chern} give the stronger
bound \eqref{e:orbit-basic-bound} directly, without first computing
the cascade. For the exceptional types $E_6,E_7,E_8,F_4,G_2$, the
same formulas apply; the numbers of positive roots are
$36,63,120,24,6$, respectively. The recursive cascade additionally
computes $\Theta$ for comparison.

For example, let $\alpha_1$ be the short simple root and
$\alpha_2$ the long simple root of $G_2$. Its cascade is
$\{3\alpha_1+2\alpha_2,\alpha_1\}$. Since
\[
 (3\alpha_1+2\alpha_2)^\vee
     =\alpha_1^\vee+2\alpha_2^\vee,
\]
the simple-sphere areas
$a_i=\langle\lambda,\alpha_i^\vee\rangle$ give
$\Theta=2a_1+2a_2$. For a regular rational orbit, $a_1,a_2>0$,
$d=6$, $c_1=c_2=2$, and
$p\mathbb Z=a_1\mathbb Z+a_2\mathbb Z$. Thus
\[
 \#\operatorname{Fix}(\varphi)\geq
 \left\lceil\frac{2p}{\min\{a_1,a_2\}}\right\rceil.
\]
The cascade-class and original cup-length estimates are,
respectively,
\[
 \left\lceil\frac{4p}{a_1+a_2}\right\rceil
 \quad\text{and}\quad
 \left\lceil\frac{7p}{2(a_1+a_2)}\right\rceil,
\]
and are bounded above by the displayed basic-sphere estimate.
On either nonzero wall $d=5$. If $a_1>0$ and $a_2=0$, then
$\Delta_P=\{\alpha_2\}$ and
\[
 c_1(TX)[C_{\alpha_1}]
 =\langle2\rho-\alpha_2,\alpha_1^\vee\rangle=5;
\]
since $p=a_1$, the fixed-point bound is $5$. If $a_1=0$ and
$a_2>0$, then $\Delta_P=\{\alpha_1\}$ and
\[
 c_1(TX)[C_{\alpha_2}]
 =\langle2\rho-\alpha_1,\alpha_2^\vee\rangle=3,
\]
so the bound is $3$. These are again lower bounds, without a claim
that they are attained.

\section{Symplectic toric manifolds}\label{Sec:toric}

We use an irredundant Delzant presentation, with one coordinate for each facet. We recall a definition and some facts on symplectic toric manifolds, and refer the reader to the books~\cite{Au,MS} for more details on this subject.

Let $\C^n$ be the standard $2n$-dimensional real vector space equipped with the standard symplectic form $\omega_0$, i.e., the imaginary part of $-\frac{1}{2}\sum dz_i\wedge d\bar{z}_i$. Let $\T^k$ be a $k$-dimensional torus with Lie algebra $\mathfrak{t}$. Denote the integer lattice of $\mathfrak{t}$ by
\[
\Delta:=\{\xi\in\fT\;|\;\exp (\xi)=1\}
\]
and its dual lattice by
\[
\Delta^*:=\{w\in \mathfrak{t}^*|\langle w,\xi\rangle\in\Z\;\;\hbox{for any}\;\xi\in\Delta\}.
\]
Suppose that $\T^k$ acts diagonally on $\C^n$ by a homomorphism $\rho=(\rho_1\ldots,\rho_n)$ with  actions $\rho_i:\T^k\times\C\to\C$ given by
\[
(\exp(\xi),z_i)\longmapsto \exp(-2\pi\sqrt{-1}\langle \fw_i,\xi \rangle)\cdot z_i,\quad i=1,\ldots, n
\]
where $\fw_i\in \Delta^*$. The moment map  of this action is
\[
\mu(z)=\pi\sum_{i=1}^n|z_i|^2\fw_i
\]
for any $z\in\C^n$. Here we require that $\mu$ is proper and the vectors $\fw_i$ span the vector space $\fT^*$, and hence the linear space of solutions $(\alpha_1,\ldots,\alpha_n)\in\R^n$ for the linear equation $\sum_{i=1}^n\alpha_i\fw_i=0$ in $\fT^*$ has dimension $n-k$.

\begin{df}\label{df:toric}
A toric symplectic manifold is a symplectic quotient
at a regular value $\tau\in \fT^*$ of the moment map $\mu$ by $\T^k$ such that $\T^k$ acts freely on $\mu^{-1}(\tau)$, namely
\[
M_\tau=\mu^{-1}(\tau)/\T^k
\]
which has a real dimension $2m=2n-2k$ and a symplectic form $\omega_\tau\in \Omega^2(M_\tau)$ induced by the symplectic reduction of $(\C^n,\omega_0)$.
\end{df}
We denote by $H_*(M_\tau)$ (resp. $H^*(M_\tau)$) the quotient of the integral (co)homology group by its torsion subgroup.
For $j\in I:=\{1,\ldots,n\}$, we denote by
\[
{\rm cone}(I_j)=\big\{\sum_{i\neq j}\alpha_i\fw_i\;\big|\;\alpha_i\geq 0\big\}
\]
the positive cone spanned by the vectors $\fw_i,\; i\in I\setminus\{j\}$. Consider now the subspace
$$E_j=\{z=(z_1,\ldots,z_n)\in\C^n\;|\;z_j=0\}.$$
Obviously, $\tau\in {\rm cone}(I_j)$ if and only if $E_j\cap\mu^{-1}(\tau)\neq 0$. In this case, one can see that the corresponding submanifold
\[
X_j=E_j\cap\mu^{-1}(\tau)/\T^k=\{[z]\in M_\tau\;|\;z_j=0\}.
\]
is a smooth divisor of $M_\tau$ and determines a codimension-$2$ class $[X_j]$ in $H_{2m-2}(M_\tau)$.  It was shown in \cite{GuS} that $[X_j]$ ($1\leq j\leq n$) generate the homology ring of $H_*(M_\tau)$ w.r.t. the usual homology intersection product $\cap$ in the following sense:
\begin{itemize}
\item the homology classes $[X_1],\ldots, [X_n]$ span $H_{2m-2}(M_\tau)$,
\item every relation $\sum_{i=1}^n\alpha_i\fw_i=0$ implies a corresponding relation $\sum_{i=1}^n\alpha_i[X_i]=0$,
\item for any subset $J=\{i_1,\ldots, i_j\}\subset I$, the relation $\tau\notin \{\sum_{i\in I\setminus J}\alpha_i\fw_i\;\big|\;\alpha_i\geq 0\}$ implies a corresponding relation $[X_{i_1}]\cap\cdots\cap[X_{i_j}]=0$,
\item the above relations are the only ones for all classes $[X_i]$ in $H_*(M_\tau)$.
\end{itemize}
Note that $H_2(M_\tau)$ can be identified with the sublattice of $\Delta\subset\fT$:
\[\Delta(\tau)=\{\xi\in\Delta\;|\;\langle \fw_j,\xi\rangle=0\quad \hbox{for any}\;j\;\hbox{with}\;\tau\notin {\rm cone}(I_j) \}.\]
We denote the bijective map by
\begin{equation}\label{e:class}
\Delta(\tau)\longrightarrow H_2(M_\tau),\quad \eta\longmapsto A_\eta.
\end{equation}
Then, the dual $H^2(M_\tau)$ corresponds to a quotient of the integer lattice in $\fT^*$. Furthermore,  the first Chern class $c_1(TM)$ and the cohomology class of $\omega_\tau$ are given by
\[c_1(TM)=\overline{\fw}_1+\ldots+\overline{\fw}_n,\quad [\omega_\tau]=\overline{\tau}\]
where $\overline{\gamma}$ is the image of $\gamma\in\fT^*$ under the correspondence
\[
\fT^*=\Delta^*\otimes\R\longrightarrow H^2(M_\tau;\R), \quad \gamma\longmapsto \overline{\gamma}.
\]
induced by (\ref{e:class}).
Let $C(\tau)\subset\fT^*$ be the component of $\tau$ in the set of regular values of the moment map $\mu$ which is called the \textit{K\"{a}hler cone} of $M_\tau$. Its dual cone in $\Delta(\tau)$ is given by
\begin{equation}\label{e:effcone}
\Delta^{\rm eff}(\tau)=\{\xi\in \Delta(\tau)\;|\;\langle t,\xi\rangle\geq 0\quad\hbox{for any}\;t\in C(\tau)\}.
\end{equation}
The real cone underlying $\Delta^{\rm eff}(\tau)$ is the Mori cone, generated by the classes of torus-invariant rational curves. The definition in~(\ref{e:effcone}) records its integral lattice points; no assertion about saturation of a curve-class semigroup is needed here.
\begin{df}\label{df:fano}
We say that a toric manifold $M_\tau$ is \textit{Fano} if the vector $c_1=\sum_{i=1}^n\fw_i$ satisfies $\langle c_1,\xi\rangle>0$ for all $\xi\in \Delta^{\rm eff}(\tau)\setminus\{0\}$.
\end{df}
Let $\mathcal{D}^{\rm eff}(\tau)\subset\Z^n$ denote the cone
\begin{equation}\label{e:cone}
\mathcal{D}^{\rm eff}(\tau)=\{(\langle\fw_1,\xi\rangle,\ldots,\langle\fw_n,\xi\rangle)\;|\;\xi\in \Delta^{\rm eff}(\tau)\}.
\end{equation}
Then we have a bijection $\Delta^{\rm eff}(\tau)\to \mathcal{D}^{\rm eff}(\tau)$, and denote its inverse map by
\begin{equation}\label{e:bi}
\mathcal{D}^{\rm eff}(\tau)\longrightarrow\Delta^{\rm eff}(\tau),\quad d\longmapsto \xi_d.
\end{equation}
We emphasize that $\mathcal{D}^{\rm eff}(\tau)$ is not necessarily contained in the positive quadrant of $\Z^n$.

For $d\in\mathcal D^{\rm eff}(\tau)$, let $\xi_d\in\Delta(\tau)$
be its preimage under~(\ref{e:bi}), and let $A_{\xi_d}$ be the class
in~(\ref{e:class}).  When all entries of a tuple $d$ are nonnegative,
$[X]^{*d}$ denotes the $\sum_i d_i$-fold quantum product
in~(\ref{e:Xd}), with the empty product equal to $[M_\tau]$.
For a general tuple we use only the nonnegative parts $d^+$ and $d^-$
defined below.
The Fano hypothesis allows one to use the ordinary divisor classes in
Batyrev's quantum relations.  More precisely, Gonz\'alez and Iritani
\cite[Section~1, equation~(2)]{GI} recall the Seidel-element relation
for every integral class $A\in H_2(M_\tau;\Z)$ and the equality of
these elements with the ordinary toric divisor classes in the Fano
case.  Specializing the Novikov monomial $q^A$ to $T^{\omega(A)}$
and applying Poincar\'e duality gives the following statement; see
also~\cite[Chapter~11]{MS}.
\begin{thm}[Fano toric quantum relations]\label{thm:fano}
If  $M_\tau$ is a Fano toric manifold, then the equality
\[
[X]^{*d^+}=T^{\langle[\omega],A_{\xi_d}\rangle}[X]^{*d^-}
\]
holds in the small quantum homology $H_*(M_\tau;\Lambda)$ for every $d\in\mathcal{D}^{\rm eff}(\tau)$, where the $n$-tuples $d^+$ and $d^-$ are given by
$$d_i^+:=\max\{d_i,0\},\quad d_i^-:=\max\{-d_i,0\}.$$
\end{thm}

The relation is therefore available for all the lattice degrees used
above, not just for primitive relations.  It concerns the small
quantum product.  Neither a corresponding relation for arbitrary
bulk deformation nor the minimum-area condition~\eqref{e:infarea}
is asserted here.

\section*{Acknowledgements}
I thank Jinxin Xue for his help and encouragement and for support through NSFC grant 12271285 and the New Cornerstone Investigator Program. I am indebted to Egor Shelukhin  for pointing out an error in the proof of Conjecture~\ref{thm:mainthm} in the first version of this manuscript. I thank  Guanheng Chen, Vincent Humili\`{e}re, Tian-Jun Li, Guangcun Lu,  Zhiyu Tian, Jun Zhang and Shuo Zhang for valuable and inspiring remarks.  I am grateful to  Kaoru Ono and Michael Usher for helpful discussions about spectral invariants with bulk studied in~K. Fukaya,  Y.-G. Oh,  H. Ohta and K. Ono~\cite{FOOO} and M. Usher~\cite{Us}, respectively. Finally, I thank Huagui Duan, Hui Liu, Jian Wang, Guowei Yu, Duanzhi Zhang, Qinglong Zhou and Chaofeng Zhu, the organizers of the International Conference on Advanced Topics in Dynamical Systems in Tianjin (October 2025), Kaoru Ono and Yong-Geun Oh, the organizers of RIMS \& IBS-CGP joint workshop in Kyoto (November 2025), Jinxin Xue and Jun Zhang, the organizers of International Conference on Symplectic Geometry in Beijing (July 2026), for giving me an opportunity to present a preliminary version of this work and for the superb job they did in organizing these conferences.

\section*{AI disclosure}
This manuscript revises earlier versions~\cite{Go1}. The mathematical ideas and proofs were generated and written by human. AI assistants were used to assist with reference checks, notation, and the critical examination of parts of the mathematical arguments.

\end{document}